\documentclass[12pt, reqno]{amsart}
\usepackage[ruled,vlined]{algorithm2e}
\usepackage{fancyhdr,graphicx,amsmath,amssymb,amsthm}
\usepackage[left=1.0in,right=1.0in,bottom=1.0in]{geometry}
\usepackage{tcolorbox, enumitem, verbatim}
\usepackage{yfonts}
\usepackage{xcolor}
\usepackage{hyperref}
\hypersetup{
    colorlinks=true,   
    linkcolor=blue,    
    citecolor=red,     
    urlcolor=magenta,  
    filecolor=cyan     
}
\usepackage{tikz}
\usepackage[utf8]{inputenc}
\usepackage{mathtools}
\usepackage{faktor}
\usepackage{quiver}
\usepackage{tikz-cd}
\usepackage{bbm}
\usepackage{faktor}
\usepackage{mathabx}
\usepackage{thmtools}
\usepackage{xcolor}
\usepackage{mathrsfs}
\usepackage[nameinlink,capitalise,noabbrev]{cleveref}
\usepackage[backend=biber, style=alphabetic]{biblatex}
\usepackage{bm}

\theoremstyle{plain}
\newtheorem{theorem}{Theorem}[section]
\newtheorem{coroll}[theorem]{Corollary}
\newtheorem{lemma}[theorem]{Lemma}
\newtheorem{prop}[theorem]{Proposition}

\theoremstyle{definition}

\newtheorem{definition}[theorem]{Definition}
\newtheorem{defprop}[theorem]{Definition/Proposition}
\newtheorem*{definition*}{Definition}
\newtheorem{example}[theorem]{Example}

\theoremstyle{remark}
\newtheorem*{remark}{Remark}    

\newcommand{\CC}{\mathbb{C}}
\newcommand{\AAA}{\mathbb{A}}
\newcommand{\ZZ}{\mathbb{Z}}

\newcommand{\PP}{\mathbb{P}}

\newcommand{\GG}{\mathbb{G}}

\newcommand{\HH}{\mathbb{H}}

\newcommand{\DR}{\mathrm{DR}}
\newcommand{\Dol}{\mathrm{Dol}}
\newcommand{\tw}{\mathrm{tw}}
\newcommand{\Dh}{\mathrm{DH}}
\newcommand{\nilp}{\mathrm{nilp}}
\newcommand{\unip}{\mathrm{unip}}
\newcommand{\pair}{\mathrm{pair}}
\newcommand{\trip}{\mathrm{trip}}
\newcommand{\univ}{\mathrm{univ}}
\newcommand{\Hod}{\mathrm{Hod}}

\newcommand{\rank}{\operatorname{rank}}
\newcommand{\im}{\operatorname{im}}

\newcommand{\BAR}{\overline}

\newcommand{\sh}[1]{\mathcal{#1}}
\newcommand{\scr}[1]{\mathscr{#1}}
\newcommand{\Rm}[1]{\mathrm{#1}}
\newcommand{\Sf}[1]{\mathsf{#1}}
\newcommand{\op}[1]{\operatorname{#1}}
\newcommand{\Ide}[1]{\mathfrak{#1}}
\newcommand{\der}{\mathrm{d}}

\newcommand{\ul}[1]{\underline{#1}}
\newcommand{\del}{\partial}
\newcommand{\delbar}{\BAR{\partial}}
\newcommand{\shHom}{\mathcal{H}\mathrm{om}}

\newcommand{\Spec}{\operatorname{Spec}}

\newcommand{\coker}{\operatorname{coker}}

\newcommand{\Sum}[1]{\underset{#1}{\sum}}

\newcommand{\tab}{\hspace*{\parindent}}
\DeclareMathOperator{\Hom}{Hom}%

\newcommand{\genlegendre}[4]{%
  \genfrac{(}{)}{}{#1}{#3}{#4}%
  \if\relax\detokenize{#2}\relax\else_{\!#2}\fi
}

\DeclarePairedDelimiter{\set}{\{}{\}}
\DeclarePairedDelimiter{\abs}{\lvert}{\rvert}

\DeclarePairedDelimiter{\brac}{(}{)}
\DeclarePairedDelimiter{\sqbrac}{[}{]}

\DeclareMathAlphabet{\pazocal}{OMS}{zplm}{m}{n}

\newcommand{\Comment}[1]{}

\def\XXint#1#2#3{{\setbox0=\hbox{$#1{#2#3}{\int}$}
     \vcenter{\hbox{$#2#3$}}\kern-.5\wd0}}

\allowdisplaybreaks

\makeatletter
\def\l@subsection{\@tocline{2}{0pt}{2.5pc}{5pc}{}}
\makeatother

\DeclareFieldFormat[article,inbook,incollection,inproceedings]{title}{#1}

\title[NAHC for extensions on quasi-projective curves]{The non-abelian Hodge correspondence for extensions on quasi-projective curves}
\author{Quoc Anh Tran}
\address{Department of Mathematics, University of Chicago, Chicago, Illinois, USA}
\email{quocanh@uchicago.edu}
\urladdr{https://quocanh-tran.github.io/}
\date{\today}

\begin{document}
\begin{abstract}
  Let $\bar{X}$ be a smooth projective connected complex curve, $D \subset \bar X$ be a finite set of reduced points, and $X=\bar X\setminus D$.  We give a moduli-theoretic construction of an exact equivalence between logarithmic flat bundles on $(\bar X, D)$ with nilpotent residues and semistable logarithmic Higgs bundles of degree zero with nilpotent residues. This extends the tame Simpson-Mochizuki correspondence on polystable bundles. 
\end{abstract}
\maketitle
\vspace{-0.5cm}
\tableofcontents
\newpage

\section{Introduction}
\label{sec:intro}
On a smooth projective variety, by equipping the underlying $C^\infty$-vector bundles with a harmonic metric, Simpson gave a correspondence between semisimple flat connections and polystable Higgs bundles with vanishing Chern classes. He also extended this to an equivalence of categories between (not necessarily semisimple) flat connections and semistable Higgs bundles with vanishing Chern classes. This full correspondence was crucial in Cao's \cite{cao2019albanese} proof that the Albanese map is locally trivial for projective manifolds with nef anticanonical bundle. \\
\tab Now let $\bar{X}$ be a smooth projective curve, $D \subset \bar{X}$ a reduced divisor, and put $X = \bar{X}\setminus D$. Write $\Sf{Conn}_{\nilp}(\bar{X}, D)$ for the category of logarithmic flat bundles on $(\bar{X}, D)$ with nilpotent residues, and $\Sf{Higgs}_{\nilp}^{0, \Rm{ss}}(\bar{X}, D)$ the category of semistable degree 0 logarithmic Higgs bundles with nilpotent residues. The tame harmonic theory gave an equivalence of the polystable bundles; in the curve case this was due to Simpson \cite{simpson1990harmonic}, and in general a theorem of Mochizuki \cite{mochizuki2009kobayashi}. Our main result is an equivalence for iterated extensions, i.e., non-polystable bundles:
\begin{theorem}
  There exists a $\CC$-linear exact equivalence of abelian categories 
  \[\Sf{SM}: \Sf{Conn}_{\nilp}(\bar{X}, D) \to \Sf{Higgs}_{\nilp}^{0, \Rm{ss}}(\bar{X}, D)\]
  extending the Simpson-Mochizuki equivalence $\Sf{SM}^{\Rm{ps}}$ on polystable bundles. 
\end{theorem}
\subsection{Simpson's approach in the projective case}
The bridge from the polystable correspondence to the one above requires control of extensions and compositions of morphisms, or equivalently cup products. In the projective case, i.e. $D = \emptyset$, Simpson \cite[Lemmas 2.2 and 3.3-3.5]{simpson1992higgs} gave these categories the structure of differential graded categories, where morphisms and extensions are computed using the following resolutions of the De Rham and Dolbeault complexes \eqref{eq:Dol/DR-complexes} (see also the notation of \cref{def:harmonic-decomp1,def:harmonic-decomp2})
\begin{align*}
  \Rm{R}\Gamma(X, C_{\DR}(E, F)) &\simeq \sqbrac*{A^0(\ul{\Hom}(E, F)) \xrightarrow{D} A^1(\ul{\Hom}(E, F)) \xrightarrow{D} \dots} \\
  \Rm{R}\Gamma(X, C_{\Dol}(E, F)) &\simeq \sqbrac*{A^0(\ul{\Hom}(E, F)) \xrightarrow{\bar{\del}_h + \theta} A^1(\ul{\Hom}(E, F)) \xrightarrow{\bar{\del}_h + \theta} \dots}
\end{align*}
by $\ul{\Hom}(E, F)$-valued $C^{\infty}$-differential forms. He then used the analytic inputs of $\del \delbar$-lemma and K\"{a}hler identities to prove a formality result giving an explicit cup-product-compatible quasi-isomorphism of the two complexes of differential forms. This then allows for a correspondence of all extensions by the theory of differential graded categories. It is worth noting that an abstract quasi-isomorphism
\[\Rm{R}\Gamma(X, C_{\DR}(E, F)) \simeq \Rm{R}\Gamma(X, C_{\Dol}(E, F))\]
is not sufficient.

\medskip
In the quasi-projective case, we have similar resolutions by $C^{\infty}$ logarithmic forms. However, Simpson's K\"{a}hler formality argument does not give the desired logarithmic comparison here. In the context of this paper, one would need a much stronger refinement of \cref{thm:relative-DR-Dol-quasi-isom-with-coeff} which is compatible with cup products on the resolutions by logarithmic smooth forms. We sidestep this issue, only requiring compatibility with identity and composition on $H^0$. Instead, our approach is moduli-theoretic, treating a non-polystable bundle as a deformation of its semisimplification with respect to the canonical socle filtration (see \cref{sec:socle-filt-Rees}). 
\subsection{A sketch of the proof}
More concretely, given $\bar{E} \in \Sf{Conn}_{\nilp}(\bar{X}, D)$, pick a frame to get a point in the framed moduli space $R_{\DR}$ (see \cref{sec:framed-DR-Dol-Hod-spaces}). Consider the socle filtration $S_\bullet \bar{E}$, the Rees construction gives a family over $\AAA^1$ with the fiber over $0$ being the semisimplification $\bar{E}^{\Rm{ss}} = \Rm{gr}_S \bar{E}$. Equivalently, we have a $\GG_m$-equivariant arc $\alpha_{\bar{E}}^{\DR}: \AAA^1 \to R_{\DR}$. Completing at 0 we get a formal arc $\widehat{\alpha}_{\bar{E}}^{\DR}: \Rm{Spf}\ \CC[[t]] \to \widehat{R}_{\DR, \bar{E}^{\Rm{ss}}}$. As a consequence of \cite[Lemma 8.11]{bakker2024linear} we get a formal isomorphism 
\[\tau_{E^{\Rm{ss}}}: \widehat{R}_{\DR, \bar{E}^{\Rm{ss}}} \xrightarrow{\simeq} \widehat{R}_{\Dol, \Sf{SM}^{\Rm{ps}}(\bar{E}^{\Rm{ss}})}\]
which we prove is $\GG_m$-equivariant in \cref{thm:Bett-DR-Dol-formal-isom}. Hence we end up with a $\GG_m$-equivariant formal arc $\widehat{\alpha}_{\bar{E}}^{\Dol}: \Rm{Spf}\ \CC[[t]] \to \widehat{R}_{\Dol, \Sf{SM}^{\Rm{ps}}(\bar{E}^{\Rm{ss}})}$, which can be algebraized by \cref{lem:Gm-equiv-formal-arc-algebraization} to an arc $\alpha_{\bar{E}}^{\Dol}: \AAA^1 \to R_{\Dol, \Sf{SM}^{\Rm{ps}}(\bar{E}^{\Rm{ss}})}$. Forgetting the framing, let
\[\Sf{SM}(\bar{E}) = \alpha_{\bar{E}}^{\Dol}(1)\]
which is well-defined by \cref{lem:independence-framing-Rees-cochar}. To define $\Sf{SM}$ for a morphism $f: \bar{E} \to \bar{F}$, the natural idea is to construct a moduli space parametrizing framed triples $(\bar{E}, \bar{F}, f)$ called $\widehat{R}_{\DR}^{\trip}$ (see \cref{sec:moduli-triples}) and run the previous argument for triples. The proof of \cite[Lemma 8.11]{bakker2024linear} cannot be repeated here, as they relied on $R_{B}$ being smooth when $X = \bar{X}\setminus D$ is an affine curve; the moduli of triples $R^{\trip}_{B}$ is not. Instead, we use the theory of twistor modules with artinian coefficients to identify the deformation functors of triples. More concretely, we have a pro-twistor algebra $\bm{\Lambda}_{E^{\Rm{ss}}, F^{\Rm{ss}}}$ where 
\[\Rm{Spf}_{\PP^1}\ \bm{\Lambda}_{E^{\Rm{ss}}, F^{\Rm{ss}}}\vert_{\lambda = 1} \simeq \widehat{R}_{\DR, (\bar{E}^{\Rm{ss}}, \bar{F}^{\Rm{ss}})}, \tab \Rm{Spf}_{\PP^1}\ \bm{\Lambda}_{E^{\Rm{ss}}, F^{\Rm{ss}}}\vert_{\lambda = 0} \simeq \widehat{R}_{\Dol, \Sf{SM}^{\Rm{ps}}(\bar{E}^{\Rm{ss}}, \bar{F}^{\Rm{ss}})}\]
\tab Any such twistor algebra is trivializable away from $\lambda = \infty$ (\cref{lem:affinesplitting}), and the formal isomorphism of pairs
\[\tau_{E^{\Rm{ss}}, F^{\Rm{ss}}}: \widehat{R}_{\DR, (\bar{E}^{\Rm{ss}}, \bar{F}^{\Rm{ss}})} \xrightarrow{\simeq} \widehat{R}_{\Dol, \Sf{SM}^{\Rm{ps}}(\bar{E}^{\Rm{ss}}, \bar{F}^{\Rm{ss}})}\]
is given by this trivialization. Furthermore, there is a universal $\bm{\Lambda}_{E^{\Rm{ss}}, F^{\Rm{ss}}}$-admissible variation of mixed twistor structures $(\widehat{\scr{E}}, \widehat{\scr{F}})$ where the fibers over 1, 0 give universal logarithmic flat pair $(\widehat{\scr{E}}\vert_{\lambda = 1}, \widehat{\scr{F}}\vert_{\lambda = 1})$ and universal logarithmic Higgs pair $(\widehat{\scr{E}}\vert_{\lambda = 0}, \widehat{\scr{F}}\vert_{\lambda = 0})$ over $\widehat{R}_{\DR, (\bar{E}^{\Rm{ss}}, \bar{F}^{\Rm{ss}})}$ and $\widehat{R}_{\Dol, \Sf{SM}^{\Rm{ps}}(\bar{E}^{\Rm{ss}}, \bar{F}^{\Rm{ss}})}$ respectively. \cref{thm:relative-DR-Dol-quasi-isom-with-coeff} then gives a quasi-isomorphism 
\[\Phi: K_{\DR}(\widehat{\scr{E}}\vert_{\lambda = 1}, \widehat{\scr{F}}\vert_{\lambda = 1}) \xrightarrow{\simeq} K_{\Dol}(\widehat{\scr{E}}\vert_{\lambda = 0}, \widehat{\scr{F}}\vert_{\lambda = 0})\]
whose $H^0$ identifies the flat morphisms with the Higgs morphisms. For arbitrary artinian local $A$, we get the corresponding quasi-isomorphism for $A$-deformations by universality and pulling back. Thus, we have a $\GG_m$-equivariant formal isomorphism (\cref{thm:formal-isom-triples-and-composition})
\[\Psi_{E^{\Rm{ss}}, F^{\Rm{ss}}, f^{\Rm{ss}}}: \widehat{R}^{\trip}_{\DR, (\bar{E}^{\Rm{ss}}, \bar{F}^{\Rm{ss}}, f^{\Rm{ss}})} \xrightarrow{\simeq} \widehat{R}^{\trip}_{\Dol, \Sf{SM}^{\Rm{ps}}(\bar{E}^{\Rm{ss}}, \bar{F}^{\Rm{ss}}, f^{\Rm{ss}})}\]
and a similar algebraization argument allows us to define $\Sf{SM}(f)$. This $\Psi_{E^{\Rm{ss}}, F^{\Rm{ss}}, f^{\Rm{ss}}}$ is compatible with composition on universal families, hence $\Sf{SM}$ as defined is actually a functor, see \cref{thm:functor-verification}. On the other hand, an equivariant algebraization result for complexes on $\AAA^1$ gives a quasi-isomorphism \eqref{eq:quasi-isom-pulled-back-A1}
\[\Phi_{\bar{E}, \bar{F}}^{\Sf{SM}}: \Rm{R}\Gamma(\bar{X}, C_{\DR}(\bar{E}, \bar{F})) \xrightarrow{\simeq} \Rm{R}\Gamma(\bar{X}, C_{\Dol}(\Sf{SM}(\bar{E}), \Sf{SM}(\bar{F})))\]
by pulling back the universal logarithmic DR/Dol complexes (\cref{sec:moduli-triples}) along the equivariant arcs. Taking $H^0$ gives 
\[H^0(\Phi_{\bar{E}, \bar{F}}^{\Sf{SM}}): \Hom_{\Sf{Conn}_{\Rm{nilp}}(\bar{X}, D)}(\bar{E}, \bar{F}) \xrightarrow{\simeq} \Hom_{\Sf{Higgs}_{\nilp}^{0, \Rm{ss}}(\bar{X}, D)}(\Sf{SM}(\bar{E}), \Sf{SM}(\bar{F}))\]
and by \cref{lem:morphism-formal-transport-and-equiv-algebraization}, $\Sf{SM}(f) = H^0(\Phi_{\bar{E}, \bar{F}}^{\Sf{SM}})(f)$. Consequently, $\Sf{SM}$ is fully faithful. Then a categorical argument, \cref{lem:exactness-and-socle-preservation} implies that $\Sf{SM}$ is exact, preserves the socle filtration, and the filtration on $\Sf{SM}(\bar{E}) = \alpha_{\bar{E}}^{\Dol}(1)$ associated to the algebraized Rees arc $\alpha_{\bar{E}}^{\Dol}$ is actually the socle filtration on $\Sf{SM}(\bar{E})$. Then, by symmetry, we can do all the constructions in the reverse direction to get a quasi-inverse functor to $\Sf{SM}$. 
\begin{remark}
  The curve assumption is mainly used for the comparison \cref{thm:relative-DR-Dol-quasi-isom-with-coeff}. If this can be generalized to higher dimensions, perhaps using a multi-V-filtration, one should get the full correspondence in all dimensions. 
\end{remark}
\subsection{Acknowledgement}
The author would like to thank his advisor, Benjamin Bakker, for suggesting this project, numerous conversations and suggestions, as well as his enduring support. The author is also grateful to Izzet Coskun for conversations around moduli of vector bundles, and Yohan Brunebarbe for his interest and comments toward a previous version. 

\section{Preliminaries on twistor theory}
\label{sec:background}
We largely follow \cite{bakker2024linear} for the conventions. For the following definitions, our $(\bar{X}, D)$ is a pair of a projective smooth complex connected curve $\bar{X}$ with a set of reduced points $D$. 

\subsection{Logarithmic flat connections and Higgs bundles} Let $\Omega_{\bar{X}}^1(\log D)$ be the line bundle of logarithmic one-forms. 
\begin{definition}
  A logarithmic connection $(E, \nabla)$ is a vector bundle $E$ on $\bar X$ with a $\CC$-linear map $\nabla: E \to E \otimes \Omega_{\bar X}^1(\log D)$ satisfying the product rule $\nabla(fs) = \der f \otimes s + f\nabla s$ for local sections $f \in \sh{O}_{\bar X}$ and $s \in E$. Locally, using a holomorphic frame around $p \in \bar{X}$, one can write 
  \[\nabla = \der + A(z)\frac{\der z}{z}\]
  and the residue is defined to be 
  \[\Rm{Res}_p\nabla \coloneq A(0) \in \Rm{End}(E_p)\]
  which is independent, up to conjugacy, of the choice of coordinate and frame. 
\end{definition}

\begin{definition}
  A logarithmic Higgs bundle $(E, \theta)$ is a vector bundle $E$ on $\bar{X}$ with a $\sh{O}_{\bar X}$-linear map $\theta: E \to E \otimes \Omega_{\bar X}^1(\log D)$. Locally, $\theta$ is a matrix of logarithmic one-forms, and its residue $\Rm{Res}_p\theta$ is defined to be the coefficient of $\der z/ z$, valued at $p$. \\
  \tab We say that $E$ is stable (resp. semistable) if for every nonzero proper subbundle $F \subset E$, with $\theta(F) \subset F \otimes \Omega_{\bar X}^1(\log D)$, we have $\mu(F) < \mu(E)$ (resp. $\mu(F) \leq \mu(E)$) where $\mu = \deg / \rank$.
\end{definition}

Consider $(E, \nabla)$ a flat connection on $X = \bar{X} \backslash D$ with open immersion $j: X \hookrightarrow \bar{X}$. By the Riemann-Hilbert correspondence, this corresponds to a local system on $X$. Suppose that the local monodromy $T_p$ around every $p \in D$ is unipotent, and set 
\[N_p = \frac{-1}{2\pi i} \log T_p = \frac{-1}{2\pi i} \Sum{k \geq 1} \frac{(-1)^{k + 1}}{k}(T_p - I)^k\] 
then $N_p$ is nilpotent, and we have the functorial Deligne's canonical extension $(\bar{E}, \nabla)$, see \cite{deligne2006equations}, which is the logarithmic flat connection on $(\bar{X}, D)$ satisfying
\begin{itemize}
  \item the residue at $p$ is $N_p$ with eigenvalues 0;
  \item $j_*E \simeq \bar{E} \otimes_{\sh{O}_{\bar{X}}} \sh{O}_{\bar{X}}(*D)$, hence we also call $\bar{E}$ a logarithmic lattice for the meromorphic extension.
\end{itemize}
\tab More explicitly, if $s$ is a multivalued flat section on the punctured disk around $p$, then $\exp(N \log z)s$ is single-valued, and we can form a local basis of $\bar{E}$ using these single-valued sections. \\
\tab If $E, F$ have nilpotent residues $N_E, N_F$ at $p \in D$, then 
\[N_{E \otimes F} = N_E \otimes 1 + 1 \otimes N_F, \tab  N_{E^\vee} = -N_E^\vee, \tab N_{\underline{\Rm{Hom}}(E, F)}: \phi \mapsto N_F \circ \phi - \phi \circ N_E\]
and consequently, Deligne's canonical extension commutes with tensor products, duals, and internal $\underline{\Rm{Hom}}$. 

\subsection{Mixed twistor structures and algebras} Mixed twistor structures were introduced by Simpson in \cite{simpson1997mixed} as a generalization of Hodge structures.
\begin{definition}
  A complex mixed twistor structure (MTS) is a pair $(V, W_\bullet)$, where $V$ is a vector bundle on $\PP^1$, and $W_\bullet V$ is a locally split increasing filtration, called the weight filtration, by subbundles such that 
  \[\Rm{gr}^W_k V \simeq \sh{O}_{\PP^1}(k)^{\oplus m_k} \tab k \in \ZZ\]
  We say that $(V, W_\bullet)$ is pure of weight $k$ if $V \simeq \sh{O}_{\PP^1}(k)^{\oplus m}$. A morphism of MTS is an $\sh{O}_{\PP^1}$-linear filtered map.
\end{definition}

These form an abelian, rigid tensor category $\mathsf{MTS}_\CC$ where every morphism is automatically strict with respect to the weight filtration. In $\mathsf{MTS}_\CC$, tensor products add weights, duals reverse weights, and we have an unit 
\[\sh{T}(0) \coloneq \sh{O}_{\PP^1}\]
pure of weight 0. 
\begin{lemma}[{\cite[Lemma 4.15]{bakker2024linear}}]
  \label{lem:delignesplitting}
  Fix a point $\lambda \in \PP^1$. Every mixed twistor structure admits a functorial splitting of its weight filtration when restricted to $\PP^1 \backslash \set{\lambda}$. Furthermore, the splitting is compatible with tensor products and duals. 
\end{lemma}

\begin{definition}[$\sh{T}(0)-$MTS-algebra]
  A $\sh{T}(0)$-MTS-algebra is a commutative unital algebra object $\sh{A}$ in $\mathsf{MTS}_\CC$, i.e., $\sh{A} \in \mathsf{MTS}_\CC$ with morphisms of MTS
  \[\sh{T}(0) \xrightarrow{u} \sh{A}, \tab \sh{A} \otimes \sh{A} \xrightarrow{\mu_{\sh{A}}} \sh{A}\]
  satisfying the usual unit, associativity, and commutativity diagrams. Similarly, we can define a $\bm{\Lambda}-$MTS-algebra (and module) from a $\sh{T}(0)-$MTS-algebra $\bm{\Lambda}$. \\
  \tab Now, $\sh{A}$ is also a coherent commutative $\sh{O}_{\PP^1}-$algebra, so we can view it as a relative affine scheme $\Spec_{\PP^1} \sh{A} \to \PP^1$. Then it makes sense to define $\abs{\sh{A}} \coloneq \sh{A}\vert_{\lambda = 1}$, which is a complex algebra. 
\end{definition}

\begin{definition}[Artinian local $\sh{T}(0)-$MTS-algebra]
  An artinian local $\sh{T}(0)$-MTS-algebra is a $\sh{T}(0)-$MTS-algebra $\sh{A}$ for which $\abs{\sh{A}}$ is a local artinian $\CC$-algebra with residue field $\CC$, together with a morphism $\sh{A} \twoheadrightarrow \sh{T}(0)$. \\
  \tab Since $\mathsf{MTS}_\CC$ is abelian, $\Ide{m}_{\sh{A}} \coloneq \ker(\sh{A} \twoheadrightarrow \sh{T}(0))$ is a locally free sheaf on $\PP^1$. We have a splitting coming from the unit morphism $\sh{T}(0) \to \sh{A}$, hence a decomposition $\sh{A} = \sh{T}(0) \oplus \Ide{m}_{\sh{A}}$. Morphisms of artinian local $\sh{T}(0)$-MTS-algebras preserve the unit and residue morphisms, and together these form a category $\mathsf{MTS\text{-}Art}$. 
\end{definition}

We have the following affine splitting, essentially as a consequence of \cite[Lemma 4.35]{bakker2024linear}. 
\begin{lemma}
  \label{lem:affinesplitting}
  Let $\sh{A}$ be an artinian $\sh{T}(0)-$MTS-algebra and $\sh{K} = \brac*{\sh{K}^\bullet, \der^\bullet}$ a bounded complex of $\sh{A}$-MTS-modules. On $\AAA^1_\lambda \coloneq \PP^1 \setminus \set{\infty}$ we have a functorial isomorphism of complexes of $\abs*{\sh{A}} \otimes_\CC \sh{O}_{\AAA^1_\lambda}$-modules
  \[\varphi_{\sh{K}}: \sh{K}\vert_{\AAA^1_\lambda} \simeq \sh{O}_{\AAA^1_\lambda} \otimes_\CC \abs*{\sh{K}}\]
  which is compatible with morphisms, derived tensor products, and artinian $\sh{T}(0)$-MTS-algebra coefficient base change. 
\end{lemma}
\begin{proof}
  The functorial splitting of \cref{lem:delignesplitting} gives 
  \[\phi_{\sh{A}}:  \sh{A}\vert_{\AAA^1_\lambda} \xrightarrow{\simeq} \brac*{\Rm{gr}^W \sh{A}}\vert_{\AAA^1_\lambda} \]
  \tab Since $\Rm{gr}^W_k \sh{A}$ is pure of weight $k$, evaluation gives an isomorphism 
  \[\op{ev}_k: H^0\brac*{\PP^1, \Rm{gr}^W_k \sh{A} \otimes \sh{O}_{\PP^1}(-k)} \otimes_\CC \sh{O}_{\PP^1}(k) \xrightarrow{\simeq} \Rm{gr}^W_k \sh{A}\]
  \tab Now, notice that we also have an isomorphism coming from evaluation at $\lambda = 1$:
  \[H^0\brac*{\PP^1, \Rm{gr}^W_k \sh{A} \otimes \sh{O}_{\PP^1}(-k)} \xrightarrow{\simeq} \Rm{gr}^W_k \sh{A}\vert_{\lambda = 1} = \abs*{\Rm{gr}^W_k \sh{A}}\]
  thus we can trivialize $\sh{O}_{\PP^1}(k)\vert_{\AAA^1_\lambda}$ to get $\tau_{\sh{A}}: \Rm{gr}^W_k \sh{A}\vert_{\AAA^1_\lambda} \xrightarrow{\simeq} \sh{O}_{\AAA^1_\lambda} \otimes_\CC \abs{\Rm{gr}^W_k \sh{A}}$ which is also functorial (since evaluations are functorial) and just the identity at $\lambda = 1$. Collecting these isomorphisms, we end up with 
  \[\varphi_{\sh{A}}\coloneq (1 \otimes \phi_{\sh{A},1}^{-1}) \circ \tau_{\sh{A}} \circ \phi_{\sh{A}}: \sh{A}\vert_{\AAA^1_\lambda} \xrightarrow{\simeq} \sh{O}_{\AAA^1_\lambda} \otimes_\CC \abs{\sh{A}} \] 
  which is functorial, and compatible with tensor products. Then functoriality with respect to the multiplication $\sh{A} \otimes \sh{A} \to \sh{A}$ and unit $\sh{T}(0) \to \sh{A}$ implies that $\varphi_{\sh{A}}$ is an isomorphism of algebras. Let $\sh{M}$ be a $\sh{A}-$MTS-module, then we have the same isomorphism $\varphi_{\sh{M}}$, and functoriality with respect to the module map $\sh{A} \otimes \sh{M} \to \sh{M}$ gives that $\varphi_{\sh{M}}$ is an isomorphism of $\sh{O}_{\AAA^1_\lambda} \otimes_\CC \abs*{\sh{A}}$-modules. \\
  \\
  \tab Moving on to a complex $\sh{K} = \brac*{\sh{K}^\bullet, \der^\bullet}$, for each $i$ we have 
  \[\varphi_{\sh{K}^i}: \sh{K}^i\vert_{\AAA^1_\lambda} \xrightarrow{\simeq} \sh{O}_{\AAA^1_\lambda} \otimes \abs{\sh{K}^i}\]
  and since each $\der^i$ is a morphism of $\sh{A}-$MTS-modules, functoriality assembles these $\varphi_{\sh{K}^i}$ into a isomorphism $\varphi_{\sh{K}}$of complexes. Clearly $\varphi_{\sh{K}}$ is functorial. Now suppose we have another complex $\sh{M} = \brac*{\sh{M}^\bullet, \der^\bullet}$, then individual tensor compatibility and flatness of $\sh{O}_{\AAA^1_\lambda}$ over $\CC$ implies 
  \begin{align*}
    \brac*{\sh{K} \otimes^{\Rm{L}}_\sh{A} \sh{M}} \vert_{\AAA^1_\lambda} &\simeq \sh{K}\vert_{\AAA^1_\lambda} \otimes^{\Rm{L}}_{\sh{A}\vert_{\AAA^1_\lambda}} \sh{M}\vert_{\AAA^1_\lambda} &\text{since open restriction is exact}\\
    &\simeq \brac*{\sh{O}_{\AAA^1_\lambda} \otimes_\CC \abs{\sh{K}}} \otimes^{\Rm{L}}_{\sh{O}_{\AAA^1_\lambda} \otimes_\CC \abs{\sh{A}}} \brac*{\sh{O}_{\AAA^1_\lambda} \otimes_\CC \abs{\sh{M}}}\\
    &\simeq \sh{O}_{\AAA^1_\lambda} \otimes_\CC \brac*{\abs{\sh{K}} \otimes^{\Rm{L}}_{\abs{\sh{A}}} \abs{\sh{M}}} &\text{since $\sh{O}_{\AAA^1_\lambda} \otimes_\CC \abs{\sh{A}}$ is flat over $\abs{\sh{A}}$} 
  \end{align*}
  \tab Finally, for $f: \sh{A} \to \sh{A}'$ a morphism of artinian $\sh{T}(0)-$MTS-algebras, 
  \[\brac*{\sh{A}' \otimes^{\Rm{L}}_\sh{A} \sh{K}} \vert_{\AAA^1_\lambda} \simeq \sh{O}_{\AAA^1_\lambda} \otimes_\CC \brac*{\abs{\sh{A}'} \otimes^{\Rm{L}}_{\abs{\sh{A}}} \abs{\sh{K}}}\] as complexes of $\abs{\sh{A}'} \otimes_\CC \sh{O}_{\AAA^1_\lambda}-$modules, giving compatibility with artinian base change. 
\end{proof}

\begin{definition}[Pro-MTS and pro-$\sh{T}(0)$-MTS-algebra]
  \leavevmode
  \begin{enumerate}
    \item A pro-MTS is a pro-object in $\mathsf{MTS}_\CC$. 
    \item A pro-$\sh{T}(0)$-MTS-algebra is a $\sh{T}(0)$-algebra object in the category of pro-MTS. Equivalently, it's an inverse system of $\sh{T}(0)-$MTS algebras.
    \item A pro-$\sh{T}(0)$-MTS-algebra $\sh{A}$ is complete, noetherian, artinian, or local when $\abs{\sh{A}}$ is as a $\CC$-algebra. 
    \item If $\bm{\Lambda}$ is a pro-$\sh{T}(0)$-MTS-algebra, we define a pro-$\bm{\Lambda}$-MTS-module to be a pro-$\bm{\Lambda}$-module object in $\mathsf{MTS}_\CC$. 
  \end{enumerate}
\end{definition}
\begin{remark}
  In this paper, every pro-$\sh{T}(0)$-MTS-algebra $\bm{\Lambda}$ is complete noetherian with 
  \[\bm{\Lambda} = \varprojlim_n \bm{\Lambda}_n, \tab \bm{\Lambda}_n \coloneq \bm{\Lambda}/\Ide{m}_{\bm{\Lambda}}^{n + 1}\]
  where each $\bm{\Lambda}_n$ is an artinian local $\sh{T}(0)-$MTS-algebra, with quotients $\bm{\Lambda}_{n + 1} \to \bm{\Lambda}_n$ being morphisms in $\mathsf{MTS}_\CC$. In this situation, a pro-$\bm{\Lambda}$-MTS-module can be defined levelwise. We can also define a formal scheme $\op{Spf} \abs{\bm{\Lambda}}$. 
\end{remark}

\subsection{Harmonic bundles and variations of twistor structures}
We first review harmonic bundles, following \cite[Section 7.3]{bakker2024linear}. This theory was developed by Corlette in \cite{corlette1988flat}, Simpson in \cite{simpson1990harmonic, simpson1992higgs}, and by Mochizuki in \cite{mochizuki2006kobayashi, mochizuki2009kobayashi}. 
\begin{definition}
  \label{def:harmonic-decomp1}
  Let $\sh{E}$ be a complex $\sh{C}^\infty$-vector bundle on a complex manifold $X$ equipped with a flat $C^\infty$-connection $D$. For $h$ a hermitian metric on $\sh{E}$, we have a decomposition
  \[D = D_h + \Psi\]
  where $D_h$ is unitary with respect to $h$, and $\Psi$ is self-adjoint for $h$. Take a further type decomposition into $(1, 0)$ and $(0, 1)$ components
  \[D_h = \partial_h + \bar{\partial}_h, \tab \Psi = \theta + \theta^*\]
  \tab We say that $h$ is pluriharmonic if $(\theta + \bar{\partial}_h)^2 = 0$. A harmonic bundle $(\sh{E}, D, h)$, equivalently $(\sh{E}, \bar{\partial}_h, \theta, h)$, is then a smooth complex vector bundle $\sh{E}$ with a flat smooth connection $D$ and a pluriharmonic metric $h$. 
\end{definition}

\begin{remark}
  A harmonic bundle $(\sh{E}, D, h)$ can be recovered from the data $(\sh{E}, \bar{\partial}_h, \theta, h)$ since $D_h$ is the unique Chern connection on $(E, \bar{\partial}_h)$ compatible with $h$, and $\theta^*$ is just the adjoint of $\theta$. 
\end{remark}

\begin{definition}
  \label{def:harmonic-decomp2}
  Given a harmonic bundle $(\sh{E}, \bar{\partial}_h, \theta, h)$ with the above decomposition, the associated holomorphic flat connection is 
  \[\brac*{E^{\Rm{DR}}, \nabla} = \brac*{(\sh{E}, D^{0, 1} = \bar{\partial}_h + \theta^*), D^{1, 0} = \partial_h + \theta}\]
  i.e., the holomorphic structure of $E^{\Rm{DR}}$ is given by $\bar{\partial}_h + \theta^*$, and the holomorphic connection is given by $\partial_h + \theta$. The associated Higgs bundle is 
  \[\brac*{E^{\Rm{Dol}}, \theta} = \brac*{(\sh{E}, \bar{\partial}_h), \theta}\]
  i.e., the holomorphic structure on $E^{\Rm{Dol}}$ is given by $\bar{\partial}_h$ and the Higgs field is $\theta$. 
\end{definition}

\begin{definition}
  Let $(\bar{X}, D)$ be a smooth projective complex variety $\bar{X}$ with a simple normal crossing divisor $D$. A harmonic bundle $(\sh{E}, \bar{\partial}_h, \theta, h)$ on $X = \bar{X} \setminus D$ is tame if $(E^{\Rm{Dol}}, \theta)$ is the restriction of a logarithmic Higgs bundle on $(\bar{X}, D)$. \\
  \tab A tame harmonic bundle $(\sh{E}, \bar{\partial}_h, \theta, h)$ is purely imaginary (resp. nilpotent) if the eigenvalues of the residues of $\theta$ in a logarithmic extension $\brac*{\bar{E}^{\Rm{Dol}}, \theta}$ are purely imaginary (resp. zero). We say that $(\sh{E}, \bar{\partial}_h, \theta, h)$ has unipotent local monodromy if the local system corresponding to $\brac*{E^{\Rm{DR}}, \nabla}$ has unipotent local monodromy. 
\end{definition}
\begin{remark}
  These are independent of the choice of log smooth compactification $(\bar{X}, D)$ of $X$, and of the logarithmic extension $\brac*{\bar{E}^{\Rm{Dol}}, \theta}$. 
\end{remark}
\begin{theorem}[{\cite[Theorem 1.1]{mochizuki2009kobayashi}}]
  \label{thm:mochizuki-correspondence}
  Let $\brac*{\bar{X}, D}$ be a log smooth curve. There is an equivalence of categories, via purely imaginary tame harmonic bundles with unipotent local monodromy, 
  \[\mathsf{SM}^{\Rm{ps}}: \mathsf{Conn}^{\Rm{ss}}_{\Rm{nilp}}(\bar{X}, D) \xrightarrow{\sim} \mathsf{Higgs}_{\Rm{nilp}}^{0, \Rm{poly}}(\bar{X}, D)\]
  between semisimple logarithmic flat vector bundle with nilpotent residues, and degree 0 polystable logarithmic Higgs bundles with nilpotent residues. 
\end{theorem}

Let $\scr{A}_X \coloneq \sh{C}^\infty_X \boxtimes \sh{O}_{\PP^1}$ be the sheaf of $C^\infty$-functions on $X_{\PP^1} \coloneq X \times \PP^1$ which are holomorphic in the $\PP^1$-direction, and choose generating sections $x_0, x_\infty \in H^0(\PP^1, \sh{O}_{\PP^1}(1))$ vanishing at $0$ and $\infty$ respectively. Then we get a derivative 
\[x_0 \partial_{X_{\PP^1}/\PP^1} + x_\infty \bar{\partial}_{X_{\PP^1}/\PP^1}: \scr{A}_X \to \scr{A}_X \otimes \Omega^1_{X_{\PP^1}/\PP^1}(1)\]
with $\Omega^1_{X_{\PP^1}/\PP^1}(1) \coloneq \Omega_X^1 \boxtimes \sh{O}_{\PP^1}(1)$, where $\Omega_X^1$ is the sheaf of $C^\infty$-one-forms on $X$. Let $\brac*{\sh{E}, \bar{\partial}_h, \theta, h}$ be a tame harmonic bundle on $X$, and define $\sh{E}_{\PP^1} \coloneq \sh{E} \boxtimes \sh{O}_{\PP^1}$. Then we have a $x_0 \partial_{X_{\PP^1}/\PP^1} + x_\infty \bar{\partial}_{X_{\PP^1}/\PP^1}$-connection 
\[\sh{D} \coloneq x_0(\partial_h + \theta^*) + x_\infty(\theta + \bar{\partial}_h): \sh{E}_{\PP^1} \to \sh{E}_{\PP^1} \otimes \Omega^1_{X_{\PP^1}/\PP^1}(1), \tab \sh{D}^2 = 0\]
where we abuse notation so that $\partial_h + \theta^*, \theta + \bar{\partial}_h$ really mean their natural extensions to $\sh{E}_{\PP^1}$. 
\begin{definition}
  The resulting $\brac*{\sh{E}_{\PP^1}, \sh{D}}$ is the variation of pure twistor structures ($\CC-$VTS) of weight 0 associated to the tame harmonic bundle $\brac*{\sh{E}, \bar{\partial}_h, \theta, h}$. We get the weight $k$ $\CC$-VTS associated to $\brac*{\sh{E}, \bar{\partial}_h, \theta, h}$ by instead taking $\sh{E}_{\PP^1, k} \coloneq \sh{E} \boxtimes \sh{O}_{\PP^1}(k)$. 
\end{definition}
\begin{remark}
  Every semisimple complex local system on $X$ underlies a purely imaginary tame harmonic bundle, hence a variation of pure twistor structures. 
\end{remark}

\begin{definition}
  \label{def:C-VMTS-def}
  A graded-polarized variation of mixed twistor structures ($\CC$-VMTS) on $X$ is a triple $\brac*{\scr{V}, W_\bullet, \sh{D}}$ such that 
  \begin{itemize}
    \item $\scr{V}$ is an $\mathscr{A}_X$-module on $X_{\PP^1}$;
    \item $\sh{D}: \scr{V} \to \scr{V} \otimes \Omega^1_{X_{\PP^1}/\PP^1}(1)$ is a $x_0 \partial_{X_{\PP^1}/\PP^1} + x_\infty \bar{\partial}_{X_{\PP^1}/\PP^1}$-connection with $\sh{D}^2 = 0$;
    \item $W_\bullet \scr{V}$ is a $\sh{D}-$flat increasing weight filtration of $\sh{E}_{\PP^1}$;
    \item each $\brac*{\Rm{gr}^W_k \scr{V}, \Rm{gr}_k^W \sh{D}}$ has a pluriharmonic hermitian metric $h_k$ making it a $\CC-$VTS of weight $k$ associated to a purely imaginary tame harmonic bundle. 
  \end{itemize}
  Morphisms are $\scr{A}_X$-linear morphisms which preserve $W_\bullet$ and commutes with $\sh{D}$. 
\end{definition}
\begin{remark}
  The data of a $\CC$-VMTS gives a family of holomorphic $\lambda$-connections, as follows. Look at the chart $\PP^1\setminus \set{\infty}$, with $x_\infty$ being the trivializing section for $\sh{O}_{\PP^1}(1)$. The restriction of $(\scr{V}, \sh{D})$ to a fiber $\lambda = x_0/x_\infty$ gives a $C^\infty$-connection 
  \[D_{\lambda} \coloneq x_\infty^{-1} \sh{D}\vert_{X \times \set{\lambda}}: \sh{V}_\lambda \to \sh{V}_\lambda \otimes \Omega_X^1\]
  \tab Taking the type decomposition $D_\lambda = D_\lambda^{1, 0} + D_\lambda^{0, 1}$, since $D_\lambda$ is a $\lambda \partial + \bar{\partial}$-connection, 
  \[D_\lambda^{0, 1}(fs) = fD_\lambda^{0, 1}s + \bar{\partial}f \otimes s, \tab D_\lambda^{1, 0}(fs) = fD_\lambda^{1, 0}s + \lambda \partial f \otimes s\]
  and if $f$ is holomorphic, $D_\lambda^{1, 0}(fs) = fD_\lambda^{1, 0}s + \lambda \der f \otimes s$. Then $D^{0, 1}_\lambda$ gives holomorphic structure on $\sh{V}_\lambda$, and $D_\lambda^{1, 0} D_\lambda^{0, 1} + D_\lambda^{0, 1} D_\lambda^{1, 0} = (D_\lambda^2)^{1, 1} = 0$ implies that $D^{1, 0}_{\lambda}$ is a holomorphic $\lambda$-connection with respect to $(\sh{V}_\lambda, D_\lambda^{0, 1})$.
\end{remark}

\subsection{Twistor modules and admissible variations with coefficients}
We treat the theory of twistor modules as a blackbox. For concrete details and proofs, see \cite{mochizuki2015mixed,sabbah2005polarizable}. \\
\tab Let $\sh{D}_{Y}$ be the sheaf of differential operators on a smooth complex variety $Y$, with its order filtration $F_\bullet$. We then define $\scr{R}_{Y} \coloneq \bigoplus_{k \geq 0} F_k \sh{D}_{Y} \lambda^k$. Locally this is generated over $\sh{O}_{Y \times \AAA^1_\lambda} = \sh{O}_Y[\lambda]$ by $\set*{\lambda \partial_{y_i}}_{i = 1}^{\dim Y}$. An $\scr{R}_{Y}-$module structure should be thought of as the data of a $\lambda-$connection, so a $\scr{R}_{Y}-$module which is locally free over $\sh{O}_Y[\lambda]$ gives a flat connection and a Higgs bundle by taking the fibers $\lambda = $1 and 0 respectively.

\begin{definition}
  An $\scr{R}_{Y}$-triple $\brac*{\scr{M}', \scr{M}'', C}$ consists of two coherent $\scr{R}_{Y}$-modules and a sesquilinear pairing along $\abs{\lambda} = 1$, see \cite[Section 2.1]{mochizuki2015mixed}. Mochizuki's category $\mathsf{MTM}(Y)$ of graded-polarized mixed twistor $\sh{D}$-modules is a full subcategory of the category of filtered $\scr{R}_{Y}-$triples, defined by imposing certain technical conditions on the nearby and vanishing cycles, see \cite[Section 7.1 and 7.2]{mochizuki2015mixed}. 
\end{definition}

\begin{example}
  A polarized $\CC-$VTS (resp. graded-polarized $\CC-$VMTS) on $Y$ gives a smooth polarized pure (resp. graded-polarized mixed) twistor $\sh{D}$-module on $Y$. For the pure case, see \cite[Theorem 1.2]{mochizuki2003asymptotic}, and for the mixed case, see \cite[Proposition 10.2.11 and 10.3.1]{mochizuki2015mixed} applied to empty boundary. 
\end{example}

\begin{remark}
  The category $\mathsf{MTM}$ is abelian \cite[Proposition 7.2.5]{mochizuki2015mixed} and enjoys nice functorialities. In particular, there are formalisms of derived pushforward, tensor products, and internal Homs \cite[Section 14.3]{mochizuki2015mixed}.
\end{remark}

\begin{definition}
  Let $(\bar{X}, D)$ be a projective log smooth curve with $X = \bar{X} \setminus D$, an admissible variation of mixed twistor structure ($\CC$-AVMTS) on $X$ is a mixed twistor $\sh{D}-$module on $\bar{X}$ whose restriction to $X$ underlies a tame purely imaginary graded-polarizable $\CC-$VMTS. 
\end{definition}

\begin{example}
  A $\CC-$VTS, associated to a tame purely imaginary harmonic bundle, has a canonical admissible extension \cite[Theorem 19.2]{mochizuki2003asymptotic}. 
\end{example}

\begin{definition}
  For a $\CC$-algebra $A$, an $A$-local system is a locally constant sheaf of finitely generated $A$-modules. We say that a $A$-local system is free if its stalks are finite free $A$-modules. 
\end{definition}
In the following, we will abuse notation to let $\sh{A}$ be both a $\sh{T}(0)$-MTS-algebra, which is a vector bundle over $\PP^1$, as well as the constant variation on $X$. 
\begin{definition}
  Let $\sh{A}$ be an artinian local $\sh{T}(0)$-MTS-algebra. An $\sh{A}$-AVMTS is an AVMTS $\scr{T}$ on $\bar{X}$ with a $\CC$-VMTS $\scr{V} = j^*\scr{T}$ and an action morphism $\sh{A} \otimes_{\sh{T}(0)} \scr{V} \to \scr{V}$ which satisfies the unit and associativity diagrams. The underlying local system of $\scr{V}$ is a $\abs{\sh{A}}-$local system. It's free if the underlying $\abs{\sh{A}}-$local system is. \\
  \tab If $\bm{\Lambda}$ is a pro-$\sh{T}(0)-$MTS-algebra, then a pro-$\bm{\Lambda}-$AVMTS is of the form 
  \[\scr{V} = \varprojlim_n \scr{V}_n, \tab \scr{V}_n = \scr{V}/\Ide{m}_{\bm{\Lambda}}^{n + 1}\scr{V}\]
  where $\scr{V}$ has an underlying $\abs{\bm{\Lambda}}-$local system, with each $\scr{V}_n$ a $\CC$-AVMTS for which the quotient map $\scr{V}_{n + 1} \to \scr{V}_n$ is a morphism of $\CC$-AVMTS, and for which $\bm{\Lambda}/\Ide{m}_{\bm{\Lambda}}^{n + 1} \subset \Rm{End}(\scr{V}_n)$ is a sub $\CC-$MTS. We have a similar notion of freeness here. 
\end{definition}
\begin{definition}[artinian $\sh{T}(0)$-MTS-algebra coefficient base change]
  \label{lem:artinian-coeff-base-change}
  It suffices to make sense of the tensor product $\sh{A}' \otimes_{\sh{A}}$ for $\sh{A} \xrightarrow{\phi} \sh{A}'$ a morphism in $\Sf{MTS\text{-}Art}$. For an $\sh{A}$-AVMTS $\scr{V}$, define 
  \[\scr{V}_{\sh{A}'} = \sh{A}' \otimes_{\sh{A}} \scr{V} = \coker\brac*{\sh{A}' \otimes \sh{A} \otimes \scr{V} \xrightarrow{\alpha - \beta} \sh{A}' \otimes \scr{V}}\]
  where $\alpha: \sh{A}' \otimes \sh{A} \otimes \scr{V} \xrightarrow{1 \otimes \phi \otimes 1} \sh{A}' \otimes \sh{A}' \otimes \scr{V} \xrightarrow{\mu_{\sh{A}'} \otimes 1} \sh{A}' \otimes \scr{V}$, and $\beta: \sh{A}' \otimes (\sh{A} \otimes \scr{V}) \to \sh{A}' \otimes \scr{V}$. The resulting $\sh{A}' \otimes_{\sh{A}} \scr{V}$ is still an AVMTS since the category of $\CC$-AVMTS is abelian. Fiberwise this is just the regular tensor product. 
\end{definition}

Let $\scr{T}$ be a $\CC-$AVMTS on projective log smooth curve $(\bar{X}, D)$ with $j: X \hookrightarrow \bar{X}$, and $\scr{V} = j^*\scr{T}$ be the $\CC$-VMTS on $X$ with unipotent local monodromy. Then we have the direct image $j_*\scr{V} \in \mathsf{MTM}(\bar{X})$,  with $\scr{M}[*D]$ the underlying $\scr{R}_{\bar{X}}-$module on the affine patch $\AAA^1_\lambda = \PP^1 \setminus \set{\infty}$ (for its construction, see \cite[Lemmas 3.1.1-3.1.4]{mochizuki2015mixed}). For every $p \in D$ with local equation $(z = 0)$, there is a canonical Kashiwara-Malgrange V-filtration $V_\bullet \scr{M}[*D]$ such that we have  
\[\lambda \partial_z: \Rm{gr}^V_b \scr{M}[*D] \xrightarrow{\simeq} \Rm{gr}^V_{b + 1} \scr{M}[*D] \tab b > -1, \tab z: V_0\scr{M}[*D] \xrightarrow{\simeq} V_{-1}\scr{M}[*D]\tag{2.1}\label{eq:V-filt-isoms}\]
where the two isomorphisms comes from \cite[Remark 3.3.9(5) and Section 3.4.a]{sabbah2005polarizable}. 
\begin{definition}
  \label{def:AVMTS-extension}
  Suppose that $\scr{T}$ is a free $\sh{A}$-AVMTS with unipotent local monodromy, then one has a notion of functorial Deligne's extension $(\scr{V}, W_\bullet \scr{V}, \bar{\scr{V}}, \bar{\scr{V}}', \sh{D})$ for $(\scr{V}, W_\bullet, \sh{D}) = j^*\scr{T}$ as well, see \cite[Lemma 8.7 and Corollary 8.8]{bakker2024linear}. We can view $(\scr{V}, \sh{D})$ as a family of $\lambda-$connections (see the remark following \cref{def:C-VMTS-def}), and $\bar{\scr{V}}$ and $\bar{\scr{V}}'$ are logarithmic extensions on the charts $\AAA^1_\lambda, \AAA^1_{\lambda^{-1}}$, respectively. The $\sh{A}$-action on $\scr{V}$ extends to $\abs{\sh{A}} \otimes \sh{O}_{\AAA^1_\lambda}$ and $\abs{\sh{A}} \otimes \sh{O}_{\AAA^1_{\lambda^{-1}}}$-actions on $\bar{\scr{V}}$ and $\bar{\scr{V}}'$. 
\end{definition}
\begin{remark}
  Over $\lambda \neq 0,  \infty$, these are just the normal Deligne's canonical extensions (we can scale each of these by $\lambda^{-1}$ to get a flat connection). Since $D \times \set{\infty}$ and $D \times \set{0}$ are codimension 2 in $\bar{X} \times \AAA^1_\lambda$ and $\bar{X} \times \AAA^1_{\lambda^{-1}}$, these extensions $\bar{\scr{V}}, \bar{\scr{V}}'$ can be classified as the unique locally free logarithmic extensions which agree with the Deligne's canonical extensions over $\lambda \neq 0, \infty$.
\end{remark}

\begin{lemma}
  \label{lem:boundary-lattices}
  Let $\sh{A}$ be an artinian local $\sh{T}(0)-$MTS-algebra, and $\scr{V} = j^*\scr{T}$ a free $\sh{A}$-AVMTS with unipotent local monodromy. Let $\scr{M}[*D]$ be the underlying $\scr{R}_{\bar{X}}$-module of $j_*\scr{V}$. Then, for every $p \in D$ with local equation $(z = 0)$, then we have identifications
  \[V_{-1}\scr{M}[*D] = \bar{\scr{V}}, \tab V_0\scr{M}[*D] = z^{-1}\bar{\scr{V}}\]
  which are functorial in $\scr{V}$ and commute with artinian $\sh{T}(0)$-MTS-algebra coefficient base change. 
\end{lemma}

\begin{proof}
  For $\lambda_0 \in \AAA^1_\lambda$, the superscript $^{(\lambda_0)}$ will denote taking a smaller family of $\lambda-$connections with $\lambda$ in a small neighborhood of $\lambda_0$. Let $\scr{M} = \scr{R}_{\bar{X}}(*D) \otimes_{\scr{R}_{\bar{X}}} \scr{M}[*D]$ then \cite[Morphism (5.15)]{mochizuki2015mixed}, with $I = \set{p}, J = K = \emptyset$, simplifies to
  \[V_0\scr{R}_{\bar{X}} \otimes_{V_0\scr{R}_{\bar{X}}} \brac*{V_0\scr{R}_{\bar{X}} \cdot Q_{b + 1}^{(\lambda_0)} \scr{M}} \to \scr{M}^{(\lambda_0)}[*D] \tag{2.3}\label{eq:map-def-V-filt}\]
  with the image defined to be $V_b^{(\lambda_0)} \scr{M}^{(\lambda_0)}[*D]$. Notice that $Q_{b}^{(\lambda_0)} \scr{M}$ is a logarithmic lattice \cite[Equation (5.2)]{mochizuki2015mixed}, and $V_0\scr{R}_{\bar{X}}$ is generated by $z\lambda\partial_z$ which preserves logarithmic lattices. Furthermore, by \cite[Lemma 5.3.8]{mochizuki2015mixed} the morphism \eqref{eq:map-def-V-filt} is injective, and we get 
  \[V_{-1}^{(\lambda_0)} \scr{M}^{(\lambda_0)}[*D] = Q_0^{(\lambda_0)}\scr{M}\]
  \tab Finally, $Q_0^{(\lambda_0)}\scr{M}$ has nilpotent residues due to the fact that the relevant KMS datum is $(0, 0)$ and \cite[Section 5.1.2]{mochizuki2015mixed} , so it agrees with the Deligne extension for $\lambda \neq 0$. After gluing we get $V_{-1}\scr{M}[*D] = Q_0\scr{M}$ where $Q_0\scr{M}$ is a locally free extension of $\scr{V}$ which agrees with $\bar{\scr{V}}$ over $\lambda \neq 0$, i.e., they agree except for a codimension 2 locus. As a result they must actually agree, and we get the first equality; the second follows from \eqref{eq:V-filt-isoms}.\\
  \tab By functoriality of the V-filtration and the Deligne extension $\bar{\scr{V}}$, both are $A[\lambda]$-submodules of $\scr{M}[*D]$. So the first equality is that of $A[\lambda]-$submodules. Since this is an equality, to prove that the identification commutes with artinian $\sh{T}(0)-$MTS-algebra base change, it suffices to show that is true for the Deligne extension $\bar{\scr{V}}$, i.e., for $\sh{A} \to \sh{A}'$ in $\Sf{MTS\text{-}Art}$ with $A' = \abs*{\sh{A}'}$ we need $\bar{\scr{V}}_{\sh{A}'} = \bar{\scr{V}} \otimes_{A[\lambda]} A'[\lambda]$. Over $\lambda = 1$, the local system $\scr{V}_{\sh{A}'}\vert_{\lambda = 1}$ is by construction (\cref{lem:artinian-coeff-base-change}) just $A' \otimes \scr{V}\vert_{\lambda = 1}$. Hence the new local monodromy operator around $p \in D$ is $T_{p, A'} = 1 \otimes T_{p, A}$. The same is true of $N_{p, A'} = \frac{-1}{2\pi i} \log T_{p, A'}$ since it's a finite polynomial of $T_{p, A'}$, and thus the exponential lattice also commutes with $A \to A'$. Over $\lambda \neq 0$, the Deligne extension is defined via scaling by $\lambda^{-1}$, so we get that $\bar{\scr{V}}_{\sh{A}'} = \bar{\scr{V}} \otimes_{A[\lambda]} A'[\lambda]$ over $\bar{X} \times \AAA^1_\lambda - D \times \set{0}$. Again, these are locally free and $D \times \set*{0}$ is codimension 2 in $\bar{X} \times \AAA^1_\lambda$, so we get a global isomorphism. Functoriality of the Deligne extension over $\bar{X} \times \AAA^1_\lambda$ gives agreement as $A'[\lambda]-$modules. 
\end{proof}

\section{Constructions of moduli spaces and deformation theory}
For the rest of this section, we work with $(\bar{X}, D)$ consisting of a smooth projective curve $\bar{X}$, and a finite set of reduced points $D$; $X \coloneq \bar{X} \setminus D$. 
\subsection{The framed Betti character variety and universal deformation algebra}
\begin{definition}
  For a rank $r \geq 1$ and a base point $x \in X(\CC)$, we define the framed character variety to be the affine scheme
  \[R_B(X, x, r) \coloneq \Hom\brac*{\pi_1(X^{\Rm{an}}, x), \Rm{GL}_r}\]
\end{definition}
\begin{theorem}[{\cite[Theorem 4.2]{bakker2024linear}}]
  \label{thm:Betti-MTS-algebra}
  Let $(V, \phi)$ be a framed rank $r$ local system underlying a $\CC$-AVMTS $\scr{V}$, then the completed local ring $\hat{\sh{O}}_{R_B(X, x, r), (V, \phi)}$ has a pro$-\sh{T}(0)-$MTS-algebra structure $\bm{\Lambda}_V$ with 
  \[\abs{\bm{\Lambda}_V} = \hat{\sh{O}}_{R_B(X, x, r), (V, \phi)}\] 
  and the universal formal local system $\hat{V}$ underlies a free pro-$\bm{\Lambda}_V$-AVMTS $\hat{\scr{V}}$. The universal framing $\hat{\phi}: \hat{V}_x \to \hat{\sh{O}}_{R_B(X, x, r), (V, \phi)} \otimes_\CC \CC^r$ is the fiber over 1 of a morphism of pro-$\bm{\Lambda}_V$-MTS-modules $\hat{\phi}: \hat{\scr{V}}_x \to \bm{\Lambda}_V \otimes_{\sh{T}(0)} \sh{T}(0)^{\oplus r}$.\\
  \tab These structures are uniquely determined by the following universal property: For any artinian local $\sh{T}(0)-$MTS-algebra $\sh{A}$ and any $\sh{A}$-AVMTS $\scr{U}$ with a framing $\psi: \scr{U}_x \to \sh{A} \otimes_{\sh{A}} \sh{T}(0)^{\oplus r}$ which restricts to $(\scr{V}, \phi)$ mod $\Ide{m}_{\sh{A}}$, there is an unique morphism $\bm{\Lambda}_V \to \sh{A}$ of local pro-$\sh{T}(0)$-MTS-algebras such that 
  \[\brac*{\sh{A}, (\scr{U}, \psi)} \simeq \sh{A} \otimes_{\bm{\Lambda}_V} \brac*{\bm{\Lambda}_V, (\hat{\scr{V}}, \hat{\phi})}\]
  or, in other words, $\bm{\Lambda}_V$ represents the functor 
  \[\sh{A} \mapsto \{\text{framed free $\sh{A}$-AVMTS deformations of $(\scr{V}, \phi)$}\}/\simeq\] 
\end{theorem}
For each $p \in D$, choose a boundary loop $\gamma_p$ around $p$. $R_B(X, x, r)$ carries a universal representation 
\[\rho^{\Rm{univ}}: \pi_1(X^{\Rm{an}}, x) \to \Rm{GL}_r(\sh{O}(R_B))\]
and we have a universal local monodromy $T^{\Rm{univ}}_p \coloneq \rho^{\Rm{univ}}(\gamma_p) \in \Rm{GL}_r(\sh{O}(R_B))$.
\begin{definition}
  The unipotent locus is defined to be the closed subscheme
  \[R_B^{\Rm{unip}}(X, x, r) \coloneq V(\det(u - T_p^{\Rm{univ}}) - (u - 1)^r = 0)\]
\end{definition}

\subsection{The framed De Rham, Dolbeault, and Hodge moduli spaces}
\label{sec:framed-DR-Dol-Hod-spaces}
Fix $x \in X(\CC)$ again. For a $\CC-$scheme $T$, denote geometric points to be $\bar{t} \to T$. Let $\bar{X}_T \coloneq \bar{X}\times T, D_T \coloneq D \times T$ with $\pi_T: \bar{X}_T \to T$, and consider the moduli functors 
\begin{align*}
  \mathsf{R}_{\Rm{Dol}}(r)(T) &= \left\{
    \begin{array}{c|l}
      (E_T, \theta_T, \phi) & E_T \text{ locally free of rank $r$ on }\bar{X}_T,\\
      & \sh{O}_{\bar{X}_T}-\text{linear }\theta_T: E_T \to E_T \otimes \Omega^1_{\bar{X}_T/T}(\log D_T),\\
      & \theta_T \wedge \theta_T = 0,\\
      & (E_{\bar{t}}, \theta_{\bar{t}}) \text{ semistable of degree }0 \text{ on } \bar{X}\times \set*{\bar{t}},\\
      &\phi: E_T\vert_{x \times T} \simeq \sh{O}_T^r
    \end{array}
  \right\}\Big/\simeq \\
  \mathsf{R}_{\Rm{DR}}(r)(T) &= \left\{
    \begin{array}{c|l}
      (E_T, \nabla_T, \phi) & E_T \text{ locally free of rank $r$ on }\bar{X}_T,\\
      & \pi_T^{-1}\sh{O}_T-\text{linear }\nabla_T: E_T \to E_T \otimes \Omega^1_{\bar{X}_T/T}(\log D_T),\\
      & \nabla_T^2 = 0, \tab \nabla_T(fs) = \der_{\bar{X}_T/T} f \otimes s + f \nabla_T s,\\
      & (E_{\bar{t}}, \nabla_{\bar{t}}) \text{ semistable of degree }0\text{ on } \bar{X} \times \set{\bar{t}},\\
      &\phi: E_T\vert_{x \times T} \simeq \sh{O}_T^r
    \end{array}
  \right\}\Big/\simeq \\
  \mathsf{R}_{\Rm{Hod}}(r)(T) &= \left\{
    \begin{array}{c|l}
      (\lambda, E_{T}, \nabla_{\lambda,T}, \phi) & E_T \text{ locally free of rank $r$ on }\bar{X}_T, \tab \lambda \in \Gamma(T, \sh{O}_T),\\
      & \pi_T^{-1}\sh{O}_T-\text{linear }\nabla_{\lambda, T}: E_T \to E_T \otimes \Omega^1_{\bar{X}_T/T}(\log D_T),\\
      & \nabla_{\lambda, T}^2 = 0, \tab \nabla_{\lambda, T}(fs) = \lambda \der_{\bar{X}_T/T} f \otimes s + f \nabla_{\lambda, T} s\\
      &\phi: E_T\vert_{x \times T} \simeq \sh{O}_T^r
    \end{array}
  \right\}\Big/\simeq 
\end{align*}
\begin{defprop}
  \leavevmode
  \begin{enumerate}
    \item $\mathsf{R}_{\Rm{Dol}}(r)$ is representable by a quasiprojective scheme $R_{\Rm{Dol}}(\bar{X}, D, x, r)$ parametrizing degree 0 semistable logarithmic Higgs bundles with a framing. 
    \item $\mathsf{R}_{\Rm{DR}}(r)$ is representable by a quasiprojective scheme $R_{DR}(\bar{X}, D, x, r)$ parametrizing degree 0 semistable logarithmic connections with a framing. 
    \item $\mathsf{R}_{\Rm{Hod}}(r)$ is representable by a separated, locally of finite type over $\AAA^1_\lambda$, algebraic space $R_{\Rm{Hod}}(\bar{X}, D, x, r) \xrightarrow{\lambda} \AAA^1_\lambda$ parametrizing logarithmic $\lambda-$connections with a framing. 
  \end{enumerate}
  With no risk of confusion, we will simplify notation $R_{\Box}(r) \coloneq R_{\Box}(\bar{X}, D, x, r)$ for $\Box \in \set*{\Rm{Dol}, \Rm{Hod}, \Rm{DR}}$. 
\end{defprop}
\begin{proof}
  For (1), see \cite[Theorem 7.2]{bakker2024linear} and \cite[Theorem 4.10]{simpson1994moduli}. For (2), see \cite[Theorem 3.3]{nitsure1993moduli}. Finally, for (3), see \cite[Proposition 2.1]{simpson2022twistor}
\end{proof}
\begin{remark}
  The restriction to $\lambda = 1, 0$ of $R_{\Rm{Hod}}(\bar{X}, D, x, r)$ gives all logarithmic connections and all logarithmic Higgs bundles, respectively. These are bigger than $R_\DR(r)$ and $R_\Dol(r)$. However, we do have 
  \[\brac*{R_\Hod(r)\vert_{\lambda = 0}}^{\Rm{semistable}, \deg = 0} = R_\Dol(r), \tab \brac*{R_\Hod(r)\vert_{\lambda = 1}}^{\Rm{semistable}, \deg = 0} = R_\DR(r) \tag{3.2}\label{eq:restriction-of-Hod}\]
\end{remark}
\begin{remark}
  Notice that we have an algebraic $\Rm{GL}_r$ action on all these moduli spaces. For a $T$-flat bundle $(E_T, \nabla_T, \phi)$, $g \in \Rm{GL}_r(T)$ acts by $g \cdot (E_T, \nabla_T, \phi) \coloneq (E_T, \nabla_T, g \circ \phi)$. The action Dolbeault and Hodge spaces are similarly defined, and the action on the Betti space is via conjugation.
\end{remark}

\begin{definition}
  Similar to the Betti case, we can define the unipotent loci as closed sub-schemes/algebraic space
  \[R^{\Rm{nilp}}_\Box(r) \subset R_{\Box}(\bar{X}, D, x, r)\]
  for $\Box \in \set*{\Rm{Dol}, \Rm{Hod}, \Rm{DR}}$. Furthermore, we have restrictions
  \[R^{\Rm{nilp}}_{\Rm{Hod}}(r)\vert_{\lambda = 1} = R^{\Rm{nilp}}_{\Rm{DR}}(r), \tab \brac*{R^\nilp_\Hod(r)\vert_{\lambda = 0}}^{\text{semistable}, \deg = 0} = R^\nilp_\Dol(r)\]
\end{definition}

\subsection{The nilpotent Deligne-Hitchin germ and preferred sections}
Let $R^{\nilp, \Rm{loc}}_{\Hod}(\bar{X}, D, x, r)$ be the germ of an open neighborhood of $R^\nilp_\Hod(\bar{X}, D, x, r)^{\Rm{an}}$ in $R_\Hod(\bar{X}, D, x, r)^{\Rm{an}}$. Consider the complex conjugate variety $(\bar{X}^c, D^c)$, we can similarly define $R^{\nilp, \Rm{loc}}_{\Hod}(\bar{X}^c, D^c, x, r)$ and we have a holomorphic isomorphism 
\[R^{\nilp, \Rm{loc}}_{\Hod}(\bar{X}, D, x, r) \vert_{\GG_m} \simeq R^{\nilp, \Rm{loc}}_{\Hod}(\bar{X}^c, D^c, x, r) \vert_{\GG_m}\]
and there is a gluing procedure to get the nilpotent Deligne-Hitchin germ $R^{\nilp, \Rm{loc}}_\Dh(r) \to \PP^1$, which is a countably finite-type complex analytic space. There is a functor of points description, and $R^{\nilp, \Rm{loc}}_\Dh(r)\to \PP^1$ parametrizes exactly the extensions in \cref{def:AVMTS-extension}, after forgetting the weight $W_\bullet$ and adding a framing. Thus a $\CC-$AVMTS $\scr{V} = j^*\scr{T}$ with unipotent local monodromy gives an analytic morphism $\PP^1 \to R^{\nilp, \Rm{loc}}_\Dh(\bar{X}, D, x, r)$. \\
\tab If we add an artinian local $\sh{A}$-action on $\scr{V}$, then the extended actions on $\bar{\scr{V}}, \bar{\scr{V}}'$ turns them into vector bundles on $\bar{X} \times \AAA^1_\lambda \times \Spec \abs{\sh{A}}$ and $\bar{X} \times \AAA^1_{\lambda^{-1}} \times \Spec \abs{\sh{A}}$, respectively. These assembles into a $\Spec_{\PP^1}\sh{A}$-point of $R^{\nilp, \Rm{loc}}_\Dh(r) \to \PP^1$, hence a $\PP^1$-morphism $\Spec_{\PP^1}\sh{A} \to R^{\nilp, \Rm{loc}}_\Dh(r)$. For more details, see \cite[Section 8.2 and Corollary 8.8]{bakker2024linear}. 

\begin{definition}
  If $\scr{V} = j^*\scr{T}$ is a $\CC$-AVMTS, and $\scr{V}$ is $\CC-$VTS on $X$, then the resulting map $\PP^1 \to R^{\nilp, \Rm{loc}}_\Dh(\bar{X}, D, x, r)$ is called a preferred section. 
\end{definition}
More concretely, if $(V, \phi) \in R^\unip_B(X, x, r)^{\Rm{ss}}$ is a semisimple local system with unipotent local monodromy, then it underlies a $\CC$-VTS with unipotent local monodromy, giving a preferred section 
\[s_{(V, \phi)}: \PP^1 \to R^{\nilp, \Rm{loc}}_\Dh(\bar{X}, D, x, r)\]
with $s_1 \coloneq s_{(V, \phi)}(1)$ being the framed Deligne's canonical extension of $(V, \phi)$ to a logarithmic flat bundle, and $s_0 \coloneq s_{(V, \phi)}(0) = \mathsf{SM}^{\Rm{ps}}(s_1)$, i.e., the logarithmic polystable Higgs bundle which we get via the Mochizuki's correspondence in \cref{thm:mochizuki-correspondence}. 
\begin{prop}
  \label{thm:Bett-DR-Dol-formal-isom}
  Recall the pro-$\sh{T}(0)-$MTS-algebra $\bm{\Lambda}_V$ in \cref{thm:Betti-MTS-algebra}. 
  \begin{enumerate}
    \item \cite[Lemma 8.11]{bakker2024linear} We have a classifying map for the canonical logarithmic extension of the universal pro-$\bm{\Lambda}_V$-AVMTS $\hat{\scr{V}}$:
    \[\hat{s}_{(V, \phi)}: \Rm{Spf}_{\PP^1} \bm{\Lambda}_V \xrightarrow{\simeq} \widehat{R^{\nilp, \Rm{loc}}_\Dh(r)}_{s_{(V, \phi)}}\]
    which is an isomorphism onto the completion of $R^{\nilp, \Rm{loc}}_\Dh(r)$ along the preferred section. 
    \item Specializing at $\lambda = 1, 0$ we get 
    \[\Rm{Spf}_{\PP^1} \bm{\Lambda}_V \vert_{\lambda = 1} \xrightarrow{\simeq} \widehat{R_\DR(r)}_{s_1}, \tab \Rm{Spf}_{\PP^1} \bm{\Lambda}_V \vert_{\lambda = 0} \xrightarrow{\simeq} \widehat{R_\Dol(r)}_{s_0}\]
    where the targets are completions at $s_1$ and $s_0$ respectively. As a consequence of \cref{lem:affinesplitting}, we get a trivialization $\varphi_{\bm{\Lambda}_V}$ of $\bm{\Lambda}_V$, and  
    \[\tau_V \coloneq (\hat{s}\vert_{\lambda = 0})\circ \varphi_{\bm{\Lambda}_V, 0} \circ \varphi_{\bm{\Lambda}_V, 1}^{-1}\circ (\hat{s}\vert_{\lambda = 1})^{-1}: \widehat{R_\DR(r)}_{s_1} \xrightarrow{\simeq} \widehat{R_\Dol(r)}_{s_0} \tag{3.3}\label{eq:formal-DR-Dol-immersion}\]
    \item Let $H \subset \Rm{GL}_r$ be an algebraic group fixing $[(V, \phi)] \in R_B(X, x, r)$. Then $H$ also fixes $[s_0] \in R_{\Dol}(r), [s_1] \in R_\DR(r)$, and induces $H^\DR$ and $H^\Dol$ actions on $\widehat{R_\DR(r)}_{s_1}$ and $\widehat{R_\Dol(r)}_{s_0}$ respectively. The isomorphism $\tau_V$ \eqref{eq:formal-DR-Dol-immersion}, as well as its inverse $\tau_V^{-1}$ are equivariant. 
  \end{enumerate}
\end{prop}

\begin{proof}
  Recall that 
  \[\brac*{R_\Hod(r)\vert_{\lambda = 0}}^{\Rm{semistable}, \deg = 0} = R_\Dol(r), \tab \brac*{R_\Hod(r)\vert_{\lambda = 1}}^{\Rm{semistable}, \deg = 0} = R_\DR(r)\]
  and $s_0, s_1$ are polystable bundles of degree 0. Since semistability and degree 0 are open conditions, we get the following identifications of completions
  \[\widehat{R^{\nilp, \Rm{loc}}_\Dh(r)}_{s_{(V, \phi)}} \vert_{\lambda = 1} \simeq \widehat{R_\DR(r)}_{s_1}, \tab \widehat{R^{\nilp, \Rm{loc}}_\Dh(r)}_{s_{(V, \phi)}} \vert_{\lambda = 0} \simeq \widehat{R_\Dol(r)}_{s_0}\]
  giving (2). \\
  \tab For (3), let's first prove that $H$ fixes $[s_0]$ and $[s_1]$. Let $(E, \nabla, \phi)$ be the framed flat bundle on $X$ associated to $(V, \phi)$, $V$ is the sheaf of flat sections of $(E, \nabla)$ and we can canonically identify the fibers $V_x \simeq E_x$ thus justifying the same framing $\phi$. Furthermore if $(\bar{E}, \nabla)$ is the Deligne extension then the framing at $x \notin D$ is the same $\phi$ (so $s_1 = (\bar{E}, \nabla, \phi)$). Suppose that $g \in H$ fixes $[(V, \phi)]$, then $g_x \coloneq \phi^{-1} \circ g \circ \phi: V_x \to V_x$ commutes with all monodromy operators thus extends to an automorphism $a_g: V \xrightarrow{\simeq} V$ satisfying $\phi \circ a_g = g \circ \phi$. Now, $g_x$ commutes with local monodromy hence it preserves the Deligne's lattice, thus $a_g$ extends to an automorphism $\bar{a}_g: \bar{E} \xrightarrow{} \bar{E}$. And since the framing is the same, $\bar{a}_g: (\bar{E}, \nabla, \phi) \xrightarrow{\simeq} (\bar{E}, \nabla, g \circ \phi)$, i.e., $g$ fixes $[s_1]$. \\
  \tab Via \cref{thm:mochizuki-correspondence}, we get an automorphism $\mathsf{SM}^{\Rm{ps}}(\bar{a}_g)$ of $\mathsf{SM}^{\Rm{ps}}(\bar{E}, \nabla)$. The functor is defined via a tame nilpotent harmonic bundle $\sh{E}$ so we have $\bar{E}_x = \sh{E}_x = \mathsf{SM}^{\Rm{ps}}(\bar{E}, \nabla)_x$. It follows that $g$ fixes $[s_0] = [(\mathsf{SM}^{\Rm{ps}}(\bar{E}, \nabla), \phi)]$. \\
  \tab Since $(V, \phi)$ is semisimple, we get an automorphism $a_g$ of the weight 0 $\CC$-VTS $\scr{V}$ as well, thus the change of framing $(\scr{U}, \psi) \mapsto (\scr{U}, g \circ \psi)$ gives an automorphism of the functor of framed free AVMTS deformations of $(\scr{V}, \phi)$. By \cref{thm:Betti-MTS-algebra}, we get an automorphism $\tilde{a}_g: \bm{\Lambda}_V \simeq \bm{\Lambda}_V$ of pro-$\sh{T}(0)$-MTS-algebras. We have a commutative diagram 
  \[\begin{tikzcd}
	{\Rm{Spf}_{\PP^1}\bm{\Lambda}_V} && {\Rm{Spf}_{\PP^1}\bm{\Lambda}_V} \\
	\\
	{R_\Dh^{\nilp, \Rm{loc}}(r)} && {R_\Dh^{\nilp, \Rm{loc}}(r)}
	\arrow["{c_g}", from=1-1, to=1-3]
	\arrow["{\hat{s}}"', from=1-1, to=3-1]
	\arrow["{\hat{s}}", from=1-3, to=3-3]
	\arrow["g", from=3-1, to=3-3]
  \end{tikzcd}\]
  since change of framing commutes with logarithmic extension. By functoriality of \cref{lem:affinesplitting}, we have 
  \[\begin{tikzcd}
	{\Rm{Spf}_{\AAA^1}\brac*{\bm{\Lambda}_V \vert_{\AAA^1_\lambda}}} && {\Rm{Spf}_{\AAA^1}\brac*{\bm{\Lambda}_V \vert_{\AAA^1_\lambda}}} \\
	\\
	{\Rm{Spf}_{\AAA^1} \brac*{\hat{O}_{R_B, (V, \phi)} \otimes \sh{O}_{\AAA^1_\lambda}}} && {\Rm{Spf}_{\AAA^1} \brac*{\hat{O}_{R_B, (V, \phi)} \otimes \sh{O}_{\AAA^1_\lambda}}}
	\arrow["{c_g}", from=1-1, to=1-3]
	\arrow["{\varphi_{\bm{\Lambda}_V}}", from=3-1, to=1-1]
	\arrow["g", from=3-1, to=3-3]
	\arrow["{\varphi_{\bm{\Lambda}_V}}"', from=3-3, to=1-3]
\end{tikzcd}\]
  and together they give the desired equivariance of $\tau_V$ and $\tau_V^{-1}$. 
\end{proof}

\subsection{Moduli spaces of triples}
\label{sec:moduli-triples}
We give a slightly generalized exposition of the ideas in \cite[Theorem 3.3-3.5]{nitsure2005constructionhilbertquotschemes}. Let $B$ be a $\CC$-scheme.
\begin{definition}
  A complex $K^\bullet \in D^{\geq 0}(\sh{O}_B)$ is perfect if, affine-locally on $U \subseteq B$, $K^\bullet$ is quasi-isomorphic to a bounded complex of finite rank locally free $\sh{O}_U$-modules. We denote $\sh{H}^0(K^\bullet) \in \Rm{QCoh}(B)$ to be the zeroth cohomology object of $K^\bullet$. 
\end{definition}
\begin{definition}
  For $K^\bullet \in D^{\geq 0}(\sh{O}_B)$ a perfect complex, denote 
  \[M^0(K^\bullet) \coloneq \Spec_B \Rm{Sym}_{\sh{O}_B}(\sh{H}^0((K^\bullet)^\vee))\] 
  where $(K^\bullet)^\vee \coloneq \Rm{R}\shHom_{\sh{O}_B}(K^\bullet, \sh{O}_B)$. 
\end{definition}

\begin{lemma}[{\cite[\href{https://stacks.math.columbia.edu/tag/08JX}{Tag 08JX}]{stacks-project}$+\epsilon$}]
  \label{lem:relative-hom-scheme-lemma}
  Let $f: T \to B$ be a morphism of schemes and $K^\bullet \in D^{\geq 0}(\sh{O}_B)$ a perfect complex, then we have functorial bijections
  \[\Hom_B(T, M^0(K^\bullet)) \simeq \HH^0(T, \Rm{L}f^*K^\bullet)\]
\end{lemma}
\begin{proof}
  By the proof of \cite[\href{https://stacks.math.columbia.edu/tag/08JX}{Tag 08JX}]{stacks-project}, $(T \xrightarrow{f} B) \mapsto \HH^0(T, \Rm{L}f^*K^\bullet)$ is a fppf sheaf. On the other hand, $(T \xrightarrow{f} B) \mapsto \Hom_B(T, M^0(K))$ is clearly a fpqc sheaf (representable by $M^0(K)$), so by \cite[\href{https://stacks.math.columbia.edu/tag/04U0}{Tag 04U0}]{stacks-project}, we can reduce to the case where $B = \Spec R$ and $K^\bullet$ is quasi-isomorphic to $(P^\bullet, \der^\bullet)$ a complex of projective $R-$modules with no negative term. Then 
  \[\sh{H}^0((K^\bullet)^\vee) = \coker\brac*{P^{1, \vee} \xrightarrow{\der^{0, \vee}} P^{0, \vee}}\]
  and 
  \begin{align*}
    \Hom_B(T, M^0(K^\bullet)) &\simeq \Hom_{\sh{O}_B-\Rm{alg}}\brac*{\Rm{Sym}_{\sh{O}_B} \brac*{\coker\brac*{P^{1, \vee} \xrightarrow{\der^{0, \vee}} P^{0, \vee}}}, f_*\sh{O}_T}\\
    &\simeq \Hom_{\sh{O}_T-\Rm{alg}}\brac*{f^*\Rm{Sym}_{\sh{O}_B} \brac*{\coker\brac*{P^{1, \vee} \xrightarrow{\der^{0, \vee}} P^{0, \vee}}}, \sh{O}_T}\\
    &\simeq \Hom_{\sh{O}_T-\Rm{alg}}\brac*{\Rm{Sym}_{\sh{O}_B} \brac*{f^*\coker\brac*{P^{1, \vee} \xrightarrow{\der^{0, \vee}} P^{0, \vee}}}, \sh{O}_T}\\
    &\simeq \Hom_{\sh{O}_T} \brac*{f^*\coker\brac*{P^{1, \vee} \xrightarrow{\der^{0, \vee}} P^{0, \vee}}, \sh{O}_T}\\
    &\simeq \Hom_{\sh{O}_T} \brac*{\coker\brac*{f^*P^{1, \vee} \xrightarrow{\der^{0, \vee}} f^*P^{0, \vee}}, \sh{O}_T}\\
    &\simeq H^0\brac*{T, \ker\brac*{f^*P^0 \xrightarrow{f^*\der^0} f^*P^1}}\\
    &\simeq \ker\brac*{H^0(T, f^*P^0) \xrightarrow{H^0(f^*\der^0)} H^0(T, f^*P^1)} \simeq \HH^0(T, \Rm{L}f^* K^\bullet)
  \end{align*}
  where all these isomorphisms are functorial in $T \to B$. 
\end{proof}
\begin{definition}
  For $\pi: M^0_B(H^\bullet) \to B$, the universal section is $u \in \HH^0(M^0_B(K^\bullet), \Rm{L}\pi^*K^\bullet)$ corresponding to the identity in $\Hom_B(M^0_B(K^\bullet), M^0_B(K^\bullet))$.
\end{definition}
\begin{coroll}
  \label{lem:relative-hom-and-restriction}
  For $\phi: B' \to B$, there is a canonical isomorphism 
  \[M^0_B(K^\bullet) \times_B B' \xrightarrow{\simeq} M^0_{B'}(\Rm{L}\phi^* K^\bullet)\]
  which is compatible with universal sections, i.e., for $\pi': M^0_{B'}(\Rm{L}\phi^* K^\bullet) \to B', \pi: M^0_B(K^\bullet) \to B$, and universal sections $u' \in \HH^0\brac*{M^0_{B'}(\Rm{L}\phi^* K^\bullet), \Rm{L}(\phi \circ \pi)^* K^\bullet}, u \in \HH^0\brac*{M^0_B(K^\bullet), \Rm{L}\pi^*K^\bullet}$, we have $u' = \Rm{L}\phi^*(u)$. 
\end{coroll}
\begin{proof}
  For $T \xrightarrow{f'} B'$, denote $f: T \xrightarrow{f'} B' \xrightarrow{\phi} B$. We have 
  \[\Hom_{B'}(T, M^0_B(K^\bullet) \times_B B') \simeq \Hom_B(T, M_B^0(K^\bullet)) \simeq \HH^0(T, \Rm{L}f^*K^\bullet)\]
  \tab On the other hand, 
  \[\Hom_{B'}(T, M^0_{B'}(\Rm{L}\phi^* K^\bullet)) \simeq \HH^0\brac*{T, \Rm{L}(f')^*(\Rm{L}\phi^* K^\bullet)} = \HH^0\brac*{T, \Rm{L}f^*K^\bullet}\]
  so by Yoneda we get both statements. 
\end{proof}
\begin{coroll}
  \label{lem:MTS-enhancement-relative-hom-representability}
  Let $\sh{A}$ be an artinian $\sh{T}(0)$-MTS-algebra, and $\sh{K} \in D^{\geq 0}(\sh{A}\Rm{-MTS-module})$ is perfect. Write $\sh{K}^\vee \coloneq \Rm{R}\ul{\Hom}_{\sh{A}}(\sh{K}, \sh{A})$. Then for every morphism of artinian local $\sh{T}(0)$-MTS-algebras $\sh{A} \to \sh{B}$ we have a bijection 
  \[\Hom_{\sh{A}\Rm{-alg}}(\Rm{Sym}_{\sh{A}} \sh{H}^0(\sh{K}^\vee), \sh{B}) \simeq \Hom_{D(\sh{B}\Rm{-MTS-module})}(\sh{B}, \sh{B} \otimes_{\sh{A}}^{\Rm{L}} \sh{K})\]
  which is functorial in $\sh{B}$. 
\end{coroll}
\begin{proof}
  First,
  \[\Hom_{\sh{A}\Rm{-alg}}(\Rm{Sym}_{\sh{A}} \sh{H}^0(\sh{K}^\vee), \sh{B}) \simeq \Hom_{\sh{A}\Rm{-MTS-module}}(\sh{H}^0(\sh{K}^\vee), \sh{B})\]
  and since $\sh{K}^\vee \in D^{\leq 0}$, we have by degree reason
  \[\Hom_{D(\sh{A}\Rm{-MTS-module})}(\sh{K}^\vee, \sh{B}) \simeq \Hom_{\sh{A}\Rm{-MTS-module}}(\sh{H}^0(\sh{K}^\vee), \sh{B})\]
  \tab Finally, by push-pull adjunction and duality we get 
  \begin{align*}
    \Hom_{\sh{A}\Rm{-alg}}(\Rm{Sym}_{\sh{A}} \sh{H}^0(\sh{K}^\vee), \sh{B}) &\simeq \Hom_{D(\sh{A}\Rm{-MTS-module})}(\sh{K}^\vee, \sh{B})\\
    &\simeq \Hom_{D(\sh{B}\Rm{-MTS-module})}(\sh{B} \otimes_{\sh{A}}^{\Rm{L}}\sh{K}^\vee, \sh{B})\\
    &\simeq \Hom_{D(\sh{B}\Rm{-MTS-module})}((\sh{B} \otimes_{\sh{A}}^{\Rm{L}}\sh{K})^\vee, \sh{B})\\
    &\simeq \Hom_{D(\sh{B}\Rm{-MTS-module})}(\sh{B}, \sh{B} \otimes_{\sh{A}}^{\Rm{L}} \sh{K})
  \end{align*}
  as desired. 
\end{proof}
Let $b: \Spec \CC \to B$ be a closed point, then a $\CC$-point of $M^0(K^\bullet)$ lying over $b$ is the same as a $\CC-$linear map 
\[\sh{H}^0((K^\bullet)^\vee) \otimes_{\sh{O}_B} \CC_b \to \CC\]
which, since $(K^\bullet)^\vee \in D^{\leq 0}$, is an element of 
\begin{align*}
  H^0((K^\bullet)^\vee \otimes^{\Rm{L}}_{\sh{O}_B} \CC_b)^\vee &\simeq H^0((K^\bullet \otimes^{\Rm{L}}_{\sh{O}_B} \CC_b)^\vee)^\vee \\
  &\simeq H^0(K^\bullet \otimes_{\sh{O}_B}^{\Rm{L}} \CC_b) \simeq \Hom_{D(\CC)}(\CC, K^\bullet \otimes_{\sh{O}_B}^{\Rm{L}} \CC_b)
\end{align*}
where the first isomorphism comes from \cite[\href{https://stacks.math.columbia.edu/tag/0A6A}{Tag 0A6A}]{stacks-project}. 

\begin{definition}
  Suppose that $\bm{\Lambda} = \varprojlim \bm{\Lambda}_n$ is a pro-$\sh{T}(0)$-MTS-algebra, with $\Lambda = \abs*{\bm{\Lambda}}$. Let $B = \Rm{Spf}\ \Lambda$, $\sh{K} = \set*{\sh{K}_n} \in D^{\geq 0}(\bm{\Lambda}_n\text{-MTS-module})$. Then, for $K \coloneq \abs*{\sh{K}}$, one can form a formal scheme $M^0_B(K) \to B$. By the previous discussion, a $\CC$-point of $M^0_B(K)$ lying over $\Spec \CC \xrightarrow{b} B$ is the same as a morphism $\CC \to K \otimes^{\Rm{L}}_{\Lambda} \CC$. We say that this is a Tate point if it underlies a morphism in $D(\CC-\text{MTS})$
  \[\sh{T}(0) \to \sh{K} \otimes^{\Rm{L}}_{\bm{\Lambda}} \sh{T}(0)\]
\end{definition}
\begin{prop}
  \label{lem:MTS-structure-relative-Hom-scheme}
  In the setting above, if $s$ is a Tate point of $M^0_B(K) \to B = \Rm{Spf} \Lambda$, then the completed local ring $\hat{\sh{O}}_{M_B^0(K), s}$ underlies a pro-$\bm{\Lambda}$-MTS-algebra $\bm{\Gamma}$. This construction commutes with pro-$\sh{T}(0)$-MTS-algebra coefficient base change, and fiber products over a pro-$\sh{T}(0)$-MTS-algebra base. 
\end{prop}
\begin{proof}
  We can work levelwise, and notice that the coordinate algebra of $M^0_{\Spec \Lambda_n}(K_n)$, which is $\Rm{Sym}_{\Lambda_n} \sh{H}^0(K_n^\vee)$, underlies $\Rm{Sym}_{\bm{\Lambda}_n} \sh{H}^0(\sh{K}_n^\vee)$ which is a $\bm{\Lambda}_n$-MTS-algebra (the differentials of $\sh{K}^\vee$ are MTS-morphisms, and the category is abelian). A Tate point gives a morphism of $\bm{\Lambda}_n-$MTS-algebras $\Rm{Sym}_{\bm{\Lambda}_n} \sh{H}^0(\sh{K}_n^\vee) \to \sh{T}(0)$, hence is defined by a MTS-ideal (equivalently, the quotient still has a MTS structure). \\
  \tab Coefficient base change follows from being bounded above by 0. Indeed, by the argument in \cref{lem:relative-hom-scheme-lemma}, we can pick a locally free quasi-isomorphic $P^\bullet$, where 
  \[\sh{H}^0(\sh{K}_n^\vee) = \coker(P^{1, \vee} \to P^{0, \vee})\]
  hence given $\bm{\Lambda}_n \to \sh{A}$, using right exactness of tensor product we get
  \[\sh{H}^0(\sh{K}_n^\vee) \otimes_{\bm{\Lambda}_n} \sh{A} \simeq \sh{H}^0((\sh{K}_n \otimes_{\bm{\Lambda}_n}^{\Rm{L}} \sh{A}))\]
  
  Finally, notice that due to being bounded above by 0, $\sh{H}^0(\sh{K}_n^\vee \oplus \sh{L}_n^\vee) \simeq \sh{H}^0(\sh{K}_n^\vee) \oplus \sh{H}^0(\sh{L}_n^\vee)$ and we have
  \[M^0(K_n) \times_{\Spec \Lambda_n} M^0(L_n) \simeq M^0(K_n \oplus L_n)\]
  so compatibility with fiber products follows. 
\end{proof}
\begin{remark}
  In this setting where $\sh{K}_n \in D^{\geq 0}(\bm{\Lambda}_n-\text{MTS-module})$, restriction is exact, so 
  \[M_{\Lambda_n}^0(K_n) = M^0_{\bm{\Lambda}_n}(\sh{K}_n)\vert_{\lambda = 1}\]
  by \cref{lem:relative-hom-and-restriction}. The universal section on $M_{\Lambda_n}^0(K_n)$ is the restriction of the universal section on $M^0_{\bm{\Lambda}_n}(\sh{K}_n)$. 
\end{remark}
For $\Box \in \set*{\DR, \Dol}$, denote $B = R_\Box^\pair(r, s) \coloneq R_\Box(r) \times R_\Box(s)$ the moduli of pairs. Let $\bar{X}_B \coloneq \bar{X} \times B$, $D_B \coloneq D \times B$ with projection $\pi_B: \bar{X}_B \to B$, and $E^\univ_\Box, F^\univ_\Box$ be the universal bundles. We can form the complexes
\begin{align*}
  C_{\DR}(E^\univ_\DR, F^\univ_\DR) &\coloneq \sqbrac*{\shHom(E^\univ_\DR, F^\univ_\DR) \xrightarrow{\nabla} \shHom(E^\univ_\DR, F^\univ_\DR) \otimes \Omega^1_{\bar{X} \times B/B}(\log D_B)}\\
  C_{\Dol}(E^\univ_\Dol, F^\univ_\Dol) &\coloneq \sqbrac*{\shHom(E^\univ_\Dol, F^\univ_\Dol) \xrightarrow{\theta} \shHom(E^\univ_\Dol, F^\univ_\Dol) \otimes \Omega^1_{\bar{X} \times B/B}(\log D_B)}\\
  K_\Box(E^\univ_\Box, F^\univ_\Box) &\coloneq \Rm{R}\pi_{B, *} C_\Box(E^\univ_\Box, F^\univ_\Box) \tag{3.4}\label{eq:Dol/DR-complexes}
\end{align*}
\begin{prop}
  \label{lem:base-change-perfectness-Dol-DR-complexes}
  $K_\Box(E^\univ_\Box, F^\univ_\Box)$ is perfect in $D(\sh{O}_B)$ with no negative cohomology sheaf. For any $f: T \to B$, we have corresponding families $E_{\Box, T}, F_{\Box, T}$ on $\bar{X} \times T$ and we can form similar complexes. Then the formulation of $K_\Box$ commutes with base change, i.e., 
  \[K_{\Box, T}(E_{\Box, T}, F_{\Box, T}) = \Rm{L}f^* K_\Box(E^\univ_\Box, F^\univ_\Box)\]
  As a consequence, $K_{\Box, T}(E_{\Box, T}, F_{\Box, T})$ is perfect in $D(\sh{O}_T)$ and base change works for $T' \to T$ over $B$ as well. 
\end{prop}
\begin{proof}
  For the Dolbeault complex, this follows from \cite[\href{https://stacks.math.columbia.edu/tag/0A1G}{Tag 0A1G}]{stacks-project} and that relative logarithmic Kahler differentials commutes with base change. For the De Rham complex, the same proof as in \cite[\href{https://stacks.math.columbia.edu/tag/0FM0}{Tag 0FM0}]{stacks-project} works. 
\end{proof}

\begin{prop}
  \label{lem:moduli-of-triples-and-deformation}
  The schemes 
  \[R^\trip_\Box(r, s) \coloneq M^0(K_\Box^\univ) \to R_\Box^\pair(r, s)\]
  represents the functor of framed triples $(E, F, f)$ where $E, F$ are semistable degree 0 logarithmic flat connections (resp. logarithmic Higgs bundle), and $f$ is a flat/Higgs morphism. If $(E, F, f)$ is a closed point, and $A$ is artinian local, then we have the framed deformation functor
  \[\Rm{Def}^\trip_{\Box, (E, F, f)} = \left\{
    \begin{array}{c|l}
      (E_A, F_A, f_A) & (E_A, F_A) \in \Rm{Def}^\pair_{\Box, (E, F)}(A) \\
      & f_A \in \HH^0(K_{\Box, A}(E_A, F_A)), \tab f_A \ \Rm{mod}\ \Ide{m}_A = f
    \end{array}
  \right\}\Big/\simeq \]
\end{prop}
\begin{proof}
  Let $B = R_\Box^\pair(r, s)$. By \cref{lem:relative-hom-scheme-lemma}, for $T \xrightarrow{f} B$,  
  \begin{align*}
    \Hom_B(T, R^\trip_\Box(r, s)) &\simeq \HH^0(\Rm{L}f^* K_\Box^\univ)\\
    &\simeq \HH^0\brac*{K_{\Box, T}(E_{\Box, T}, F_{\Box, T})}\\
    &\simeq \HH^0\brac*{C_{\Box, T}(E_{\Box, T}, F_{\Box, T})}
  \end{align*}
  which gives the flat/Higgs morphisms between $E_{\Box, T}$ and $F_{\Box, T}$. 
\end{proof}
We also have the Betti moduli space of triples $R^{\trip}_{\Rm{B}}(r, s)$, where instead of a universal logarithmic DR/Dol complex, we work with $K^\bullet$ being just the locally constant sheaf $\ul{\Hom}(E^{\univ}, F^{\univ})$ with $E^{\univ}, F^{\univ}$ being the universal framed local systems pulled back to $X \times R_{\Rm{B}}(r) \times R_{\Rm{B}}(s)$.  

\begin{prop}
  \label{thm:universal-MTS-framed-triple}
  Let $(E, \phi_E), (F, \phi_F)$ be framed semisimple complex local systems of rank $r, s$ on $X$, with unipotent local monodromy, and $f: E \to F$ a morphism of local systems. Fix the canonical weight 0 $\CC$-VTS $\scr{E}, \scr{F}$; $f$ then extends to a morphism of $\CC$-VTS. Denote 
  \[\Gamma_{E, F, f} \coloneq \widehat{\sh{O}}_{R^{\trip}_{\Rm{B}}(r, s), (E, F, f)},\tab \Lambda_{E, F} \coloneq \widehat{\sh{O}}_{R_{\Rm{B}}^{\pair}(r, s), (E, F)} \simeq \widehat{\sh{O}}_{R_{\Rm{B}}(r), E} \widehat{\otimes} \widehat{\sh{O}}_{R_{\Rm{B}}(s), F}\]
  then we have a pro-$\sh{T}(0)$-MTS-algebra $\bm{\Lambda}_{E, F}$ on $\Lambda_{E, F}$ (\cref{thm:Betti-MTS-algebra}) and the following are true: 
  \begin{enumerate}
    \item There is a pro-$\sh{T}(0)$-MTS-algebra $\bm{\Gamma}_{E, F, f}$ with $\abs*{\bm{\Gamma}_{E, F, f}} = \Gamma_{E, F, f}$, and $\bm{\Lambda}_{E, F} \to \bm{\Gamma}_{E, F, f}$ is a morphism of pro-$\sh{T}(0)-$MTS-algebras. 
    \item There is a universal triple $(\widehat{\scr{E}}, \widehat{\scr{F}}, \widehat{f})$ where $\widehat{\scr{E}}, \widehat{\scr{F}}$ are free pro-$\bm{\Gamma}_{E, F, f}$-AVMTS and $\widehat{f}$ a morphism of $\bm{\Gamma}_{E, F, f}$-AVMTS, such that for every artinian local $\sh{T}(0)$-MTS-algebra $\sh{A}$, pullback gives a functorial bijection
    \[\Hom_{\sh{T}(0)-\Rm{alg}}(\bm{\Gamma}_{E, F, f}, \sh{A}) \xrightarrow{\simeq} \left\{ 
      \begin{array}{c}
        \text{framed free }\sh{A}-\text{AVMTS triples }(\scr{E}_{\sh{A}}, \scr{F}_{\sh{A}}, f_{\sh{A}}) \\
        \text{with }f_{\sh{A}} \text{ a morphism of }\sh{A}-\text{AVMTS}\\
        (\scr{E}, \scr{F}, f) \simeq (\scr{E}_{\sh{A}}, \scr{F}_{\sh{A}}, f_{\sh{A}})\ \Rm{mod}\ \Ide{m}_{\sh{A}}
      \end{array}
    \right\}\Big/ \simeq\]
    \item Let $\bar{E}, \bar{F}$ be the Deligne extensions of $E, F$. Then we have identifications 
    \[\bm{\Gamma}_{E, F, f}\vert_{\lambda = 1} \simeq \widehat{\sh{O}}_{R^{\trip}_{\DR}, (\bar{E}, \bar{F}, f)}, \tab \bm{\Gamma}_{E, F, f}\vert_{\lambda = 0} \simeq \widehat{\sh{O}}_{R^{\trip}_{\Dol}, \Sf{SM}^{\Rm{ps}}(\bar{E}, \bar{F}, f)}\]
    and restrictions of $(\widehat{\scr{E}}, \widehat{\scr{F}}, \widehat{f})$ give corresponding DR/Dol universal families. 
  \end{enumerate}
\end{prop}
\begin{proof}
  Let $a_X: X \to \Rm{pt}$ and $q: \bar{X} \times \AAA^1_\lambda \to \AAA^1_\lambda$ be the projection. We will repeatedly make use of the following lemma 
  \begin{lemma}
    Let $\sh{A}$ be an artinian local $\sh{T}(0)$-MTS-algebra, and $\scr{E}, \scr{F}$ free $\sh{A}-$AVMTS. Then 
    \[\Hom_{\sh{A}\Rm{-AVMTS}}(\scr{E}, \scr{F}) \simeq \Hom_{D(\sh{A}\Rm{-MTS-mod})}(\sh{A}, (a_X)_*\ul{\Hom}_{\sh{A}}(\scr{E}, \scr{F})[-1]) \tag{3.5}\label{eq:adjunction-AVMTS-hom}\]
  \end{lemma}
  \begin{proof}
    Let $\scr{A} \coloneq (a_X)^*\sh{A}[1]$. Since $\scr{E}, \scr{F}$ are free we have $\ul{\Hom}_{\sh{A}}(\scr{E}, \scr{F}) \simeq \scr{E}^\vee \otimes_{\sh{A}} \scr{F}$. By tensor-hom adjunction,
    \[\Hom_{\sh{A}\Rm{-AVMTS}}(\scr{E}, \scr{F}) \simeq \Hom_{\sh{A}\Rm{-AVMTS}}(\scr{A}, \ul{\Hom}_{\sh{A}}(\scr{E}, \scr{F}))\]
    \tab We have a fully faithful embedding $\sh{A}$-AVMTS $\hookrightarrow \sh{A}\text{-}\Sf{MTM}$, and by push-pull adjunction in $\sh{A}\text{-}\Sf{MTM}$ we get 
    \begin{align*}
      \Hom_{\sh{A}\Rm{-AVMTS}}(\scr{A}, \ul{\Hom}_{\sh{A}}(\scr{E}, \scr{F})) &\simeq \Hom_{\sh{A}\text{-}\Sf{MTM}}((a_X)^*\sh{A}[1], \ul{\Hom}_{\sh{A}}(\scr{E}, \scr{F}))\\
      &\simeq \Hom_{D(\sh{A}\Rm{-MTS-mod})}(\sh{A}, (a_X)_*\ul{\Hom}_{\sh{A}}(\scr{E}, \scr{F})[-1])
    \end{align*}
    as desired. 
  \end{proof}
  Let $\bm{\Lambda}_{n} \coloneq \bm{\Lambda}_{E, F}/\Ide{m}_{\bm{\Lambda}}^{n + 1}$, and $\Lambda_n \coloneq \abs*{\bm{\Lambda}_{n}}$. From \cref{thm:Betti-MTS-algebra}, we have universal pair $(\widehat{\scr{E}}, \widehat{\scr{F}})$ on $\bm{\Lambda}_{E, F}$. Pulling back gives universal pair $(\scr{E}_n, \scr{F}_n)$ of $\bm{\Lambda}_n-$AVMTS. Let (see \cref{sec:The-comparison} for the relevant definitions)
  \[\sh{K}_n \coloneq (a_X)_* \ul{\Hom}_{\bm{\Lambda}_n}(\scr{E}_n, \scr{F}_n)[-1], \tab K_n^{\Rm{tw}} = \Rm{R}q_* \Rm{DR}_{\scr{R}}(\scr{M}_{\ul{\Hom}_{\bm{\Lambda}_n}(\scr{E}_n, \scr{F}_n)}[*D])\]
  then $\sh{K}_n \in D^{\geq 0}(\bm{\Lambda}_n-\text{MTS-module})$, with $K^{\Rm{tw}}_n$ being the restriction of $\sh{K}_n$ to the chart $\AAA^1_\lambda$. As a result, $\Rm{L}\iota_c \sh{K}_n \simeq \Rm{L}\iota_c K^{tw}_n$ for $c \in \AAA^1_\lambda$. Let $\sh{K} = \set*{\sh{K}_n}$ then we can form 
  \[M^0_{\bm{\Lambda}_{E, F}}(\sh{K})_{f} = \Rm{Spf}_{\PP^1}\ \bm{\Gamma}_{E, F, f}\]
  where $f$ is a morphism of $\CC-$VTS hence a morphism of $\CC$-AVMTS, and by \eqref{eq:adjunction-AVMTS-hom},
  \[\Hom_{\CC\Rm{-AVMTS}}(\scr{E}, \scr{F}) \simeq \Hom_{D(\CC\Rm{-MTS})}(\sh{T}(0), (a_X)_*\ul{\Hom}_{\sh{T}(0)}(\scr{E}, \scr{F})) \]
  $f$ is a Tate point. We then have that $\bm{\Gamma}_{E, F, f}$ is a pro$-\bm{\Lambda}_{E, F}$-MTS-algebra by \cref{lem:MTS-structure-relative-Hom-scheme}, giving (1). Next denote 
  \[(\widehat{\scr{E}}_{\bm{\Gamma}}, \widehat{\scr{F}}_{\bm{\Gamma}}) \coloneq \bm{\Gamma}_{E, F, f} \otimes_{\bm{\Lambda}_{E, F}} (\widehat{\scr{E}}, \widehat{\scr{F}})\]
  \tab By \cref{lem:MTS-enhancement-relative-hom-representability}, the universal section $\widehat{f} \in \Hom_{\bm{\Lambda}_{E, F}\Rm{-alg}}(\Rm{Sym}_{\bm{\Lambda}_{E, F}} \sh{H}^0(\sh{K}^\vee), \bm{\Gamma}_{E, F, f})$, given by composing the universal section of $M^0_{\bm{\Lambda}_{E, F}}(\sh{K})$ with the completion at $f$, corresponds to a morphism in $D(\bm{\Gamma}_{E, F, f}\Rm{-MTS-module})$
  \[\bm{\Gamma}_{E, F, f} \to \bm{\Gamma}_{E, F, f} \otimes^{\Rm{L}}_{\bm{\Lambda}_{E, F}} \sh{K} \simeq (a_X)_* \ul{\Hom}_{\Gamma_{E, F, f}}(\widehat{\scr{E}}_{\bm{\Gamma}}, \widehat{\scr{F}}_{\bm{\Gamma}})[-1]\]
  where the last equality comes from $(\widehat{\scr{E}}, \widehat{\scr{F}})$ being free (hence $\ul{\Hom}$ commutes with tensor product). By \eqref{eq:adjunction-AVMTS-hom}, we get a corresponding morphism $\widehat{f}: \widehat{\scr{E}}_{\bm{\Gamma}} \to \widehat{\scr{F}}_{\bm{\Gamma}}$ of $\bm{\Gamma}_{E, F, f}$-AVMTS. Now let $(\scr{E}_{\sh{A}}, \scr{F}_{\sh{A}}, f_{\sh{A}})$ be a framed free $\sh{A}$-AVMTS deformation of $(\scr{E}, \scr{F}, f)$. By \cref{thm:Betti-MTS-algebra}, there is a unique $\bm{\Lambda}_{E, F} \to \sh{A}$ such that 
  \[(\scr{E}_{\sh{A}}, \scr{F}_{\sh{A}}) \simeq \sh{A} \otimes_{\bm{\Lambda}_{E, F}} (\widehat{\scr{E}}, \widehat{\scr{F}})\]
  and since $f_{\sh{A}}$ is a morphism of $\sh{A}\Rm{-AVMTS}$, \eqref{eq:adjunction-AVMTS-hom} gives a corresponding 
  \[\sh{A} \to (a_X)_*\ul{\Hom}_{\sh{A}}(\scr{E}_{\sh{A}}, \scr{F}_{\sh{A}})[-1] \simeq \sh{A} \otimes_{\bm{\Lambda}_{E, F}}^{\Rm{L}} \sh{K}\]
  which, by \cref{lem:MTS-enhancement-relative-hom-representability}, is the same as an $\bm{\Lambda}_{E, F}$-algebra morphism $\Rm{Sym}_{\bm{\Lambda}_{E, F}}\sh{H}^0(\sh{K}^\vee) \to \sh{A}$. Since $\sh{A}$ is artinian, hence complete, we get a factorization $\bm{\Gamma}_{E, F, f} \to \sh{A}$, finishing (2). \\
  \tab Next, by \cref{thm:relative-DR-Dol-quasi-isom-with-coeff} and the De Rham comparison, we have 
  \[\Rm{L}\iota_1 \sh{K}_n \simeq \Rm{R}\Gamma(C_{\DR}(\scr{E}_{n, 1}, \scr{F}_{n, 1})) \simeq \Rm{R}\Gamma(\ul{\Hom}_{\Lambda_n}(E_n, F_n))\]
  where $E_n, F_n$ are the universal local systems over $\Lambda_n$. Thus \cref{lem:relative-hom-and-restriction} gives that the fiber over 1 of $\bm{\Gamma}_{E, F, f}$ is isomorphic to $\widehat{\sh{O}}_{R_B^{\trip}, (E, F, f)}$. The other two remaining identifications in (3) follow from a similar argument, where by \cref{thm:Betti-MTS-algebra},
  \[\bm{\Lambda}_{E, F} \vert_{\lambda = 1} \simeq \widehat{\sh{O}}_{R_{\DR}^{\pair}, (\bar{E}, \bar{F})}, \tab \bm{\Lambda}_{E, F} \vert_{\lambda = 0} \simeq \widehat{\sh{O}}_{R_{\Dol}^{\pair}, \Sf{SM}^{\Rm{ps}}(\bar{E}, \bar{F})}\]
  by \cref{thm:relative-DR-Dol-quasi-isom-with-coeff},
  \[\Rm{L}\iota_1 \sh{K} \simeq \Rm{R}\Gamma(C_{\DR}(\widehat{\scr{E}}_1, \widehat{\scr{F}}_1)), \tab \Rm{L}\iota_0 \sh{K} \simeq \Rm{R}\Gamma(C_{\Dol}(\widehat{\scr{E}}_0, \widehat{\scr{F}}_0))\]
  with $(\widehat{\scr{E}}_1, \widehat{\scr{F}}_1))$ universal over $\widehat{\sh{O}}_{R_{\DR}^{\pair}, (\bar{E}, \bar{F})}$ and $(\widehat{\scr{E}}_0, \widehat{\scr{F}}_0)$ universal over $ \widehat{\sh{O}}_{R_{\Dol}^{\pair}, \Sf{SM}^{\Rm{ps}}(\bar{E}, \bar{F})}$, and finally that the universal section commutes with restriction due to \cref{lem:relative-hom-and-restriction}.
\end{proof}

\section{Equivariant formal De Rham - Dolbeault correspondence for triples}
\subsection{The socle filtration and Rees cocharacter} 
\label{sec:socle-filt-Rees}
The categories $\mathsf{Conn}_{\Rm{nilp}}(\bar{X}, D)$ and $\mathsf{Higgs}_{\Rm{nilp}}^{0, ss}(\bar{X}, D)$ are $\CC$-linear abelian of finite length, and extension-closed in the bigger categories of $\sh{D}_{\bar{X}}(\log D)$ and $\Rm{Sym} T_X(-\log T)$-modules. Their simple objects are, respectively, irreducible logarithmic flat connections and stable degree 0 logarithmic Higgs bundles.

\begin{definition}[{\cite[Definition 1.13.1]{etingof2015tensor}}]
  Let $\mathsf{C}$ be either of the above categories. For $E \in \mathsf{C}$, denote $\Rm{soc}(E)$ the sum of all simple subobjects of $E$, i.e., the unique maximal semisimple subobject. We have the socle filtration $S_\bullet E$
  \[S_0E = 0, \tab S_1E = \Rm{soc}(E), \tab S_{i + 1}E/S_iE = \Rm{soc}(E/S_iE)\]
  satisfying, for $f: E \to F$,
  \begin{enumerate}
    \item $f(S_iE) \subseteq S_iF$;
    \item $S_i(E \oplus F) = S_iE \oplus S_iF$;
    \item $\Rm{gr}_S E$ is semistable;
    \item $(\Rm{gr}_S g) \circ (\Rm{gr}_S f) = \Rm{gr}_S(g \circ f)$. 
  \end{enumerate}
\end{definition}

\begin{definition}
  Let $\pi: \bar{X} \times \AAA^1_\lambda \to \bar{X}$ be the projection. Define the Rees family associated to $S_\bullet E$ to be 
  \[\Rm{Rees}_S(E) \coloneq \Sum{i \in \ZZ} t^i \pi^*(S_i E) \subset \pi^*E \otimes_{\sh{O}_{\bar{X}}[t]} \sh{O}_{\bar{X}}[t, t^{-1}]\]
  which is locally free on $\bar{X} \times \AAA^1_t$, and 
  \begin{enumerate}
    \item $\Rm{Rees}_S(E)$ has a flat connection (resp. Higgs field) from restricting $\pi^*\nabla$ (resp. $\pi^*\theta$);
    \item $\Rm{Rees}_S(E) \vert_{t = 1} \simeq E$ and $\Rm{Rees}_S(E) \vert_{t = 0} \simeq \Rm{gr}_S E$, equipped with either $\Rm{gr}_S \nabla$ or $\Rm{gr}_S \theta$;
    \item for $f: E \to F$ we have $\Rm{Rees}_S(f): \Rm{Rees}_S(E) \to \Rm{Rees}_S(F)$ from restricting $\pi^*f$, and 
    \[\Rm{Rees}_S(f) \vert_{t = 1} \simeq f, \tab \Rm{Rees}_S(f) \vert_{t = 0} \simeq \Rm{gr}_S f, \tab \Rm{Rees}_S(g \circ f) = \Rm{Rees}_S(g) \circ \Rm{Rees}_S(f)\]
  \end{enumerate}
\end{definition}
Taking framing into account, $\Rm{Rees}_S(E)$ corresponds to a morphism $\AAA_t^1 \to R_\Box(r)$ which agrees (we can adapt the proof of \cite[Lemma 4.4.3]{huybrechts2010geometry}) with the following arc coming from a $\GG_m$-action on $R_\Box(r)$. Let the framed fiber be $\phi: E_x \simeq \CC^r$ and choose an adapted splitting
\[E_x = \bigoplus_i G_i, \tab S_i E_x = \bigoplus_{j \leq i} G_j\]
and $\gamma': \GG_m(\CC) \to \Rm{GL}(E_x)$ where $\gamma'(t)\vert_{G_i} = t^{-i}\Rm{id}_{G_i}$. Define $\gamma_E \coloneq \phi \circ \gamma' \circ \phi: \GG_m(\CC) \to \Rm{GL}_r(\CC)$. In fact this can be defined as an algebraic action, and for artinian local $A$, the action $\gamma_E: \GG_m(A) \to \Rm{GL}_r(A)$ can be defined in exactly the same way. The arc is defined via the change-of-framing action: 
\[\alpha_E: \AAA^1_t \to R_\Box(r), \tab \alpha_E(t \neq 0) \coloneq \gamma_E(t) \cdot (E, \phi), \tab \alpha_E(0) = \brac*{\Rm{gr}_S E, \Rm{gr}_S \phi}\]
\tab By definition, $\alpha_E$ is equivariant with respect to $\gamma_E$. One can also check that $[\alpha_E(0)]$ is fixed by $\gamma_E$, thus we have an induced action on $\sh{O}_{R_\Box(r), \alpha_E(0)}$. 

\subsection{A logarithmic De Rham - Dolbeault cohomology comparison with coefficients} 
We strengthen a result of Sabbah \cite[Lemma 8.12]{bakker2024linear} by adding artinian coefficients, and composition-compatibility.
\subsubsection{Internal Homs with coefficients}
Let $\sh{A}$ be a local artinian $\sh{T}(0)$-MTS-algebra, and $A = \abs{\sh{A}}$. Let $\scr{H}$ be a free $\sh{A}$-AVMTS, and $\scr{M}_{\scr{H}}^A[*D]$ the underlying $\scr{R}_{\bar{X}}-$module of $j_*\scr{H}$ on the chart $\AAA^1_\lambda$; this has an $A[\lambda]$ action coming from functoriality of the prolongation, hence the $^A$ superscript. Also denote $\bar{X}_\lambda = \bar{X} \times \AAA^1_\lambda, D_\lambda = D \times \AAA^1_\lambda$, and put $\scr{R}_{\bar{X}, A[\lambda]} \coloneq A[\lambda] \otimes_{\CC[\lambda]} \scr{R}_{\bar{X}}$. Locally, $\scr{R}_{\bar{X}, A[\lambda]}$ is generated over $A[\lambda] \otimes_{\CC[\lambda]}\sh{O}_{\bar{X}_\lambda}$ by $\lambda \partial_z$; the same is true for $\scr{R}_{\bar{X}, A[\lambda]}(*D)$ over $A[\lambda] \otimes_{\CC[\lambda]}\sh{O}_{\bar{X}_\lambda}(*D)$. Then $\scr{M}_{\scr{H}}^A[*D]$ is a $\scr{R}_{\bar{X}, A[\lambda]}-$module. Let $\scr{E}, \scr{F}$ be free $\sh{A}$-AVMTS, we can define 
\[\ul{\Hom}_{\sh{A}}(\scr{E}, \scr{F}) \coloneq \ker\brac*{\ul{\Hom}_{\sh{T}(0)}(\scr{E}, \scr{F}) \to \ul{\Hom}_{\sh{T}(0)}(\sh{A} \otimes_{\sh{T}(0)}\scr{E}, \scr{F})}\]
to be the internal Hom in the category of $\sh{A}$-AVMTS. This is a $\CC$-AVMTS because the category of $\CC$-AVMTS is abelian, and the $\sh{A}$-action comes from definition. For the underlying local system this is just $\ul{\Hom}_{A}$ of $A$-local systems. Next define, for meromorphic extensions which also carry $A[\lambda]$-action,
\begin{align*}
  \scr{M}_{\scr{E}, \scr{F}}^\CC(*D) &\coloneq \shHom_{\sh{O}_{\bar{X}_\lambda}(*D)}(\scr{M}_{\scr{E}}(*D), \scr{M}_{\scr{F}}(*D)) = \scr{M}_{\ul{\Hom}_{\sh{T}(0)}(\scr{E}, \scr{F})}(*D)\\ 
  \scr{M}_{\scr{E}, \scr{F}}^{A}(*D) &\coloneq \shHom_{A[\lambda] \otimes \sh{O}_{\bar{X}_\lambda}(*D)}(\scr{M}_{\scr{E}}(*D), \scr{M}_{\scr{F}}(*D))
\end{align*}

\begin{prop}
  Define a morphism 
  \[\Delta_{\scr{E},\scr{F}}(*D): \scr{M}_{\scr{E}, \scr{F}}^\CC(*D) \to \scr{M}_{\sh{A} \otimes \scr{E}, \scr{F}}^\CC(*D), \tab \phi \mapsto \brac*{a \otimes e \mapsto \phi(a \cdot e) - a\cdot \phi(e)}\]
  then $\Delta_{\scr{E},\scr{F}}(*D)$ commutes with $A[\lambda]-$multiplication and $\lambda \partial_z$, hence a $\scr{R}_{\bar{X}, A[\lambda]}(*D)$-linear. We have
  \[\scr{M}_{\ul{\Hom}_{\sh{A}}(\scr{E}, \scr{F})}(*D) = \ker \Delta_{\scr{E}, \scr{F}}(*D) = \scr{M}_{\scr{E}, \scr{F}}^{A}(*D)\]
  and a corresponding $\scr{R}_{\bar{X}, A[\lambda]}-$linear map $\Delta_{\scr{E}, \scr{F}}[*D]$ for which 
  \[\scr{M}^A_{\scr{E}, \scr{F}}[*D] \coloneq \scr{M}_{\ul{\Hom}_{\sh{A}}(\scr{E}, \scr{F})}[*D] = \ker \brac*{\scr{M}_{\ul{\Hom}_{\sh{T}(0)}(\scr{E}, \scr{F})}[*D] \xrightarrow{\Delta_{\scr{E}, \scr{F}}[*D]} \scr{M}_{\ul{\Hom}_{\sh{T}(0)}(\sh{A} \otimes \scr{E}, \scr{F})}[*D]}\]
  \tab Furthermore, this is compatible with the V-filtration, 
  \[V_b\scr{M}^A_{\scr{E}, \scr{F}}[*D] = \ker \brac*{V_b\scr{M}_{\ul{\Hom}_{\sh{T}(0)}(\scr{E}, \scr{F})}[*D] \xrightarrow{V_b\Delta_{\scr{E}, \scr{F}}[*D]} V_b\scr{M}_{\ul{\Hom}_{\sh{T}(0)}(\sh{A} \otimes \scr{E}, \scr{F})}[*D]}\]
\end{prop}
\begin{proof}
  That $\Delta_{\scr{E}, \scr{F}}(*D)$ is $\scr{R}_{\bar{X}, A[\lambda]}(*D)$-linear can be checked locally. We get the identification of $\scr{M}_{\ul{\Hom}_{\sh{A}}(\scr{E}, \scr{F})}(*D)$ essentially by uniqueness of meromorphic extension. The existence of $\Delta_{\scr{E}, \scr{F}}[*D]$ comes from \cite[Lemma 3.1.19]{mochizuki2015mixed} which also gives $A[\lambda]$-linearity. By \cite[Proposition 9.1.8]{mochizuki2015mixed}, $\Delta_{\scr{E}, \scr{F}}[*D]$ is strict with respect to KMS-structure, hence the V-filtration. Finally, \cite[Lemma 3.1.20 and 3.1.21]{mochizuki2015mixed} gives the remaining identifications. 
\end{proof}

\begin{coroll}
  \label{lem:boundary-Hom-lattices}
  We have functorial and artinian-base-change-compatible identifications 
  \begin{align*}
    V_{-1}\scr{M}_{\ul{\Hom}_{\sh{A}}(\scr{E}, \scr{F})}[*D] &= \shHom_{A[\lambda] \otimes \sh{O}_{\bar{X}_\lambda}}\brac*{\bar{\scr{E}}, \bar{\scr{F}}}\\
    V_0\scr{M}_{\ul{\Hom}_{\sh{A}}(\scr{E}, \scr{F})}[*D] &= z^{-1}\shHom_{A[\lambda] \otimes \sh{O}_{\bar{X}_\lambda}}\brac*{\bar{\scr{E}}, \bar{\scr{F}}}
  \end{align*}
\end{coroll}
\begin{proof}
  By \cref{lem:boundary-lattices} and the previous proposition, it remains to show that Deligne's extension commutes with taking $\shHom$, but this is straightforward from a local monodromy calculation. 
\end{proof}

\subsubsection{The comparison}
\label{sec:The-comparison}
From now on, assume that the $\sh{A}$-AVMTS we work with always have unipotent local monodromy. We have a De Rham complex 
\[\Rm{DR}_{\scr{R}}(\scr{H}) \coloneq \Rm{DR}_{\scr{R}}(\scr{M}_{\scr{H}}^A[*D]) = \sqbrac*{\scr{M}_{\scr{H}}^A[*D] \xrightarrow{\nabla_\lambda} \scr{M}_{\scr{H}}^A[*D] \otimes \Omega^{1}_{\bar{X}_\lambda/\AAA^1_\lambda}}\]
\tab Let $q: \bar{X}_\lambda \to \AAA^1_\lambda$ be the projection. Set $K^{\tw}_A(\scr{H}) \coloneq \Rm{R}q_* \Rm{DR}_{\scr{R}}(\scr{M}_{\scr{H}}^A[*D]) \in D^b(A[\lambda])$. Since $q$ is a projective morphism, this underlies a twistor object in $D^b(\sh{A}-\text{MTS-module})$. As a result, we get a splitting from \cref{lem:affinesplitting}, and a fiber identification
\[\varphi_{\scr{H}}: \Rm{L}\iota_0^* K^\tw_A(\scr{H}) \xrightarrow{\simeq} \Rm{L}\iota_1^* K^\tw_A(\scr{H})\]
On the other hand, we also have the functorial logarithmic extension $\bar{\scr{H}}_A$, also equipped with an $A[\lambda]-$action, and a logarithmic De Rham complex 
\[C^{\log}_A(\scr{H}) \coloneq \sqbrac*{\bar{\scr{H}}_A \xrightarrow{\nabla_\lambda} \bar{\scr{H}}_A \otimes \Omega^1_{\bar{X}_\lambda/\AAA^1_\lambda}(\log D_\lambda)}\]
with $K^{\log}_A(\scr{H}) \coloneq \Rm{R}q_*C^{\log}_A(\scr{H}) \in D^b(A[\lambda])$. For free $\sh{A}$-AVMTS $\scr{E}, \scr{F}$, define
\[C^{\log}_A(\scr{E}, \scr{F}) \coloneq C^{\log}_A\brac*{\ul{\Hom}_{\sh{A}}(\scr{E}, \scr{F})}, \tab \Rm{DR}_{\scr{R}}(\scr{E}, \scr{F}) \coloneq \Rm{DR}_{\scr{R}}\brac*{\ul{\Hom}_{\sh{A}}(\scr{E}, \scr{F})}\]
and similarly $K^{\log/\Rm{tw}}_A(\scr{E}, \scr{F})$. Let $\iota_c: \Spec A \to \Spec A[\lambda]$ be the $\lambda-$restriction, but remembering the $A-$action. 
\begin{prop}
  \label{thm:relative-DR-Dol-quasi-isom-with-coeff}
  There is a canonical quasi-isomorphism in $D^b(A[\lambda])$ 
  \[\Phi_{\scr{H}}^A: K^\tw_A(\scr{H}) \xrightarrow{\simeq} K^{\log}_A(\scr{H})\tag{4.1}\label{eq:twistor-log-quasi-isom}\]
  which commutes with artinian $\sh{T}(0)-$MTS-algebra coefficient base change, and restricts to quasi-isomorphisms 
  \begin{align*}
    \Rm{L}\iota_0^* K^\tw_A(\scr{H}) &\simeq R\Gamma\brac*{\bar{X}, \sqbrac*{\bar{\scr{H}}_{A, 0} \xrightarrow{\theta} \bar{\scr{H}}_{A, 0} \otimes \Omega^1_{\bar{X}}(\log D)}} \eqcolon K_{A, 0}^{\log}(\scr{H})\\ 
    \Rm{L}\iota_1^* K^\tw_A(\scr{H}) &\simeq R\Gamma\brac*{\bar{X}, \sqbrac*{\bar{\scr{H}}_{A, 1} \xrightarrow{\nabla} \bar{\scr{H}}_{A, 1} \otimes \Omega^1_{\bar{X}}(\log D)}} \eqcolon K_{A, 1}^{\log}(\scr{H})\tag{4.2} \label{eq:log-Dol-DR-derived-fibers}
  \end{align*}
  in $D^b(A)$, over $\lambda = 0, 1$. For $\sh{A}$-AVMTS $\scr{E}, \scr{F}$, let 
  \begin{align*}
    \phi^0_{\scr{E}, \scr{F}}: &\HH^0\brac*{\bar{X}, \sqbrac*{\overline{\ul{\Hom}_{\sh{A}}(\scr{E}, \scr{F})}_{A, 1} \xrightarrow{\nabla} \overline{\ul{\Hom}_{\sh{A}}(\scr{E}, \scr{F})}_{A, 1} \otimes \Omega^1_{\bar{X}}(\log D)}}\\
    &\xrightarrow{\HH^0(\Rm{L}\iota_0^*\Phi)H^0(\varphi)\HH^0(\Rm{L}\iota_1^*\Phi)^{-1}}\HH^0\brac*{\bar{X}, \sqbrac*{\overline{\ul{\Hom}_{\sh{A}}(\scr{E}, \scr{F})}_{A, 0} \xrightarrow{\theta} \overline{\ul{\Hom}_{\sh{A}}(\scr{E}, \scr{F})}_{A, 0} \otimes \Omega^1_{\bar{X}}(\log D)}}
  \end{align*}
  be the $A$-linear $\HH^0$-isomorphism. Then it preserves identity and composition, i.e., for $\scr{E}, \scr{F}, \scr{G}$, 
  \[\phi^0_{\scr{E}, \scr{E}}(\Rm{id}_{\scr{E}_1}), \tab \phi^0_{\scr{E}, \scr{G}}(v \circ u) = \phi^0_{\scr{F}, \scr{G}}(v) \circ \phi^0_{\scr{E}, \scr{F}}(u)\]
  as well as zero, addition, and $A$-scalars. Equivalently, we have a commutative diagram of $A$-modules 
  \[\begin{tikzcd}
	{H^0\brac*{K^{\log}_{A, 1}(\scr{F}, \scr{G})} \otimes_A H^0\brac*{K^{\log}_{A, 1}(\scr{E}, \scr{F})}} && {H^0\brac*{K^{\log}_{A, 1}(\scr{E}, \scr{G})}} \\
	\\
	{H^0\brac*{K^{\log}_{A, 0}(\scr{F}, \scr{G})} \otimes_A H^0\brac*{K^{\log}_{A, 0}(\scr{E}, \scr{F})}} && {H^0\brac*{K^{\log}_{A, 0}(\scr{E}, \scr{G})}}
	\arrow["\circ", from=1-1, to=1-3]
	\arrow["{\phi^0_{\scr{F}, \scr{G}} \otimes\phi^0_{\scr{E}, \scr{F}}}"', from=1-1, to=3-1]
	\arrow["{\phi^0_{\scr{E}, \scr{G}}}", from=1-3, to=3-3]
	\arrow["\circ", from=3-1, to=3-3]
\end{tikzcd}\tag{4.3} \label{eq:H0-composition-compatibility}\]
  These $H^0$-compatibilities also commute with artinian $\sh{T}(0)-$MTS-algebra coefficient base change. 
\end{prop}
\begin{proof}
  Work locally near $p \in D$ with equation $(z = 0)$. First we check that the isomorphisms in {\eqref{eq:V-filt-isoms}}, 
  \[\lambda \partial_z: \Rm{gr}^V_b \scr{M}^{A}_{\scr{H}}[*D] \xrightarrow{\simeq} \Rm{gr}^V_{b + 1} \scr{M}^{A}_{\scr{H}}[*D] \tab b > -1, \tab z: V_0\scr{M}^{A}_{\scr{H}}[*D] \xrightarrow{\simeq} V_{-1}\scr{M}^{A}_{\scr{H}}[*D]\]
  are compatible with the $A[\lambda]-$action. Functoriality of V-filtration gives that each $V_b\scr{M}^{A}_{\scr{H}}[*D]$ has a $A[\lambda]-$action. Functoriality of the prolongation gives that the $A[\lambda]-$action on $\scr{M}^{A}_{\scr{H}}[*D]$ is $\scr{R}_{\bar{X}}$-linear, hence must commute $\lambda \partial_z$ and multiplication by $\sh{O}_{\bar{X}_\lambda}$ (in particular $z$).\\
  \tab Next, we have quasi-isomorphisms (by the previous isomorphisms involving V-filtration):
  \[\begin{tikzcd}
	{V_{-1}\scr{M}^{A}_{\scr{H}}[*D]} && {V_{-1}\scr{M}^{A}_{\scr{H}}[*D] \otimes \Omega^{1}_{\bar{X}_\lambda/\AAA^1_\lambda}(\log D_\lambda)} \\
	\\
	{V_{-1}\scr{M}^{A}_{\scr{H}}[*D]} && {V_{0}\scr{M}^{A}_{\scr{H}}[*D] \otimes \Omega^{1}_{\bar{X}_\lambda/\AAA^1_\lambda}} \\
	\\
	{\scr{M}^{A}_{\scr{H}}[*D]} && {\scr{M}^{A}_{\scr{H}}[*D] \otimes \Omega^{1}_{\bar{X}_\lambda/\AAA^1_\lambda}}
	\arrow[from=1-1, to=1-3]
	\arrow["{=}", from=3-1, to=1-1]
	\arrow[from=3-1, to=3-3]
	\arrow[hook, from=3-1, to=5-1]
	\arrow["{z \otimes z^{-1}}"', from=3-3, to=1-3]
	\arrow[hook, from=3-3, to=5-3]
	\arrow[from=5-1, to=5-3]
  \end{tikzcd} \tag{4.4}\label{eq:BBT-8.12-quasi-isoms}\]
  where the top square is $A[\lambda]$-linear since multiplying with $z$ is (and the $A[\lambda]$-action is on the left). The bottom square is $A[\lambda]-$linear because inclusions $V_b\scr{M}^{A}_{\scr{H}}[*D] \hookrightarrow \scr{M}^{A}_{\scr{H}}[*D]$ is. Now using \cref{lem:boundary-lattices}, we can identify the top row with $C^{\log}_A$, and apply $Rq_*$ to get $\Phi_A: K^\tw_A \xrightarrow{\simeq} K^{\log}_A$, which is $A[\lambda]-$linear, as desired. Now, $K^{\log}_A = \Rm{R}q_* C^{\log}_A$ satisfies base change (same proof as in \cite[\href{https://stacks.math.columbia.edu/tag/0FM0}{Tag 0FM0}]{stacks-project}) with restrictions $\iota_0, \iota_1$ thus giving \eqref{eq:log-Dol-DR-derived-fibers}. \\
  \\
  \textit{Compatibility with artinian $\sh{T}(0)$-MTS-algebra coefficient base change:} 
  This follows from the identification of top row with $C^{\log}_A$, given by \cref{lem:boundary-lattices}, being compatible, and base changes 
  \[A'[\lambda] \otimes_{A[\lambda]}^{\Rm{L}} K^{\log}_{A} \xrightarrow{\simeq} K^{\log}_{A'}, \tab A'[\lambda] \otimes_{A[\lambda]}^{\Rm{L}} K^{\tw}_{A} \xrightarrow{\simeq} K^{\tw}_{A'}\]
  \textit{$H^0$-compatibility with composition:} For $\scr{E}, \scr{F}$, denote 
  \[C^V_A(\scr{E}, \scr{F}) \coloneq \brac*{V_{-1}\scr{M}_{\scr{E}, \scr{F}}^A[*D] \xrightarrow{\nabla_\lambda} V_{0}\scr{M}_{\scr{E}, \scr{F}}^A[*D] \otimes \Omega^{1}_{\bar{X}_\lambda/\AAA^1_\lambda}}\]
  and $K^V_A(\scr{E}, \scr{F}) \coloneq \Rm{R}q_* C^{\log}_A(\scr{E}, \scr{F})$. We will use $\HH^0$ interchangeably with $H^0\Gamma$, which is allowed due to the hypercohomology spectral sequence. In particular, for $c \in \set*{0, 1}$,
  \[\HH^0\brac*{\bar{X}, \Rm{L}\iota_c^*C^{\log/V}_A} = H^0(\Rm{L}\iota_c^*K^{\log/V}_A), \tab \HH^0\brac*{\bar{X}, \Rm{L}\iota_c^*\Rm{DR}_{\scr{R}}} = H^0(\Rm{L}\iota_c^* K^{\Rm{tw}}_A)\]
  \tab Let the quasi-isomorphism coming from \eqref{eq:BBT-8.12-quasi-isoms} and applying $\Rm{R}q_*$ be named 
  \[K^{\log}_A(\scr{E}, \scr{F}) \xleftarrow[\simeq]{\nu_{\scr{E}, \scr{F}}} K^V_A(\scr{E}, \scr{F}) \xrightarrow[\simeq]{\tau_{\scr{E}, \scr{F}}} K^{\Rm{tw}}_A(\scr{E}, \scr{F})\]
  which induces $A$-isomorphisms 
  \[H^0(\Rm{L}\iota_c^*K^{\log}_A(\scr{E}, \scr{F})) \xleftarrow[\simeq]{\nu_{\scr{E}, \scr{F}}^c} H^0(\Rm{L}\iota_c^*K^V_A(\scr{E}, \scr{F})) \xrightarrow[\simeq]{\tau_{\scr{E}, \scr{F}}^c} H^0(\Rm{L}\iota_c^*K^{\Rm{tw}}_A(\scr{E}, \scr{F}))\]
  \tab Restriction is exact for $C^{\log/V}_A$ and $K^{\log/V}_A$ since the terms are $\lambda$-flat, hence first two terms are just 
  \[H^0(K^{\log}_{A, c}(\scr{E}, \scr{F})) = H^0\Gamma(\bar{X}, C^{\log}_{A, c}(\scr{E}, \scr{F})), \tab H^0(K^{V}_{A, c}(\scr{E}, \scr{F})) = H^0\Gamma(\bar{X}, C^{V}_{A, c}(\scr{E}, \scr{F}))\]
  \tab Notice then that $H^0(K^{\log}_{A, c}(\scr{E}, \scr{F}))$ gives the space of horizontal log-flat sections at $c = 1$ and log Higgs sections at $c = 0$. We have compositions for these sections. Since $\nu^c_{\scr{E}, \scr{F}}$ is identity on degree zero, and one can check that composition sends two flat $V_{-1}$-sections into a flat $V_{-1}$-section (with respect to the identification \cref{lem:boundary-Hom-lattices}), so composition makes sense for $H^0(K^{V}_{A, c}(\scr{E}, \scr{F}))$ as well. 
  Define a \textit{composition law} on $H^0(\Rm{L}\iota_c^*K^{\Rm{tw}}_A(-, -))$ 
  \[\mu^c_{\Rm{tw}}(v, u) \coloneq \tau^c_{\scr{E}, \scr{G}}\brac*{(\tau_{\scr{F}, \scr{G}}^{c})^{-1}v \circ \tau_{\scr{E}, \scr{F}}^{c}u}\]
  then we have a commutative diagram
  \[\begin{tikzcd}
	{H^0(\Rm{L}\iota_c^*K^{\Rm{tw}}_A(\scr{F}, \scr{G})) \otimes H^0(\Rm{L}\iota_c^*K^{\Rm{tw}}_A(\scr{E}, \scr{F}))} && {H^0(\Rm{L}\iota_c^*K^{\Rm{tw}}_A(\scr{E}, \scr{G}))} \\
	\\
	{H^0(K^{\log}_{A, c}(\scr{F}, \scr{G})) \otimes H^0(K^{\log}_{A, c}(\scr{E}, \scr{F}))} && {H^0(K^{\log}_{A, c}(\scr{E}, \scr{G}))}
	\arrow["{\mu_{\Rm{tw}}^c}", from=1-1, to=1-3]
	\arrow["{\psi_{\scr{F}, \scr{G}}^c \otimes \psi_{\scr{E}, \scr{F}}^c}"', from=1-1, to=3-1]
	\arrow["{\psi^c_{\scr{E},\scr{G}}}", from=1-3, to=3-3]
	\arrow["\circ", from=3-1, to=3-3]\tag{4.5}\label{eq:c-tw-composition-diagram}
\end{tikzcd}\]
  by defining $\psi^c_{\scr{E}, \scr{F}} \coloneq (\tau^c_{\scr{E}, \scr{F}})^{-1} \circ \nu_{\scr{E}, \scr{F}}^c: H^0(\Rm{L}\iota_c^*K^{\Rm{tw}}_A(\scr{E}, \scr{F})) \xrightarrow{\simeq} H^0(K^{\log}_{A, c}(\scr{E}, \scr{F}))$. \\
  \tab Next, let $a_X: X \to \Rm{pt}$, and define $\sh{K}_{\scr{E}, \scr{F}} \coloneq (a_X)_*\ul{\Hom}_{\sh{A}}(\scr{E}, \scr{F})[-1] \in D^b(\sh{A}-\text{MTS-module})$. Its $\scr{R}_{\bar{X}}-$module, on the $\AAA^1_\lambda$-chart, is $K^{\Rm{tw}}_A(\scr{E}, \scr{F})$ (because non-proper pushforward is defined via proper pushforward on the compactification, see \cite[Section 14.3]{mochizuki2015mixed}). So each cohomology $\sh{H}^j(\sh{K}_{\scr{E}, \scr{F}})$ is a twistor structure, hence locally free over $\PP^1$, and for $c \in \set*{0, 1}$,
  \[\sh{H}^0(\sh{K}_{\scr{E}, \scr{F}}) \vert_{\lambda = c} \simeq H^0(\Rm{L}\iota_a^*K^{\Rm{tw}}_A(\scr{E}, \scr{F})) \tag{4.6} \label{eq:H0-twistor-strict-restriction}\]
  \tab Now, on $\AAA_\lambda^1$-chart, we identified $\sh{H}^0(\sh{K}_{\scr{E}, \scr{F}})$ with relative $\lambda-$flat $V_{-1}$-sections, and defined a composition law $\mu_{\Rm{tw}}$. We also identified $V_{-1}$ piece with the logarithmic extension of the family of $\lambda$-connection coming from a AVMTS, so it must agree with composition on $X$. On the other hand, composition on $X$ is a morphism of $\sh{A}-$AVMTS (use \cref{def:C-VMTS-def} and the subsequent remark; composition is $D_\lambda$-flat and $A[\lambda]$-linear). So $\mu_{\Rm{tw}}$ is the restriction of a morphism of $\sh{A}$-MTS-modules 
  \[\mu_{\Rm{tw}}: \sh{H}^0(\sh{K}_{\scr{F}, \scr{G}}) \otimes \sh{H}^0(\sh{K}_{\scr{E}, \scr{F}}) \to \sh{H}^0(\sh{K}_{\scr{E}, \scr{G}})\]
  with fiber over $\lambda = c$ as defined before. Functoriality of the trivialization in \cref{lem:affinesplitting} then gives 
  \[\begin{tikzcd}
	{\sh{H}^0(\sh{K}_{\scr{F}, \scr{G}}) \vert_{\lambda = 1} \otimes_A \sh{H}^0(\sh{K}_{\scr{E}, \scr{F}}) \vert_{\lambda = 1}} && {\sh{H}^0(\sh{K}_{\scr{E}, \scr{G}}) \vert_{\lambda = 1}} \\
	\\
	{\sh{H}^0(\sh{K}_{\scr{F}, \scr{G}}) \vert_{\lambda = 0} \otimes_A \sh{H}^0(\sh{K}_{\scr{E}, \scr{F}}) \vert_{\lambda = 0}} && {\sh{H}^0(\sh{K}_{\scr{E}, \scr{G}}) \vert_{\lambda = 0}}
	\arrow["{\mu^1_{\Rm{tw}}}", from=1-1, to=1-3]
	\arrow["\varphi"', from=1-1, to=3-1]
	\arrow["\varphi", from=1-3, to=3-3]
	\arrow["{\mu^0_{\Rm{tw}}}", from=3-1, to=3-3]
\end{tikzcd}\]
  which, combining with \eqref{eq:c-tw-composition-diagram} and \eqref{eq:H0-twistor-strict-restriction}, gives \eqref{eq:H0-composition-compatibility}. Identity extends to identity on the logarithmic extension, so lies in $V_{-1}$, hence is fixed by the quasi-isomorphisms in \eqref{eq:BBT-8.12-quasi-isoms}. Thus identity is fixed throughout. Every diagram so far has been $A$-linear, hence we get all desired compatibilities. 
\end{proof}

\subsection{The formal isomorphism of deformation functors for triples}
Every bundle below comes equipped with a framing. Consider $\bar{E}, \bar{F}$ polystable logarithmic flat connections with nilpotent residues. By \cref{thm:Betti-MTS-algebra}, and abusing notation to denote $E, F$ the corresponding local systems on $X = \bar{X} \setminus D$, we have a pro$-\sh{T}(0)-$MTS deformation algebra of pairs (picking the $\CC$-AVMTS on $E, F$ to be the canonical weight 0 $\CC$-VTS) 
\[\bm{B}_{E, F} \coloneq \Rm{Spf}_{\PP^1} \bm{\Lambda}_E \times_{\PP^1} \Rm{Spf}_{\PP^1} \bm{\Lambda}_F \eqcolon \Rm{Spf}_{\PP^1} \bm{\Lambda}_{E, F}\]
whose fibers over 0, 1 are, by \cref{thm:Bett-DR-Dol-formal-isom}, 
\[\bm{B}_{E, F}\vert_{\lambda = 1} \simeq \widehat{R}^{\pair}_{\DR, (\bar E, \bar F)}, \tab \bm{B}_{E, F}\vert_{\lambda = 0} \simeq \widehat{R}^{\pair}_{\Dol, \Sf{SM}^{\Rm{ps}}(\bar E, \bar F)} \tag{4.7}\label{eq:pair-MTS-fiber-identification}\]
and the splitting of \cref{lem:affinesplitting} identifies the coefficient rings with $\Lambda_{E, F} = \abs*{\bm{\Lambda}_{E, F}}$. We have a universal $\bm{\Lambda}_{E, F}-$AVMTS pair $(\widehat{\scr{E}}, \widehat{\scr{F}})$ which restricts to the universal family of logarithmic flat pairs $(\widehat{E}_1, \widehat{F}_1)$ (resp. universal family of logarithmic Higgs pairs $(\widehat{E}_0, \widehat{F}_0)$) over $\lambda = 1$ (resp. $\lambda = 0$). For each $\bm{\Lambda}_{E, F} \to \bm{\Lambda}_{E, F, n} \coloneq \bm{\Lambda}_{E, F}/\Ide{m}_{\bm{\Lambda}_{E, F}}^{n + 1}$, pulling back gives $\bm{\Lambda}_{E, F, n}$-universal $(\scr{E}_n, \scr{F}_n)$. Its special mod $\Ide{m}_{\bm{\Lambda}}$ fiber has underlying local systems $(E, F)$ with unipotent local monodromy, hence $(\scr{E}_n, \scr{F}_n)$ also has unipotent local monodromy (since $\Ide{m}_{\bm{\Lambda}}$ is nilpotent). The fibers over $\lambda = 1, 0$ give logarithmic flat pair $(\bar E_{n, 1}, \bar F_{n, 1})$ and logarithmic Higgs pair $(\bar E_{n, 0}, \bar F_{n, 0})$. Denote
\[\scr{H}_n \coloneq \ul{\Hom}_{\bm{\Lambda}_{E, F, n}}(\scr{E}_n, \scr{F}_n)\]
then \cref{thm:relative-DR-Dol-quasi-isom-with-coeff} gives a quasi-isomorphism (see the definitions in \cref{lem:base-change-perfectness-Dol-DR-complexes})
\[\Phi_n: K_{\DR, \Lambda_{E, F, n}}(\bar{E}_{n, 1}, \bar{F}_{n, 1}) \xrightarrow{\simeq} K_{\Dol, \Lambda_{E, F, n}}(\bar{E}_{n, 0}, \bar{F}_{n, 0}) \tag{4.8}\label{eq:nth-level-DR-Dol-cohomology-comparison}\]
which commutes with arbitrary artinian base change. If we have $f: \bar{E} \to \bar{F}$ a flat morphism, then it extends to a morphism of the weight 0 $\CC$-VTS, and restricts to $\Sf{SM}^{\Rm{ps}}(f)$ on $\lambda = 0$. Thus $H^0(\Phi_0)(f) = \Sf{SM}^{\Rm{ps}}(f)$. 
\begin{remark}
  Compatibility of the quasi-isomorphism with arbitrary artinian base change, not just those coming from morphisms of $\sh{T}(0)$-MTS-algebras, is due to base change preserving quasi-isomorphism. However, $H^0$-compatibility with composition does not commute with arbitrary base change, because $H^0$ does not commute with base change. 
\end{remark}
For ease of notation, for $\Box \in \set*{\DR, \Dol}$, let $\hat{B}_{\Box, \bar{E}, \bar{F}} \coloneq \widehat{R}^\pair_{\Box, \brac*{\bar{E}, \bar{F}}}$ with $\hat{B}_{\Box, (\bar{E}, \bar{F}, \bar{G})}$ being the completion of the obvious moduli of triples of framed bundles. Notice that we have a similar trivialization as in \eqref{eq:pair-MTS-fiber-identification} for $\hat{B}_{\Box, (\bar{E}, \bar{F}, \bar{G})}$. Denote, for $f: \bar{E} \to \bar{F}, g: \bar{F} \to \bar{G}$, 
\[\hat{Z}_{\Box, f, g} \coloneq \widehat{R}^{\trip}_{\Box, \bar E, \bar F, f} \times_{\hat{B}_{\Box, (\bar{E}, \bar{F}, \bar{G})}} \widehat{R}^{\trip}_{\Box, \bar F, \bar G, g}\] 
where $\widehat{R}^{\trip}_{\Box, \bar E, \bar F, f}$ is the completion of the moduli of (pair + morphism) at $(\bar{E}, \bar{F}, f)$; this is also $M^0_{\Lambda}(K_{\DR, \Lambda}(\hat{E}_1, \hat{F}_1))_{f}$ for $\Box = \DR$ (and $M^0_{\Lambda}(K_{\Dol, \Lambda}(\hat{E}_0, \hat{F}_0))_{f}$ for $\Box = \Dol$). Finally, let $\Ide{U}_{\Box, \bar E, \bar F, f}$ and $\Ide{Z}_{\Box, f, g}$ be the universal families over $\widehat{R}^{\trip}_{\Box, \bar E, \bar F, f}$ and $\hat{Z}_{\Box, f, g}$.  

\begin{theorem}
  \label{thm:formal-isom-triples-and-composition}
  For $(\bar{E}, \bar{F}, f)$ with $\bar{E}, \bar{F}$ polystable logarithmic flat connections and $f: \bar{E} \to \bar{F}$, we have an isomorphism 
  \[\Psi_{\bar E, \bar F, f}: \widehat{R}^{\trip}_{\DR, \bar{E}, \bar{F}, f} \xrightarrow{\simeq} \widehat{R}^{\trip}_{\Dol, \Sf{SM}^{\Rm{ps}}(\bar{E}, \bar{F}, f)}\]
  Furthermore, if we let $\mu_{\DR}, \mu_{\Dol}$ be morphisms defined by composition on universal families, then there is a commutative diagram
  \[\begin{tikzcd}
	{\widehat{Z}_{\DR, f, g}} && {\widehat{Z}_{\Dol, \Sf{SM}^{\Rm{ps}}(f, g)}} \\
	\\
	{\widehat{R}^{\trip}_{\DR, \bar E, \bar G, g \circ f}} && {\widehat{R}^{\trip}_{\Dol, \Sf{SM}^{\Rm{ps}}(\bar E, \bar G, g \circ f)}}
	\arrow["{\Psi_{\bar{E}, \bar{F}, f} \times \Psi_{\bar{F}, \bar{G}, g}}", from=1-1, to=1-3]
	\arrow["\mu_{\DR}"', from=1-1, to=3-1]
	\arrow["\mu_{\Dol}", from=1-3, to=3-3]
	\arrow["{\Psi_{\bar{E}, \bar{G}, g \circ f}}", from=3-1, to=3-3]\tag{4.9} \label{eq:universal-composition-compatibility-diagram}
  \end{tikzcd}\]
  and analogous diagrams for direct sums, addition and scalar multiplication of universal maps. 
\end{theorem}

\begin{proof}
  \textit{The isomorphism $\Psi$:} Let $A$ be artinian local. Denote the trivialization of \eqref{eq:pair-MTS-fiber-identification} 
  \[\tau_{E, F}: \widehat{R}^{\pair}_{\DR, \bar{E}, \bar{F}} \xrightarrow{\simeq} \widehat{R}^{\pair}_{\Dol, \Sf{SM}^{\Rm{ps}}(\bar{E}, \bar{F})}\]
  \tab By \cref{lem:moduli-of-triples-and-deformation}, an $A$ point of $\widehat{R}^{\trip}_{\DR, \bar E, \bar F, f}$ is 
  \[(\bar{E}_A, \bar{F}_A, f_A), \tab f_A \in \HH^0(K_{\DR, A}(\bar{E}_A, \bar{F}_A)), \tab f_A \mod \Ide{m}_A = f\]
  since $(\bar{E}_A, \bar{F}_A)$ is a deformation of $(\bar{E}, \bar{F})$, we have a classifying morphism $\Lambda_{E, F} \to A$. This must factor through some quotient $\Lambda_n$, hence we can pullback \eqref{eq:nth-level-DR-Dol-cohomology-comparison} to get a quasi-isomorphism
  \[\Phi_A: K_{\DR, A}(\bar{E}_A, \bar{F}_A) \xrightarrow{\simeq} K_{\Dol, A}(\tau_{E, F}(\bar{E}_A, \bar{F}_A))\]
  which is independent of $n$ by base change compatibility. Then 
  \[((\bar{E}_A, \bar{F}_A), f_A) \mapsto \brac*{\tau_{E, F}(\bar{E}_A, \bar{F}_A), H^0(\Phi_A)(f_A)}\]
  with inverse constructed using $\tau^{-1}_{E, F}$ and $\Phi_A^{-1}$, gives a functorial bijection on Artinian points. By formal Yoneda \cite{schlessinger1968functors}, we get the isomorphism $\Psi_{\bar{E}, \bar{F}, f}$. Compatibility with universal families follows from Yoneda. \\
  \\
  \textit{Compatibility with composition, direct sum, addition, scalar multiplication:} Now let $E, F, G$ be the underlying local system for logarithmic flat connections $\bar{E}, \bar{F}, \bar{G}$, equipped with the canonical weight 0 $\CC$-VTS. We then have a pro-$\sh{T}(0)$-MTS deformation algebra $\bm{\Lambda}_{E, F, G}$ whose fibers over $\lambda = 1, 0$ give $\hat{B}_{\DR, \bar{E}, \bar{F}, \bar{G}}$ and $\hat{B}_{\Dol, \Sf{SM}^{\Rm{ps}}(\bar{E}, \bar{F}, \bar{G})}$. Similar statement for universal families $\brac*{\hat{\scr{E}}, \hat{\scr{F}}, \hat{\scr{G}}}$. Let 
  \[\hat{Z}_{\DR} \eqcolon \Rm{Spf}\ \Gamma\]
  then by \cref{thm:universal-MTS-framed-triple}, we have a pro-$\bm{\Lambda}_{E, F, G}$-MTS-algebra $\bm{\Gamma}$ with $\abs*{\bm{\Gamma}} = \Gamma$. Denote 
  \[\bm{\Gamma}_n \coloneq \bm{\Gamma}/\Ide{m}_{\bm{\Gamma}}^{n + 1}, \tab \Gamma_n = \abs*{\bm{\Gamma}_n}\]
  \tab The morphism $\bm{\Lambda}_{E, F, G} \to \bm{\Gamma} \to \bm{\Gamma}_n$ is a morphism of $\sh{T}(0)$-MTS-algebras, hence pulling back universal families give $\hat{\scr{E}}_n, \hat{\scr{F}}_n, \hat{\scr{G}}_n$ all $\bm{\Gamma}_n$-AVMTS. Again let their $\lambda-$fiber be $\bar{E}_{n, \lambda}, \bar{F}_{n, \lambda}, \bar{G}_{n, \lambda}$. Since we are over $\bm{\Gamma}$ which also parametrizes deformations of $f, g$, there are universal morphisms 
  \[f_n \in H^0\brac*{K_{\DR, \Gamma_n}(\bar{E}_{n, 1}, \bar{F}_{n, 1})}, \tab g_n \in H^0\brac*{K_{\DR, \Gamma_n}(\bar{F}_{n, 1}, \bar{G}_{n, 1})}\]
  which are just the restriction of the universal morphisms $\hat{f}, \hat{g}$ along $\Gamma \to \Gamma_n$. Now by \cref{thm:relative-DR-Dol-quasi-isom-with-coeff} applied to $\bm{\Gamma}_n$, $H^0(\Phi_n)$ commutes with composition. Thus the restriction of the diagram \eqref{eq:universal-composition-compatibility-diagram} to each $\Gamma_n$ commutes. Since $\Gamma = \varprojlim \Gamma_n$, we get the desired diagram. Compatibility with addition and scalar multiplication follows from the same argument, along with the fact that $H^0(\Phi_n)$ is $\Gamma_n$-linear. Direct sum is analogous. 
\end{proof}

\subsection{Rees equivariance}
Consider $(\bar{E}, \bar{F}, f)$ with $\bar E, \bar F$ (not necessarily polystable) logarithmic flat connections with framings $\phi_{\bar{E}}, \phi_{\bar F}$, and $f: \bar{E} \to \bar{F}$ a flat morphism. Consider the socle filtrations $S_\bullet$ on these bundles, and let
\[\bar{E}_x = \oplus_j V^{\bar{E}}_j, \tab \bar{F}_x = \oplus_j V^{\bar{F}}_j\]
be two splittings adapted to $S_\bullet$. Also denote $\brac*{\bar{E}^{\Rm{ss}}, \bar{F}^{\Rm{ss}}, f^{\Rm{ss}}} \coloneq \brac*{\Rm{gr}_S \bar{E}, \Rm{gr}_S \bar{F}, \Rm{gr}_S f}$. Then we have change-of-framing actions $\gamma_{\bar{E}}^{\DR}, \gamma_{\bar{F}}^{\DR}$:
\[(\phi^{-1}_{\bar E} \circ \gamma_{\bar{E}}^{\DR} \circ \phi_{\bar E})\vert_{V^{\bar{E}}_j} = t^{-j}\Rm{id}_{V^{\bar{E}}_j}, \tab (\phi^{-1}_{\bar F} \circ \gamma_{\bar{F}}^{\DR} \circ \phi_{\bar F})\vert_{V^{\bar{F}}_j} = t^{-j}\Rm{id}_{V^{\bar{F}}_j}\]
\tab Let $\gamma^{\DR} \coloneq (\gamma_{\bar{E}}^{\DR}, \gamma_{\bar{F}}^{\DR})$ acting on $R_{\DR}^{\pair}$. If we view $f$ as a matrix, with respect to the framings, then $\gamma^{\DR}$ acts by $f \mapsto \gamma_{\bar{F}}^{\DR} \circ f \circ (\gamma_{\bar{E}}^{\DR})^{-1}$. Notice that since $f(S_i\bar{E}) \subseteq S_i \bar{F}$, $f^{\Rm{ss}}$, in its matrix form, has to be block diagonal, hence it's fixed by this $\gamma^{\DR}$-action. Thus $\gamma^{\DR}$ acts on $\widehat{R}^{\trip}_{\DR, \bar{E}^{\Rm{ss}}, \bar{F}^{\Rm{ss}}, f^{\Rm{ss}}}$. \\
\tab On the other hand, since this is just a cocharacter of $\Rm{GL}_r \times \Rm{GL}_s$, it also acts, by $\gamma^{\Rm{Dol}}$, on $\widehat{R}^{\trip}_{\Dol}$. The Mochizuki's correspondence $\Sf{SM}^{\Rm{ps}}$ fixes the fiber, i.e. $\bar{E}_x = \Sf{SM}^{\Rm{ps}}(\bar{E})_x$, and it preserves the socle filtration, so $\gamma^{\Rm{Dol}}$ fixes $\Sf{SM}^{\Rm{ps}}(\bar{E}^{\Rm{ss}}, \bar{F}^{\Rm{ss}}, f^{\Rm{ss}})$. Thus it acts on $\widehat{R}^{\trip}_{\Dol, \Sf{SM}^{\Rm{ps}}(\bar{E}^{\Rm{ss}}, \bar{F}^{\Rm{ss}}, f^{\Rm{ss}})}$. 

\begin{prop}
  \label{lem:Rees-equivariance-formal-triples}
  The formal isomorphism $\Psi_{\bar{E}^{\Rm{ss}}, \bar{F}^{\Rm{ss}}, f^{\Rm{ss}}}$ in \cref{thm:formal-isom-triples-and-composition} is equivariant. 
\end{prop}
\begin{proof}
  For artinian local ring $A$, we have by definition
  \[\Psi_{\bar{E}^{\Rm{ss}}, \bar{F}^{\Rm{ss}}, f^{\Rm{ss}}}(\bar{E}^{\Rm{ss}}_A, \bar{F}^{\Rm{ss}}_A, f_A) = (\tau_{E^{\Rm{ss}}, F^{\Rm{ss}}}(\bar{E}^{\Rm{ss}}_A, \bar{F}^{\Rm{ss}}_A), H^0(\Phi_A)(f_A))\]
  and $\tau_{E^{\Rm{ss}}, F^{\Rm{ss}}}$ is equivariant by \cref{thm:Bett-DR-Dol-formal-isom}. Notice that the $\gamma^{\DR}$ acts by conjugation on $\ul{\Hom}_A(\bar{E}^{\Rm{ss}}_A, \bar{F}^{\Rm{ss}}_A)$, hence by conjugation on the logarithmic De Rham complex. Equivariance of $\tau_{E^{\Rm{ss}}, F^{\Rm{ss}}}$ and functoriality of $\Phi_A$ then gives a commutative diagram 
  \[\begin{tikzcd}
	{K_{\DR, A}(\bar{E}^{\Rm{ss}}_A, \bar{F}^{\Rm{ss}}_A)} && {K_{\Dol, A}(\tau_{E^{\Rm{ss}}, F^{\Rm{ss}}}(\bar{E}^{\Rm{ss}}_A, \bar{F}^{\Rm{ss}}_A))} \\
	\\
	{\gamma^{\DR}(t) \cdot K_{\DR, A}(\bar{E}^{\Rm{ss}}_A, \bar{F}^{\Rm{ss}}_A)} && {\gamma^{\Dol}(t) \cdot K_{\Dol, A}(\tau_{E^{\Rm{ss}}, F^{\Rm{ss}}}(\bar{E}^{\Rm{ss}}_A, \bar{F}^{\Rm{ss}}_A))}
	\arrow["{\Phi_A}", from=1-1, to=1-3]
	\arrow["{\gamma_{\bar{F}^{\Rm{ss}}}^{\DR}(t) \circ (-) \circ \gamma_{\bar{E}^{\Rm{ss}}}^{\DR}(t)^{-1} }"', from=1-1, to=3-1]
	\arrow["{\gamma_{\bar{F}^{\Rm{ss}}}^{\Dol}(t) \circ (-) \circ \gamma_{\bar{E}^{\Rm{ss}}}^{\Dol}(t)^{-1}}", from=1-3, to=3-3]
	\arrow["{\Phi_A}", from=3-1, to=3-3]
\end{tikzcd}\]
  hence $H^0(\Phi_A)$ is equivariant. 
\end{proof}

\section{Construction of the equivalence functor and its properties}
\subsection{Equivariant complex and algebraization}
Let $(G, m)$ be an algebraic group over $\CC$, and $B$ be a finite type scheme over $\CC$ with a $G$-action: 
\[a: G \times B \to B, \tab \pi: G \times B \to B\]
where $a$ specifies the action, and $\pi$ is the projection. 
\begin{definition}[{\cite[\href{https://stacks.math.columbia.edu/tag/03LF}{Tag 03LF}]{stacks-project}}]
  A $G$-equivariant quasicoherent $\sh{O}_B$-module on $B$ is a pair $(\sh{F}, \alpha)$ where $\sh{F} \in \Rm{QCoh}(\sh{O}_B)$ and $\alpha$ is a $\sh{O}_{G \times B}$-linear isomorphism $\alpha: a^*\sh{F} \xrightarrow{\simeq} \pi^*\sh{F}$ such that the diagram
  \[\begin{tikzcd}
	{(\Rm{id}_G \times a)^*\pi^*\sh{F}} && {\pi_2^*\sh{F}} \\
	\\
	{(\Rm{id}_G \times a)^* a^*\sh{F}} && {(m \times \Rm{id}_B)^*a^*\sh{F}}
	\arrow["{\pi_{12}^*\alpha}", from=1-1, to=1-3]
	\arrow["{(\Rm{id}_G \times a)^* \alpha}", from=3-1, to=1-1]
	\arrow[no head, from=3-1, to=3-3]
	\arrow[shift left, no head, from=3-1, to=3-3]
	\arrow["{(m \times \Rm{id}_B)^* \alpha}"', from=3-3, to=1-3]
\end{tikzcd}\]
commutes in the category of $\sh{O}_{G \times G \times B}$-modules, where $\pi_2: G \times G \times B \to B, \pi_{12}: G \times G \times B \to G \times B$ are projections. 
\end{definition}
\begin{definition}
  \leavevmode
  \begin{enumerate}
    \item A $G$-equivariant complex of $\sh{O}_B$-modules is a cochain complex $(K^\bullet, \der^\bullet)$ where each $K^i$ is a $G$-equivariant quasicoherent $\sh{O}_B$-module with $\alpha_i: a^* K^i \xrightarrow{\simeq} \pi^* K^i$, and each $\der^i$ is $G$-equivariant
    \[\alpha_{i + 1} \circ (a^* \der^i) = (\pi^* \der^i) \circ \alpha_i\]
    \item Let $Z$ be a $\CC$-scheme with a trivial $G$-action, and $\Rm{pr}_1: Z \times B \to B$ be the projection. Denote
    \begin{align*}
      \tilde{a} &\coloneq \Rm{id}_Z \times a: G \times Z \times B \to Z \times B\\
      \tilde{\pi} &\coloneq \Rm{id}_Z \times \pi: G \times Z \times B \to Z \times B
    \end{align*}
    the action and projection on $Z \times B$. A $G$-equivariant complex of $\Rm{pr}^{-1}_1\sh{O}_B$-modules is a cochain complex $(K^\bullet, \der^\bullet)$ where each $K^i$ is a $G$-equivariant $\sh{O}_{Z \times B}$-module, and 
    \[\alpha_{i + 1} \circ \tilde{a}_{B}^* \der^i = (\tilde{\pi}_B^*\der^i) \circ \alpha_i\]
    where $\tilde{a}_B^*\sh{F} \coloneq \Rm{pr}_{02}^{-1}\sh{O}_{G \times B} \otimes_{\tilde{a}^{-1} \Rm{pr}_1^{-1}\sh{O}_B}\tilde{a}^{-1} \sh{F}$, with $\tilde{\pi}_B^*$ similarly defined. 
  \end{enumerate}
\end{definition}

\begin{lemma}
  \label{lem:equivariant-pullback}
  \leavevmode
  \begin{enumerate}
    \item Let $Y$ be a $\CC$-scheme with a $G$-action. Let $f: Y \to B$ be a $G$-equivariant morphism, and $(K^\bullet, \der^i)$ a $G$-equivariant complex of $\sh{O}_Y$-modules (meaning each $\der^i$ is $\sh{O}_B$-linear). Then $f^* K^\bullet$ is a $G$-equivariant complex of $\sh{O}_Y$-modules. 
    \item Let $T$ be another $\CC$-scheme with a $G$-action, and $g: T \to B$ a $G$-equivariant morphism. Form the fibred diagram 
    \[\begin{tikzcd}
	{Z \times T} && Z \times B \\
	\\
	T && B
	\arrow["{\Rm{id}_Z \times g}", from=1-1, to=1-3]
	\arrow["{\Rm{pr}_1}"', from=1-1, to=3-1]
	\arrow["{\Rm{pr}_1}", from=1-3, to=3-3]
	\arrow["g", from=3-1, to=3-3]
\end{tikzcd}\]
  Suppose that $(K^\bullet, \der^i)$ is a $G$-equivariant complex of $\Rm{pr}_1^{-1}\sh{O}_B$-modules on $Z \times B$. Then 
  \[g_B^* K^\bullet \coloneq \Rm{pr}_1^{-1}\sh{O}_T \otimes_{(\Rm{id}_Z \times g)^{-1}\Rm{pr}_1^{-1}\sh{O}_B} (\Rm{id}_Z \times g)^{-1} K^\bullet\]
  is a $G$-equivariant complex of $\Rm{pr}_1^{-1}\sh{O}_T$-modules. 
  \end{enumerate}
\end{lemma}
\begin{proof}
  See \cite[\href{https://stacks.math.columbia.edu/tag/03LG}{Tag 03LG}]{stacks-project}.
\end{proof}
Let $\Box \in \set*{\DR, \Dol}$. Recall that we have the moduli of pairs $R_{\Box}^{\pair}(r, s)$. On $\bar{X} \times R_{\Box}^{\pair}(r, s)$ we have universal bundles $\bar{E}^{\univ}_{\Box}, \bar{F}^{\univ}_{\Box}$, with the logarithmic complex $C_{\Box}(\bar{E}^{\univ}_{\Box}, \bar{F}^{\univ}_{\Box})$. We have a $G \coloneq \Rm{GL}_r \times \Rm{GL}_s$-action on $\bar{X} \times R_{\Box}^{\pair}(r, s)$ which acts trivially on $\bar{X}$ and by change of framings on $R_{\Box}^{\pair}(r, s)$. Let $\Rm{pr}_1: \bar{X} \times R_{\Box}^{\pair}(r, s)$ be the projection.

\begin{prop}
  $C_{\Dol}(\bar{E}_{\Dol}^{\univ}, \bar{F}_{\Dol}^{\univ})$ is a $G$-equivariant complex of $\sh{O}_{\bar{X} \times R_{\Dol}^{\pair}(r, s)}-$modules. $C_{\DR}(\bar{E}_{\DR}^{\univ}, \bar{F}_{\DR}^{\univ})$ is a $G-$equivariant complex of $\Rm{pr}_1^{-1}\sh{O}_{R_{\DR}^{\pair}(r, s)}$-modules. $K_{\Box}(\bar{E}^{\univ}_{\Box}, \bar{F}^{\univ}_{\Box}) = \Rm{R}(\Rm{pr}_1)_* C_{\Box}(\bar{E}^{\univ}_{\Box}, \bar{F}^{\univ}_{\Box})$ is a $G$-equivariant complex of $\sh{O}_{R_{\Box}^{\pair}(r, s)}$-modules. 
\end{prop}
\begin{proof}
  Let's prove the first statement; the second one uses the same argument just with different linear structures. Let $a: G \times R_{\Dol}^{\pair}(r, s) \to R_{\Dol}^{\pair}(r, s), \pi: G \times R_{\Dol}^{\pair}(r, s) \to R_{\Dol}^{\pair}(r, s)$ be the action and projection on $R_{\Dol}^{\pair}(r, s)$. Denote $\tilde{a} = \Rm{id}_{\bar{X}} \times a, \tilde{\pi} = \Rm{id}_{\bar{X}} \times \pi$ the action (trivial on $\bar{X}$) and projection on $\bar{X} \times R_{\Dol}^{\pair}(r, s)$. Notice that the $G$-action only changes the framings, not the underlying bundles, so if we forget about framing (which is fine since the complexes don't care about framing), uniqueness of universal bundle gives canonical isomorphisms of logarithmic Higgs bundles 
  \[\alpha_{\bar{E}}: \tilde{a}^*\bar{E}^{\univ}_{\Dol} \xrightarrow{\simeq} \tilde{\pi}^*\bar{E}^{\univ}_{\Dol}, \tab \alpha_{\bar{F}}: \tilde{a}^*\bar{F}^{\univ}_{\Dol} \xrightarrow{\simeq} \tilde{\pi}^*\bar{F}^{\univ}_{\Dol}\]
  which satisfies the cocycle condition because change-of-framing on $R_{\Dol}^{\pair}(r, s)$ just induces a change-of-framing on the universal bundles. This induces, by functoriality, an isomorphism
  \[\alpha_{\Hom}: \tilde{a}^* \shHom(\bar{E}^{\univ}_{\Dol}, \bar{F}^{\univ}_{\Dol}) \xrightarrow{\simeq} \tilde{\pi}^*\shHom(\bar{E}^{\univ}_{\Dol}, \bar{F}^{\univ}_{\Dol}), \tab \phi \mapsto \alpha_{\bar{F}} \circ \phi \circ \alpha_{\bar E}^{-1}\]
  which satisfies the cocycle condition because $\alpha_{\bar{E}}, \alpha_{\bar{F}}$ do. Thus $\shHom(\bar{E}^{\univ}_{\Dol}, \bar{F}^{\univ}_{\Dol})$ is $G$-equivariant via $\alpha_{\Hom}$, and the remaining term of the complex
  \[\shHom(\bar{E}^{\univ}_{\Dol}, \bar{F}^{\univ}_{\Dol}) \otimes \Omega^1_{\bar{X} \times R_{\Dol}^{\pair}(r, s)/R_{\Dol}^{\pair}(r, s)}(\log D \times R_{\Dol}^{\pair}(r, s))\] 
  is equivariant via $\alpha_{\Hom} \otimes 1$. Since $\alpha_{\bar{E}}$ and $\alpha_{\bar{F}}$ are isomorphisms of Higgs bundles, they commute with the Higgs fields:
  \[(\alpha_{\bar E} \otimes 1) \circ (\tilde{a}^*\theta_{\bar{E}}) = (\tilde{\pi}^*\theta_{\bar{E}}) \circ \alpha_{\bar E}, \tab (\alpha_{\bar F} \otimes 1) \circ (\tilde{a}^*\theta_{\bar{F}}) = (\tilde{\pi}^*\theta_{\bar{F}}) \circ \alpha_{\bar F}\]
  \tab We also have that the induced Higgs field on $\shHom(\bar{E}^{\univ}_{\Dol}, \bar{F}^{\univ}_{\Dol})$ is 
  \[\theta_{\Hom}(\phi) = \theta_{\bar{F}} \circ \phi - (\phi \otimes 1) \circ \theta_{\bar{E}}\]
  and, by plugging in, we have that $(\alpha_{\Hom} \otimes 1) \circ \tilde{a}^*\theta_{\Hom} = \tilde{\pi}^*\theta_{\Hom} \circ \alpha_{\Hom}$. Thus $\theta_{\Hom}$ is $G$-equivariant. \\
  \tab For the last statement, we can use base change 
  \[a^*K_{\Box} = a^* \Rm{R}(\Rm{pr}_1)_*C_{\Box} \simeq \Rm{R}(\Rm{pr}_{12})_* \tilde{a}^* C_{\Box} \simeq \Rm{R}(\Rm{pr}_{12})_*\tilde{\pi}^*C_{\Box} \simeq \pi^* \Rm{R}(\Rm{pr}_1)_*C_{\Box} \simeq \pi^* K_{\Box}\]
  and functoriality preserves the cocycle diagram. 
\end{proof}
From now on, consider the scaling $\GG_m$-action on $\AAA^1_t$. 
\begin{lemma}
  \label{lem:Gm-equiv-complex-A1-graded-free}
  Let $K^\bullet$ be a $\GG_m$-equivariant complex of $\sh{O}_{\AAA^1_t}$-modules. Suppose that $K^\bullet$ is perfect. Then $K^\bullet$ is equivariantly quasi-isomorphic to a bounded complex of finite rank graded free $\CC[t]$-modules 
  \[K^i \simeq \bigoplus_j \CC[[t]](w_{ij})\]
  with degree 0 differentials $\der^i$.
\end{lemma}
\begin{proof}
  $\GG_m$-equivariant quasicoherent $\sh{O}_{\AAA^1_t}$-modules are just graded $\CC[t]$-modules. The category of graded $\CC[t]-$modules have enough projectives, and perfectness gives bounded graded projective resolution. Finally, $\CC[t]$ is graded-local, so by \cite[Proposition 1.5.15(d)]{bruns1998cohen}, graded projective $\CC[t]$-modules are free. 
\end{proof}
\begin{lemma}
  \label{lem:Gm-equiv-formal-arc-algebraization}
  Let $Y$ be a finite type $\CC$-scheme with a $\GG_m$-action. Let $y_0 \in Y$ be a fixed point which has an invariant affine neighborhood. Then we have a bijection
  \[\set*{\GG_m-\text{equivariant }(\AAA^1_t, 0) \to (Y, y_0)} \xrightarrow{\simeq} \set*{\GG_m-\text{equivariant }(\Rm{Spf}\ \CC[[t]], 0) \to (\widehat{Y}_{y_0}, y_0)}\]
\end{lemma}
\begin{remark}
  We will apply this lemma only to $Y = R_{\DR/\Dol}, R_{\DR/\Dol}^{\trip}$ with $y_0$ either a polystable bundle or a triple $(\bar{E}, \bar{F}, f)$ with $\bar{E}, \bar{F}$ polystable. In the first case, by GIT theory we know that the existence of a $\Rm{GL}_r$-invariant affine neighborhood is guaranteed. For the second case, notice that $R_{\DR/\Dol}^{\trip} \to R_{\DR/\Dol}^{\pair}$ is affine, so we can pull back the $\Rm{GL}_r \times \Rm{GL}_s$-invariant affine neighborhood around $(\bar{E}, \bar{F}) \in R_{\DR/\Dol}^{\pair}$.   
\end{remark}
\begin{proof}
  Let $U = \Spec B$ be a $\GG_m$-invariant affine neighborhood containing $y_0$, and $\Ide{m}_0 \subset B$ the maximal ideal corresponding to $y_0$. Let the $\GG_m$-action on $U$ be denoted by $\gamma: B \to B \otimes_{\CC} \CC[u, u^{-1}]$, which is equivalent to a grading $B = \oplus_{m \in \ZZ} B_m$ where if $\gamma(b) = \sum b_n \otimes u^n$ then $b_n \in B_n$. This actions completes to $\hat{\gamma}: \widehat{B}_{\Ide{m}_0} \to \widehat{B}_{\Ide{m}_0} \widehat{\otimes}_{\CC} \CC[u, u^{-1}]$. Let $\widehat{\alpha}: \widehat{B}_{\Ide{m}_0} \to \CC[[t]]$ be a $\hat{\gamma}$-equivariant arc, i.e.,
  \[\begin{tikzcd}
	{\widehat{B}_{\Ide{m}_0}} && {\widehat{B}_{\Ide{m}_0} \widehat{\otimes}_{\CC} \CC[u, u^{-1}]} \\
	\\
	{\CC[[t]]} && {\CC[[t]] \widehat{\otimes} \CC[u, u^{-1}]}
	\arrow["{\hat{\gamma}}", from=1-1, to=1-3]
	\arrow["{\widehat{\alpha}}"', from=1-1, to=3-1]
	\arrow["{\widehat{\alpha} \otimes \Rm{id}}", from=1-3, to=3-3]
	\arrow["{t \mapsto t\otimes u}", from=3-1, to=3-3]
  \end{tikzcd}\]
  \tab Denote $\iota: B \to \widehat{B}_{\Ide{m}_0}$. For $b \in B_m$, $\hat{\gamma}(\iota(b)) = \iota(b) \otimes u^m$, so $\widehat{\alpha}(\iota(b)) \in \CC[[t]]_m$. It follows that $\alpha \coloneq \widehat{\alpha} \circ \iota: B \to \CC[[t]]$ has image contained in $\CC[t]$, hence giving an algebraic arc. Uniqueness of $\alpha$ trivial. 
\end{proof}

\begin{lemma}
  \label{lem:Gm-equiv-formal-complex-algebraization}
  Let $K^\bullet, L^\bullet$ be perfect $\GG_m$-equivariant complexes on $\AAA^1_t$. We can assume that $K^\bullet, L^\bullet$ are both graded free by \cref{lem:Gm-equiv-complex-A1-graded-free}, and denote 
  \[\hat{K}^\bullet \coloneq K^\bullet \otimes_{\CC[t]} \CC[[t]], \tab \hat{L}^\bullet \coloneq L^\bullet \otimes_{\CC[t]} \CC[[t]]\]
  be the restriction to the formal neighborhood of $\AAA^1_t$ at 0. Then every continuous $\GG_m$-equivariant morphism $\hat{f}: \hat{K}^\bullet \to \hat{L}^\bullet$ algebraizes uniquely to a $\GG_m$-equivariant $f: K^\bullet \to L^\bullet$. Furthermore, if $\hat{f}$ is a quasi-isomorphism, so is $f$. In particular $f\vert_{t = 1}: K^\bullet\vert_{t = 1} \to L^\bullet\vert_{t = 1}$ is a quasi-isomorphism. 
\end{lemma}
\begin{proof}
  Notice that a $\GG_m$-equivariant map is just a degree 0 map between graded free modules here. For two graded free summands $\CC[[t]](w)$ and $\CC[[t]](l)$, a degree 0 map $\CC[[t]](w) \to \CC[[t]](l)$ is equivalent to multiplication by a homogeneous element in $\CC[[t]]_{l - w}$. The same is true for $\CC[t]$, and notice that $\CC[t]_{m} = \CC[[t]]_{m}$ for all $m \in \ZZ$, so completion induces an isomorphisms of degree 0 parts
  \[\Hom_{\CC[t]}^\bullet(K^\bullet, L^\bullet)_0 \xrightarrow{\simeq} \Hom_{\CC[[t]], \Rm{cont}}^\bullet(\hat{K}^\bullet, \hat{L}^\bullet)_0\]
  and taking $H^0$ gives the algebraization statement. \\
  \tab Now suppose that $\hat{f}$ is a quasi-isomorphism. Put $C^\bullet \coloneq \Rm{Cone}(f)$, then $\hat{C}^\bullet = \Rm{Cone}(\hat{f})$ is acyclic. The support of $H^j(C^\bullet)$, which is a graded $\CC[t]$-module, is a closed $\GG_m$-invariant subset of $\AAA^1_t$, so it's either 0 or $\AAA^1$ or $H^j(C^\bullet) = 0$. Since $\CC[t]_{(t)} \to \CC[[t]]$ is faithfully flat, and $H^j(C^\bullet)_{(t)} \otimes_{\CC[t]_{(t)}} \CC[[t]] \simeq H^j(\hat{C}^\bullet) = 0$, we get that $H^j(C^\bullet)_{(t)} = 0$ for all $j$. Thus 0 is not in the support of $H^j(C^\bullet)$, hence $H^j(C^\bullet) = 0$ for all $j$. We get that $C^\bullet$ has to be acyclic; equivalently, $f$ is a quasi-isomorphism. 
\end{proof}

\subsection{Construction on objects}
Fix $x \in X = \bar{X} \setminus D$. We will first define the functor 
\[\Sf{SM}: \Sf{Conn}_{\Rm{nilp}}(\bar{X}, D) \to \Sf{Higgs}_{\nilp}^{0, \Rm{ss}}(\bar{X}, D)\]
extending $\Sf{SM}^{\Rm{ps}}$ on objects. Let $\bar{E} \in \Sf{Conn}_{\Rm{nilp}}(\bar{X}, D)$, pick a framing $\phi: \bar{E}_x \to \CC^r$ and a splitting of $\bar{E}_x$ adapted to the socle filtration $S_\bullet \bar{E}$. Let $\bar{E}^{\Rm{ss}} \coloneq \Rm{gr}_{S} \bar{E}$. We then get a Rees cocharacter $\gamma_{\bar E}: \GG_m \to \Rm{GL}_r$ which induces a $\GG_m$-action $\gamma_{\bar E}^{\DR}$ on $R_{\DR}(r)$; this $\GG_m$-action fixes $\bar{E}^{\Rm{ss}}$ hence induces an action on $\widehat{R}_{\DR, \bar{E}^{\Rm{ss}}}$. The same $\gamma_{\bar{E}}$ acts on $R_{\Dol}(r)$ and $\widehat{R}_{\Dol, \Sf{SM}^{\Rm{ps}}(\bar{E}^{\Rm{ss}})}$ by $\gamma_{\bar{E}}^{\Dol}$. We also have a $\GG_m$-equivariant arc 
\[\alpha_{\bar{E}}^{\DR}: \AAA^1_t \to R_{\DR}(r), \tab \alpha_{\bar{E}}^{\DR}(1) = (\bar{E}, \phi), \tab \alpha_{\bar E}^{\DR} = (\bar{E}^{\Rm{ss}}, \Rm{gr}_S \phi)\]
\tab Complete at $t = 0$ to get an equivariant formal arc $\widehat{\alpha}_{\bar E}^{\DR}: \Rm{Spf}\ \CC[[t]] \to \widehat{R}_{\DR, \bar{E}^{\Rm{ss}}}$, and then compose with the Rees equivariant isomorphism \cref{thm:Bett-DR-Dol-formal-isom},
\[\tau_{E}: \widehat{R}_{\DR, \bar{E}^{\Rm{ss}}} \xrightarrow{\simeq} \widehat{R}_{\Dol, \Sf{SM}^{\Rm{ps}}(\bar{E}^{\Rm{ss}})}\]
\tab By \cref{lem:Gm-equiv-formal-arc-algebraization}, we can algebraize $\tau_E \circ \widehat{\alpha}_{\bar{E}}^{\DR}$ to get an algebraic arc $\alpha_{\bar{E}}^{\Dol}: \AAA^1_t \to R_{\Dol}(r)$. Define $\Sf{SM}(\bar{E}) = \alpha_{\bar{E}}^{\Dol}(1)$ forgetting the framing. Pulling back the universal family along $\alpha_{\bar{E}}^{\Dol}$ gives a $\GG_m$-equivariant framed Higgs bundle on $\bar{X} \times \AAA^1$, which is equivalent to a Rees family associated to a filtration $F_{\bar{E}}$ by Higgs subbundles with $\Rm{gr}^{F_{\bar E}}\Sf{SM}(\bar{E}) = \alpha_{\bar{E}}^{\Dol}(0) = \Sf{SM}^{\Rm{ps}}(\bar{E}^{\Rm{ss}})$. By the following proposition, $\alpha_{\bar{E}}^{\Dol}(1) \in \Sf{Higgs}_{\nilp}^{0, \Rm{ss}}(\bar{X}, D)$. 
\begin{prop}
  $\Sf{SM}(\bar{E})$ has nilpotent residue. 
\end{prop}
\begin{proof}
  Since $\Rm{gr}^{F_{\bar E}}\Sf{SM}(\bar{E}) = \Sf{SM}^{\Rm{ps}}(\bar{E}^{\Rm{ss}})$, every graded summand of $F_{\bar{E}}$ has nilpotent residues. $F_{\bar{E}}$ is a filtration by sub Higgs bundles, i.e., $\theta(F_i) \subset F_i \otimes \Omega^1_{\bar{X}}(\log D)$, so any residue $N_p$ around $p \in D$ must also preserve $F_{\bar{E}}$. Pick a basis at $p$ adapted to $F_{\bar{E}}$, then $N_p$ is a block upper triangular matrix, with diagonal blocks nilpotent, so $N_p$ must be nilpotent as well. 
\end{proof}

\begin{prop}
  \label{lem:independence-framing-Rees-cochar}
  With the framing fixed, two different splittings of the socle filtration give the same framed bundle $\alpha_{\bar{E}}^{\Dol}(1)$. Changing the framing induces a change-of-framing action on $\alpha_{\bar{E}}^{\Dol}(1)$. After forgetting the framing, the resulting bundle is well defined up to a canonical isomorphism. 
\end{prop}
\begin{proof}
  Fix a framing $\phi: \bar{E}_x \to \CC^r$. Let $s, s': \bigoplus_{j} \Rm{gr}^S_j \bar{E}_x \xrightarrow{\simeq} \bar{E}_x$ be two different framings, and write $\bar{E}_x = s(\bigoplus_{j} \Rm{gr}^S_j \bar{E}_x) \eqcolon \bigoplus_j G_j$. Put $\upsilon \coloneq s' \circ s^{-1}: \bar{E}_x \to \bar{E}_x$. With respect to a basis adapted to $\bar{E}_x = \bigoplus_j G_j$, write $\upsilon$ in block form $(\upsilon_{i,j})$ where $\upsilon_{i, j}: G_j \to G_i$. Notice that for $w \in G_j$, $s^{-1}(w) \in \Rm{gr}_j^S \bar{E}_x$, thus $s'(w) \in S_j\bar{E}_x$, and $s'(w)\ \Rm{mod}\ S_{j - 1}\bar{E}_x = w$ hence
  \[\upsilon_{i, j} = 0 \tab \text{for }i > j, \tab u_{i, i} = \Rm{id}_{G_i}\]
  \tab Recall that for the splitting $s$, the Rees cocharacter $\gamma_s: \GG_m \to \Rm{GL}(\bar{E}_x)$ is $\gamma_s(t)\vert_{G_j} = t^{-j}\cdot \Rm{id}_{G_j}$. Define $\varphi_{\upsilon}: \GG_m \to \Rm{GL}(\bar{E}_x)$ by $\varphi_{\upsilon}(t) \coloneq \gamma_s(t) \circ \upsilon \circ \gamma_s(t)^{-1}$. In a basis adapted to $\bar{E}_x = \bigoplus_j G_j$, we have the block form $(\varphi_{\upsilon}(t))_{i, j} = t^{j - i} \upsilon_{i, j}$. All exponents are nonnegative, $\varphi_{\upsilon}(0) = \Rm{id}$, $\varphi_{\upsilon}(t) - \Rm{id}$ is nilpotent, so by expanding $(\Rm{id} + (\varphi_{\upsilon}(t) - \Rm{id}))^{-1}$ we get a polynomial inverse $\varphi_{\upsilon}(t)^{-1}$. Define $g_{\upsilon}: \GG_m \to \Rm{GL}(\bar{E}_x)$ by $g_{\upsilon}(t) \coloneq \upsilon \circ \varphi_{\upsilon}(t)^{-1}$. Since $\gamma_{s'} = \upsilon \circ \gamma_s \circ \upsilon^{-1}$, we have 
  \[g_{\upsilon}(0) = \upsilon, \tab g_{\upsilon}(1) = \Rm{id}, \tab g_{\upsilon}(t' \cdot t) \gamma_s(t') = \gamma_{s'}(t') g_{\upsilon}(t),
  \tab \alpha^{\DR}_{\bar{E}, s'}(t) = g_{\upsilon}(t) \cdot \alpha_{\bar{E}, s}^{\DR}(t)\]
  \tab For $A_n = \CC[t]/(t^{n + 1})$, the change-of-framing of the $A_n$-point is $g_{\upsilon}(t)\ \Rm{mod}\ t^{n + 1}$. Functoriality of \eqref{eq:formal-DR-Dol-immersion} (follows from the same diagrams in the proof of part (3)) gives that 
  \[\tau_{(E^{\Rm{ss}}, \phi \circ s')} \circ \widehat{\alpha}_{\bar{E}, s'}^{\DR}(t) = g_{\upsilon}(t) \cdot (\tau_{(E^{\Rm{ss}}, \phi \circ s)} \circ \widehat{\alpha}_{\bar{E}, s}^{\DR}(t))\] 
  which are both $\gamma_{s'}$-equivariant. By uniqueness of algebraization,
  \[\alpha_{\bar{E}, s'}^{\Dol}(t) = g_{\upsilon}(t) \cdot\alpha_{\bar{E}, s}^{\Dol}(t)\]
  and since $g_{\upsilon}(1) = 1$ we get the first statement. The change of framing is a matrix $T$; it is trivial to check that in each step of taking the Rees arc, formal isomorphism $\tau_{E^{\Rm{ss}}}$, and algebraization, we are just conjugating by this fixed matrix. Hence $\alpha_{(\bar{E}, T \circ \phi)}^{\Dol}(t) = T \cdot \alpha_{(\bar{E}, \phi)}^{\Dol}(t)$. Forgetting the framing gives a canonical isomorphism (coming from that of framed bundles) of the underlying bundles at $t = 1$. 
\end{proof}
\subsection{Construction on morphisms}
Now consider $\bar{E}, \bar{F} \in \Sf{Conn}_{\nilp}(\bar{X}, D)$. We have $\GG_m$-equivariant arc of pairs $\alpha_{\bar{E}, \bar{F}}^{\DR} = (\alpha_{\bar E}^{\DR}, \alpha_{\bar F}^{\DR})$ with a diagram  
\[\begin{tikzcd}
	{\Rm{Spf}\ \CC[[t]]} && {R_{\DR, (\bar{E}^{\Rm{ss}}, \bar{F}^{\Rm{ss}})}^{\pair}} \\
	\\
	&& {R_{\Dol, \Sf{SM}^{\Rm{ps}}(\bar{E}^{\Rm{ss}}, \bar{F}^{\Rm{ss}})}^{\pair}}
	\arrow["{\widehat{\alpha}_{\bar{E}, \bar{F}}^{\DR}}", from=1-1, to=1-3]
	\arrow["{\widehat{\alpha}_{\bar{E}, \bar{F}}^{\Dol}}"', from=1-1, to=3-3]
	\arrow["{\tau_{\bar{E}^{\Rm{ss}}, \bar{F}^{\Rm{ss}}}}", from=1-3, to=3-3]
\end{tikzcd}\]
where we can then algebraize $\widehat{\alpha}_{\bar{E}, \bar{F}}^{\Dol}$ to get $\alpha_{\bar{E}, \bar{F}}^{\Dol}: \AAA^1_t \to R^{\pair}_{\Dol}(r, s)$. Let $K^{\univ}_{\Box}$ be the universal complex on $R^{\pair}_{\Dol}(r, s)$, and define 
\[P^{\Box}_{\bar{E}, \bar{F}} \coloneq \Rm{L}(\alpha^{\Box}_{\bar{E}, \bar{F}})^* K_{\Box}^{\univ}\]
then base change gives, for $\iota_t$ the restriction on $\AAA^1_t$,
\begin{align*}
  \Rm{L}\iota_1^* P^{\DR}_{\bar{E}, \bar{F}} &= \Rm{R}\Gamma(\bar{X}, C_{\DR}(\bar{E}, \bar{F}))\\
  \Rm{L}\iota_1^* P^{\Dol}_{\bar{E}, \bar{F}} &= \Rm{R}\Gamma(\bar{X}, C_{\Dol}(\Sf{SM}(\bar{E}), \Sf{SM}(\bar{F})))
\end{align*}
\tab For $A_n \coloneq \CC[t]/(t^{n + 1})$, pull back the quasi-isomorphism \eqref{eq:nth-level-DR-Dol-cohomology-comparison} along $\widehat{\alpha}^{\DR}_{\bar{E}, \bar{F}} \vert_{A_n}$, and compatibility with $A_{n + 1} \to A_n$ then gives a quasi-isomorphism
\[\widehat{\Phi}^{\alpha}_{\bar{E}, \bar{F}}: \widehat{P}^{\DR}_{\bar{E}, \bar{F}} \xrightarrow{\simeq} \widehat{P}^{\Dol}_{\bar{E}, \bar{F}}\]
which is equivariant, by \cref{lem:equivariant-pullback}, with respect to $\gamma^{\DR}, \gamma^{\Dol}$. By \cref{lem:Gm-equiv-formal-complex-algebraization}, we get an equivariant quasi-isomorphism 
\[\Phi^{\alpha}_{\bar{E}, \bar{F}}: P^{\DR}_{\bar{E}, \bar{F}} \xrightarrow{\simeq} P^{\Dol}_{\bar{E}, \bar{F}}\]
and restricting to $t = 1$ we get 
\[\Phi_{\bar{E}, \bar{F}}^{\Sf{SM}}: \Rm{R}\Gamma(\bar{X}, C_{\DR}(\bar{E}, \bar{F})) \xrightarrow{\simeq} \Rm{R}\Gamma(\bar{X}, C_{\Dol}(\Sf{SM}(\bar{E}), \Sf{SM}(\bar{F}))) \tag{5.1} \label{eq:quasi-isom-pulled-back-A1}\]
\tab On the other hand, for $f: \bar{E} \to \bar{F}$, functoriality of the socle filtration gives $f_t \coloneq \Rm{Rees}_S(f)$ (which is a morphism over $\AAA^1_t$), and an equivariant arc of triples 
\[\alpha_f^{\DR}: \AAA^1_t \to R^{\trip}_{\DR}(r, s), \tab \alpha_f^{\DR}(1) = (\bar{E}, \bar{F}, f), \tab \alpha_f^{\DR}(0) = (\bar{E}^{\Rm{ss}}, \bar{F}^{\Rm{ss}}, f^{\Rm{ss}})\]
\tab Then by \cref{lem:Rees-equivariance-formal-triples} and \cref{lem:Gm-equiv-formal-arc-algebraization}, we get $\alpha_f^{\Dol}: \AAA^1_t \to R^{\trip}_{\Dol}(r, s)$ by algebraizing $\Psi_{\bar{E}^{\Rm{ss}}, \bar{F}^{\Rm{ss}}, f^{\Rm{ss}}} \circ \widehat{\alpha}_f^{\DR}$. The pair of bundles in $\alpha_f^{\Dol}(t)$ is exactly that of $\alpha_{\bar{E}, \bar{F}}^{\Dol}(t)$. For each $t \in \AAA^1_t$, let $s^{\Dol}_f(t)$ be the morphism of $\alpha_f^{\Dol}(t)$. 

\begin{lemma}
  \label{lem:morphism-formal-transport-and-equiv-algebraization}
  $s_f^{\Dol}(1) = H^0(\Phi^{\Sf{SM}}_{\bar{E}, \bar{F}})(f)$.
\end{lemma}
\begin{proof}
  Notice that over $\AAA^1_t$, $\Rm{Rees}_S\bar{E} = (\alpha_{\bar{E}}^{\DR})^*\bar{E}^{\univ}$ (same for $\bar{F}$) so by base change for logarithmic De Rham complex, $\Rm{Rees}_S f$ is an element of 
  \[\HH^0(\AAA^1_t, \Rm{L}(\alpha^{\DR}_{\bar{E}, \bar{F}})^* K_{\DR}^{\univ}) = \HH^0(\AAA^1_t, P^{\DR}_{\bar{E}, \bar{F}})\]
  hence by \cref{lem:relative-hom-scheme-lemma}, is equivalent to a section $s^{\DR}_f: \AAA^1_t \to M^0_{\AAA^1_t}(P^{\DR}_{\bar{E}, \bar{F}})$ with $s_f^{\DR}(1) = f$; in fact $s_f^{\DR}(t)$ is exactly the morphism in the triple $\alpha_f^{\DR}(t)$. Applying $M^0(\Phi^{\alpha}_{\bar{E}, \bar{F}})$ gives 
  \[s_f^{M}: \AAA^1_t \xrightarrow{M^0(\Phi^{\alpha}_{\bar{E}, \bar{F}}) \circ s_f^{\DR}} M_{\AAA^1_t}(P^{\Dol}_{\bar{E}, \bar{F}})\]
  \tab We have the following diagram, by definition of $\Psi_{\bar{E}^{\Rm{ss}}, \bar{F}^{\Rm{ss}}, f^{\Rm{ss}}}$,
  \[\begin{tikzcd}
	{\widehat{R}^{\trip}_{\DR, (\bar{E}^{\Rm{ss}}, \bar{F}^{\Rm{ss}}, f^{\Rm{ss}})}} && {\widehat{R}^{\trip}_{\Dol, \Sf{SM}^{\Rm{ps}}(\bar{E}^{\Rm{ss}}, \bar{F}^{\Rm{ss}}, f^{\Rm{ss}})}} \\
	\\
	{M^0_{\widehat{R}_{\DR}^{\pair}}(K^\univ_{\DR})_{f^{\Rm{ss}}}} && {M^0_{\widehat{R}_{\Dol}^{\pair}}(K^\univ_{\Dol})_{\Sf{SM}^{\Rm{ps}}(f^{\Rm{ss}})}}
	\arrow["{\Psi_{\bar{E}^{\Rm{ss}}, \bar{F}^{\Rm{ss}}, f^{\Rm{ss}}}}", from=1-1, to=1-3]
	\arrow["\simeq"', from=1-1, to=3-1]
	\arrow["\simeq", from=1-3, to=3-3]
	\arrow["{M^0(\Phi^{\univ}_{\bar{E}^{\Rm{ss}}, \bar{F}^{\Rm{ss}}})}", from=3-1, to=3-3]
  \end{tikzcd}\]
  and, by construction above, $M^0(\Phi^{\univ}_{\bar{E}^{\Rm{ss}}, \bar{F}^{\Rm{ss}}})$ pulls back to $M^0(\widehat{\Phi}^{\alpha}_{\bar{E}, \bar{F}})$, so $\widehat{s}_f^{\Dol} = \widehat{s}_f^{M}$. Uniqueness of equivariant algebraization then gives $s_f^{\Dol} = s_f^M$. In particular, 
  \[s_f^{\Dol}(1) = s_f^M(1) = H^0(\Phi^{\alpha}_{\bar{E}, \bar{F}}\vert_{t = 1})(s_f^{\DR}(1)) = H^0(\Phi^{\Sf{SM}}_{\bar{E}, \bar{F}})(f)\]
  as desired. 
\end{proof}
Now we will define the functor $\Sf{SM}$ on morphisms: $\Sf{SM}(f) \coloneq H^0(\Phi_{\bar{E}, \bar{F}}^{\Sf{SM}})(f)$.
\begin{theorem}
  \label{thm:functor-verification}
  $\Sf{SM}: \Sf{Conn}_{\Rm{nilp}}(\bar{X}, D) \to \Sf{Higgs}_{\nilp}^{0, \Rm{ss}}(\bar{X}, D)$ is a $\CC$-linear functor. 
\end{theorem}
\begin{proof}
  For this proof, we consider $\Sf{SM}(f) = \alpha_f^{\Dol}(1)$, which is harmless by the previous lemma. For composition $\bar{E} \xrightarrow{f} \bar{E} \xrightarrow{g} \bar{G}$, we have $\Rm{Rees}_S(g \circ f) = \Rm{Rees}_S(g) \circ \Rm{Rees}_S(f)$, and we have two different equivariant arcs $\alpha_{g \circ f}^{\DR}: \AAA^1 \to R^{\trip}_{\DR, g \circ f}$ and $(\alpha_f, \alpha_g): \AAA^1_t \to R^{\trip}_{\DR, f} \times R^{\trip}_{\DR, g}$. By \cref{thm:formal-isom-triples-and-composition}, we get a diagram, where $\mu_\Box$ are induced by compositions, 
  \[\begin{tikzcd}
	{\Rm{Spf}\ \CC[[t]]} &&&& \\
	{} && {\widehat{Z}_{\DR, f, g}} && {\widehat{Z}_{\Dol, \Sf{SM}^{\Rm{ps}}(f, g)}} \\
	{\Rm{Spf}\ \CC[[t]]} \\
	&& {\widehat{R}^{\trip}_{\DR, \bar E, \bar G, g \circ f}} && {\widehat{R}^{\trip}_{\Dol, \Sf{SM}^{\Rm{ps}}(\bar E, \bar G, g \circ f)}}
	\arrow["{(\widehat{\alpha}^{\DR}_f, \widehat{\alpha}^{\DR}_g)}", from=1-1, to=2-3]
	\arrow[no head, from=1-1, to=3-1]
	\arrow[shift left, no head, from=1-1, to=3-1]
	\arrow["{\Psi_{\bar{E}, \bar{F}, f} \times \Psi_{\bar{F}, \bar{G}, g}}", from=2-3, to=2-5]
	\arrow["{\mu_{\DR}}"', from=2-3, to=4-3]
	\arrow["{\mu_{\Dol}}", from=2-5, to=4-5]
	\arrow["{\widehat{\alpha}_{g \circ f}^{\DR}}", from=3-1, to=4-3]
	\arrow["{\Psi_{\bar{E}, \bar{G}, g \circ f}}", from=4-3, to=4-5]
  \end{tikzcd}\]
  so it follows that for the algebraized arcs, $\alpha^{\Dol}_{g \circ f}(1) = \alpha_g^{\Dol}(1) \circ \alpha_f^{\Dol}(1)$. Identity being preserved, as well as linearity, follows from the same argument using \cref{thm:formal-isom-triples-and-composition}.
\end{proof}

\subsection{Proof of equivalence}
\begin{prop}
  \label{lem:fully-faithfulness}
  $\Sf{SM}$ is fully faithful. 
\end{prop}
\begin{proof}
  Taking $H^0$ of \eqref{eq:quasi-isom-pulled-back-A1} we get
  \[H^0(\Phi_{\bar{E}, \bar{F}}^{\Sf{SM}}): \Hom_{\Sf{Conn}_{\Rm{nilp}}(\bar{X}, D)}(\bar{E}, \bar{F}) \xrightarrow{\simeq} \Hom_{\Sf{Higgs}_{\nilp}^{0, \Rm{ss}}(\bar{X}, D)}(\Sf{SM}(\bar{E}), \Sf{SM}(\bar{F}))\]
  and $\Sf{SM}(f) = H^0(\Phi_{\bar{E}, \bar{F}}^{\Sf{SM}})(f)$ so we are done. 
\end{proof}

\begin{prop}
  \label{lem:exactness-and-socle-preservation}
  $\Sf{SM}$ is exact and preserves the socle filtration. Moreover, for each $\bar{E} \in \Sf{Conn}_{\Rm{nilp}}(\bar{X}, D)$, the filtration $F_{\bar{E}}$ on $\Sf{SM}(\bar{E})$ coming from the algebraized Rees arc $\alpha_{\bar{E}}^{\Dol}$ is the socle filtration of $\Sf{SM}(\bar{E})$.
\end{prop}
\begin{proof}
  $\Sf{SM}$ is fully faithful, preserves length, and is essentially surjective on simple objects since the restriction to semisimple objects is $\Sf{SM}^{\Rm{ps}}$ which is an equivalence. \\
  \tab Consider $0 \to \bar{E} \xrightarrow{i} \bar{F}$. Suppose for the sake of contradiction that $\ker \Sf{SM}(i) \neq 0$. Since the categories are finite-length, there must be a nonzero simple subobject $0 \to A^{\Dol} \xrightarrow{u^{\Dol}} \ker \Sf{SM}(i)$. By essential surjectivity on simple objects, $A^{\Dol} = \Sf{SM}(A^{\DR})$, and fullness then lifts $A^{\Dol} \xrightarrow{u^{\Dol}} \Sf{SM}(\bar{E})$ to a nonzero morphism $A^{\DR} \xrightarrow{u^{\DR}} \bar{E}$. Next, 
  \[\Sf{SM}(i \circ u^{\DR}) = u^{\Dol} \circ \Sf{SM}(i) = 0\] 
  so by faithfulness, $i \circ u^{\DR} = 0$. But $\ker (i) = 0$ so $A^{\DR} = 0$, which is a contradiction. That $\Sf{SM}$ preserves $\bar{E} \to \bar{F} \to 0$ follows from a dual argument. Thus $\Sf{SM}$ is exact. \\
  \tab For $\bar{E} \in \Sf{Conn}_{\Rm{nilp}}(\bar{X}, D)$, $\Sf{SM}(\Rm{soc}(\bar{E}))$ is semisimple, hence lies in $\Rm{soc}(\Sf{SM}(\bar{E}))$. Conversely every simple subobject of $\Sf{SM}(\bar{E})$ lifts to a simple subobject of $\bar{E}$, hence $\Sf{SM}(\Rm{soc}(\bar{E})) = \Rm{soc}(\Sf{SM}(\bar{E}))$. By exactness, we can induct to get that $\Sf{SM}$ preserves the whole socle filtration. \\
  \tab For the last fact that $F_{\bar{E}} = S_\bullet$, consider the inclusion $f: S_i\bar{E} \hookrightarrow \bar{E}$. This induces $\Rm{Rees}_Sf: \Rm{Rees}_S(S_i\bar{E}) \hookrightarrow \Rm{Rees}_S(\bar{E})$. By the proof of \cref{lem:morphism-formal-transport-and-equiv-algebraization}, we get a morphism over $\AAA^1_t$ by algebraizing the Rees section $\Rm{Rees}_S f$, 
  \[s_f^{\Dol}: \Rm{Rees}_{F_{\bar{E}}}(\Sf{SM}(S_i\bar{E})) \to \Rm{Rees}_{F_{\bar{E}}}(\Sf{SM}(\bar{E}))\]
  which is $\GG_m$-equivariant and where $s_f^{\Dol}(1) = \Sf{SM}(f)$. Hence $s_f^{\Dol}(1)$ is injective by exactness. At $t = 0$, $s_f^{\Dol}(0) = \Sf{SM}^{\Rm{ps}}(f^{\Rm{ss}})$, and the socle filtration of $S_i\bar{E}$ is a truncation of the socle filtration of $\bar{E}$, so $s_f^{\Dol}(0)$ is also injective. In fact we have a diagram 
  \[\begin{tikzcd}
	{\Rm{gr}^F \Sf{SM}(S_i\bar{E})} && {\Rm{gr}^F \Sf{SM}(\bar{E})} \\
	\\
	{\Sf{SM}^{\Rm{ps}}(\Rm{gr}^S(S_i\bar{E}))} && {\Sf{SM}^{\Rm{ps}}(\Rm{gr}^S\bar{E})}
	\arrow["{\Rm{gr}^F \Sf{SM}(f) = s_f^{\Dol}(0)}", from=1-1, to=1-3]
	\arrow["\simeq"', from=1-1, to=3-1]
	\arrow["\simeq", from=1-3, to=3-3]
	\arrow["{\Sf{SM}^{\Rm{ps}}(\Rm{gr}^S f)}", from=3-1, to=3-3]
  \end{tikzcd}\]
  which implies that $\im(\Rm{gr}^F \Sf{SM}(f)) = \bigoplus_{j \leq i} \Rm{gr}^F_j \Sf{SM}(\bar{E})$. $\GG_m$-equivariance then implies that $s_f^{\Dol}$ is constant rank and injective over $\AAA^1_t$. Let $A \coloneq \im(\Sf{SM}(f))$ then $A \subseteq F_i \Sf{SM}(\bar{E})$ since $\im(\Rm{gr}^F \Sf{SM}(f)) = \bigoplus_{j \leq i} \Rm{gr}^F_j \Sf{SM}(\bar{E})$. Now, $A$ and $F_i\bar{E}$ have the same $\Rm{gr}^F$, and same length hence they have the same ranks. Thus $A = F_i\bar{E}$. On the other hand, $\Sf{SM}(f)$ the inclusion of $\Sf{SM}(S_i \bar{E})$ into $\Sf{SM}(\bar{E})$, so $A = \Sf{SM}(S_i \bar{E})$. Finally $F_i \bar{E} = \Sf{SM}(S_i \bar{E}) = S_i \Sf{SM}(\bar{E})$, as desired.  
\end{proof}

We can define an analogous functor, starting from $\Sf{Higgs}_{\nilp}^{0, \Rm{ss}}(\bar{X}, D)$ instead, 
\[\Sf{Inv}: \Sf{Higgs}_{\nilp}^{0, \Rm{ss}}(\bar{X}, D) \to \Sf{Conn}_{\Rm{nilp}}(\bar{X}, D)\]
and the fact that the algebraized equivariant arcs used for $\Sf{SM}$ are associated to the socle filtration in $\Sf{Higgs}_{\nilp}^{0, \Rm{ss}}(\bar{X}, D)$ means that we can invert every construction of $\Sf{SM}$ to get $\Sf{Inv}$, i.e., 
\begin{prop}
  There are natural isomorphisms $\Sf{SM} \circ \Sf{Inv} \simeq \Rm{id}$ and $\Sf{Inv} \circ \Sf{SM} \simeq \Rm{id}$. 
\end{prop}

\begin{theorem}
  $\Sf{SM}$ is a $\CC$-linear exact equivalence of categories. 
\end{theorem}
\begin{proof}
  By the previous proposition, $\Sf{SM}$ is essentially surjective. It is also fully faithful and exact by \cref{lem:fully-faithfulness}, \cref{lem:exactness-and-socle-preservation}. 
\end{proof}

\printbibliography

@book{deligne2006equations,
  title={{\'E}quations diff{\'e}rentielles {\`a} points singuliers r{\'e}guliers},
  author={Deligne, Pierre},
  year={2006},
  publisher={Springer}
}

@misc{bakker2024linear,
      title={The linear Shafarevich conjecture for quasiprojective varieties and algebraicity of Shafarevich morphisms}, 
      author={Benjamin Bakker and Yohan Brunebarbe and Jacob Tsimerman},
      year={2024},
      eprint={2408.16441},
      archivePrefix={arXiv},
      primaryClass={math.AG}
}

@misc{simpson1997mixed,
      title={Mixed twistor structures}, 
      author={Carlos Simpson},
      year={1997},
      eprint={alg-geom/9705006},
      archivePrefix={arXiv},
      primaryClass={alg-geom}
}

@article{corlette1988flat,
  title={Flat $ G $-bundles with canonical metrics},
  author={Corlette, Kevin},
  journal={Journal of differential geometry},
  volume={28},
  number={3},
  pages={361--382},
  year={1988},
  publisher={Lehigh University}
}

@article{simpson1990harmonic,
  title={Harmonic bundles on noncompact curves},
  author={Simpson, Carlos T},
  journal={Journal of the American Mathematical Society},
  volume={3},
  number={3},
  pages={713--770},
  year={1990}
}

@article{simpson1992higgs,
  title={Higgs bundles and local systems},
  author={Simpson, Carlos T},
  journal={Publications Math{\'e}matiques de l'IH{\'E}S},
  volume={75},
  pages={5--95},
  year={1992}
}

@book{mochizuki2006kobayashi,
     author = {Mochizuki, Takuro},
     title = {Kobayashi-Hitchin correspondence for tame harmonic bundles and an application},
     series = {Ast\'erisque},
     year = {2006},
     publisher = {Soci\'et\'e math\'ematique de France},
     number = {309},
     mrnumber = {2310103},
     zbl = {1119.14001},
     language = {en}
}

@article{mochizuki2009kobayashi,
  title={Kobayashi--Hitchin correspondence for tame harmonic bundles II},
  author={Mochizuki, Takuro},
  journal={Geometry \& Topology},
  volume={13},
  number={1},
  pages={359--455},
  year={2009},
  publisher={Mathematical Sciences Publishers}
}

@article{sabbah2005polarizable,
  title={Polarizable twistor $\sh{D}$-modules},
  author={Sabbah, Claude},
  journal={Ast{\'e}risque},
  volume={300},
  year={2005}
}

@book{mochizuki2015mixed,
  title={Mixed Twistor-Modules},
  author={Mochizuki, Takuro},
  year={2015},
  publisher={Springer}
}

@misc{mochizuki2003asymptotic,
      title={Asymptotic behaviour of tame harmonic bundles and an application to pure twistor $D$-modules}, 
      author={Takuro Mochizuki},
      year={2004},
      eprint={math/0312230},
      archivePrefix={arXiv},
      primaryClass={math.DG}
}

@article{simpson1994moduli,
  title={Moduli of representations of the fundamental group of a smooth projective variety I},
  author={Simpson, Carlos T},
  journal={Publications Math{\'e}matiques de l'IH{\'E}S},
  volume={79},
  pages={47--129},
  year={1994}
}

@article{nitsure1993moduli,
  title={Moduli of semistable logarithmic connections},
  author={Nitsure, Nitin},
  journal={Journal of the American Mathematical Society},
  volume={6},
  number={3},
  pages={597--609},
  year={1993}
}

@article{simpson2022twistor,
  title={The twistor geometry of parabolic structures in rank two},
  author={Simpson, Carlos},
  journal={Proceedings-Mathematical Sciences},
  volume={132},
  number={2},
  pages={54},
  year={2022},
  publisher={Springer}
}

@misc{nitsure2005constructionhilbertquotschemes,
      title={Construction of Hilbert and Quot Schemes}, 
      author={Nitin Nitsure},
      year={2005},
      eprint={math/0504590},
      archivePrefix={arXiv},
      primaryClass={math.AG} 
}

@misc{stacks-project,
  author       = {The {Stacks project authors}},
  title        = {The Stacks project},
  howpublished = {\url{https://stacks.math.columbia.edu}},
  year         = {2026},
}

@book{etingof2015tensor,
  title={Tensor categories},
  author={Etingof, Pavel and Gelaki, Shlomo and Nikshych, Dmitri and Ostrik, Victor},
  volume={205},
  year={2015},
  publisher={American Mathematical Society Providence, RI}
}

@book{huybrechts2010geometry,
  title={The geometry of moduli spaces of sheaves},
  author={Huybrechts, Daniel and Lehn, Manfred},
  volume={2},
  year={2010},
  publisher={Cambridge University Press}
}

@article{schlessinger1968functors,
  title={Functors of Artin rings},
  author={Schlessinger, Michael},
  journal={Transactions of the American Mathematical Society},
  volume={130},
  number={2},
  pages={208--222},
  year={1968}
}

@book{bruns1998cohen,
  title={Cohen-macaulay rings},
  author={Bruns, Winfried and Herzog, H J{\"u}rgen},
  number={39},
  year={1998},
  publisher={Cambridge university press}
}

@inproceedings{cao2019albanese,
  title={Albanese maps of projective manifolds with nef anticanonical bundles},
  author={Cao, Junyan},
  booktitle={Annales Scientifiques de l'{\'E}cole Normale Sup{\'e}rieure},
  volume={52},
  number={5},
  pages={1137--1154},
  year={2019}
}
\end{document}